\documentclass[a4paper, 12pt, oneside, reqno, notitlepage]{amsart}
\usepackage{amsmath,amssymb,amsthm,graphicx,mathrsfs,bbm,url}
\usepackage{amsthm}
\usepackage{wrapfig}
\usepackage{enumitem}
\usepackage{mathtools}
\usepackage[usenames,dvipsnames]{xcolor}
\usepackage[colorlinks=true,linkcolor=Red,citecolor=Green]{hyperref}
\usepackage[super]{nth}
\usepackage[open, openlevel=2, depth=3, atend]{bookmark}
\hypersetup{pdfstartview=XYZ}
\usepackage[font=footnotesize]{caption}
\usepackage{subcaption}
\usepackage{caption}
\usepackage{tikz}
\usepackage{setspace}
\usetikzlibrary{decorations.pathreplacing}
\usetikzlibrary{arrows}
\usetikzlibrary{shapes.misc}
\usetikzlibrary{shapes.symbols}
\usetikzlibrary{patterns}
\usepackage[normalem]{ulem}

\usepackage{graphicx}
\usepackage{cite}
\usepackage[colorinlistoftodos]{todonotes}
\usepackage{bbm}
\usepackage[top=3.35cm, bottom=3.35cm, left=2.5cm, right=2.5cm, headsep=0.2in]{geometry}

\usepackage{epstopdf}
 
\usepackage{hyperref}

\newcommand{\C}{\mathbb{C}}
\newcommand{\R}{\mathbb{R}}

\newcommand{\Z}{\mathbb{Z}}

\newcommand{\G}{\mathbf{G}}
\newcommand{\M}{\mathcal{M}}

\newcommand{\X}{\mathbf{X}}

\newcommand{\Ss}{\mathbb{S}}
\newcommand{\eps}{\varepsilon}

\newcommand{\mc}{\mathcal}

\newcommand{\End}{\mathrm{End}}

\newcommand{\B}{\mathbf{B}}

\newcommand{\W}{\mathbf{W}}
\newcommand{\dd}{d}
\newcommand{\Fr}{\mathrm{Fr}}

\newcommand{\Hol}{\mathrm{Hol}}

\DeclareMathOperator{\Tr}{Tr}

\DeclareMathOperator{\im}{Im}
\DeclareMathOperator{\re}{Re}
\DeclareMathOperator{\rk}{rank}

\DeclareMathOperator{\supp}{supp}

\DeclareMathOperator{\id}{Id}
\DeclareMathOperator{\E}{\mathcal{E}}
\DeclareMathOperator{\Hom}{Hom}
\DeclareMathOperator{\Op}{Op}

\DeclareMathOperator{\ran}{ran}
\DeclareMathOperator{\spann}{span}

\DeclareMathOperator{\e}{\mathbf{e}}

\DeclareMathOperator{\loc}{loc}

\theoremstyle{plain}
\newtheorem{theorem}{Theorem}[section]
\newtheorem{definition}[theorem]{Definition}
\newtheorem{lemma}[theorem]{Lemma}
\newtheorem{remark}[theorem]{Remark}
\newtheorem{proposition}[theorem]{Proposition}

\newtheorem{corollary}[theorem]{Corollary}

\numberwithin{equation}{section}

\begin{document}

\begin{abstract}
On a Riemannian manifold $(M, g)$ with Anosov geodesic flow, the problem of recovering a connection from the knowledge of traces of its holonomies along primitive closed geodesics is known as the \emph{holonomy inverse problem}. In this paper, we prove H\"older type stability estimates for this inverse problem: 
\begin{itemize}
	\item \emph{locally}, near generic connections; 
	\item \emph{globally}, for line bundles, and for vector bundles satisfying a certain low-rank assumption over negatively curved base $(M, g)$.
\end{itemize}
The proofs are based on a combination of microlocal analysis along with a new non-Abelian approximate Liv\v sic Theorem in hyperbolic dynamics.
\end{abstract}

\title{Stability estimates for the holonomy inverse problem}
\author[M. Ceki\'{c}]{Mihajlo Ceki\'{c}}
\date{\today}
\address{Institut f\"ur Mathematik, Universit\"at Z\"urich, Winterthurerstrasse 190, CH-8057 Z\"urich, Switzerland}
\email{mihajlo.cekic@math.uzh.ch}

\author[T. Lefeuvre]{Thibault Lefeuvre}
\address{Université de Paris and Sorbonne Université, CNRS, IMJ-PRG, F-75006 Paris, France.}
\email{tlefeuvre@imj-prg.fr}

\maketitle

\section{Introduction}

The invertibility of the geodesic $X$-ray transform is the problem of recovering a function from its integrals along geodesic lines, which has played a significant role in both pure and applied mathematics (see \cite{Kuchment-14}). More generally, one can consider the effect of an \emph{attenuation} (modelling some medium parameters) or more geometrically, we may \emph{twist} the $X$-ray transform with a \emph{connection} (or an endomorphism-valued potential, also known as \emph{Higgs field}). On Riemannian manifolds with boundary, such a problem has been extensively studied (see \cite{Paternain-13}); in the Euclidean case, the twisted $X$-ray transform is at the heart of the medical imaging method of SPECT and has been researched in depth (see \cite{Finch-03}). The related \emph{non-linear} problem is to determine the connection (or the Higgs field) from parallel transports along geodesics connecting boundary points. On closed manifolds, one integrates along \emph{closed} geodesics, and the correct non-linear data turns out to be \emph{traces} of parallel transports along closed geodesics. The determination of the connection from this data is called the \emph{holonomy inverse problem}, which is deeply related to inverse spectral theory \cite{Cekic-Lefeuvre-22b}, and has great similarity to quantum field theory, where this data is encoded in the \emph{Wilson loop operator} (see \cite{Wilson-74}).

\subsection{Geodesic Wilson loop operator}\label{ssec:geodesic-wilson}
 
Let $(M, g)$ be a smooth closed Riemannian manifold and denote by $SM = \{(x, v) \in TM \mid g_x(v, v) = 1\}$ its unit tangent bundle. Let $\pi: SM \to M$, $(x, v) \mapsto x$ be the footpoint projection. The geodesic flow $(\varphi_t)_{t \in \mathbb{R}}$ on $SM$ generated by the geodesic vector field $X$ is defined by $\varphi_t(x, v) = (\gamma_{x, v}(t), \dot{\gamma}_{x, v}(t))$, where $\gamma_{x, v}$ is the unit speed geodesic determined by $(x, v)$. We say that $(M, g)$ is \emph{Anosov} if its geodesic flow is Anosov (see Section \ref{sec:preliminaries} below for definition). It is known that if $(M, g)$ has negative sectional curvature, then it is Anosov (see \cite{Anosov-67} and \cite{Knieper-02} for a modern introduction).

Let $(\E, h) \to M$ be a smooth vector bundle equipped with a Hermitian metric in each fibre. We say that a connection $\nabla^{\E}$ on $\E$ is \emph{unitary} if it is compatible with $h$ (equivalently, parallel transport along any curve is unitary with respect to $h$), and denote the (affine) space of unitary connections by $\mc{A}_{\E}$. Given a unitary isomorphism $p: (\E, h) \to (\E, h)$ of vector bundles, the pullback action on $\mc{A}_{\E}$ is defined by $p^*\nabla^{\E}(\bullet) := p^{-1} \nabla^{\E}(p \bullet)$. \emph{Moduli space} of unitary connections is the space of connections modulo the pullback action, and for connections of H\"older-Zygmund regularity $N > 0$ it is denoted by $\mathbb{A}_{\E}^N$; in the smooth case we write $\mathbb{A}_{\E}^\infty := \cap_{N > 0} \mathbb{A}_{\E}^N$. There is a natural quotient metric $d_{C_*^N}(\bullet, \bullet)$ on $\mathbb{A}_{\E}^N$ (see \S \ref{sssection:distance-connections} below). 

Denote by $\mc{G}^\sharp$ the set of primitive closed geodesics on $(M, g)$; on Anosov manifolds, $\mc{G}^\sharp$ is in bijection with the set of primitive free homotopy classes of loops. Given a closed loop $\gamma$ and $x \in \gamma$, denote by $\Hol_{\nabla^{\E}}(x, \gamma) \in \mathrm{U}(\E(x))$ the parallel transport (or \emph{holonomy}) of $\nabla^{\E}$ along $\gamma$ at $x$. Observe that $\Hol_{\nabla^{\E}}(x, \gamma)$ depends up to conjugation on both changing the connection in the moduli equivalence class $[\nabla^{\E}] \in \mathbb{A}_{\E}$, as well as changing the basepoint $x \in \gamma$, hence its trace $\Tr(\Hol_{[\nabla^{\E}]}(\gamma))$ is well-defined. The \emph{(geodesic) Wilson loop operator} is given by
\[\mathbf{W}: \mathbb{A}_{\E} \to \ell^\infty(\mc{G}^\sharp), \quad \mathfrak{a} \mapsto (\mathbf{W}(\mathfrak{a}, \gamma))_{\gamma \in \mc{G}^\sharp}, \quad \mathbf{W}(\mathfrak{a}, \gamma) := \Tr(\Hol_{\mathfrak{a}}(\gamma)).\]
The \emph{holonomy inverse problem} asks to what extent can we determine $\mathfrak{a} \in \mathbb{A}_{\E}$ from $\mathbf{W}(\mathfrak{a})$.

In order to state the main results, let us introduce the following two conditions for $\mathfrak{a} \in \mathbb{A}_{\E}$. Observe first that there is a natural unitary connection $\nabla^{\End(\E)}$ on the endomorphism bundle $\End(\E)$ (see \S \ref{sssection:distance-connections}). Then
\begin{enumerate}[itemsep = 5pt]
	\item[{\bf(A)}] $\mathfrak{a}$ is \emph{opaque}, if there are no non-trivial subbundles $\mc{F} \subset \E$ which are preserved by parallel transport along geodesic lines (i.e. except $\mc{F} = M \times \{0\}$ and $\mc{F} = \E$). Equivalently, the kernel of $(\pi^*\nabla^{\End(\E)})_X$ acting on $C^\infty(SM, \pi^*\End(\E))$ is equal to $\mathbb{C} \id_{\E}$, where $\pi^*\nabla^{\End(\E)}$ denotes the pullback connection on the pullback bundle $\pi^*\End(\E)$ (see \cite[Section 5]{Cekic-Lefeuvre-22a} for more details).

	\item[{\bf(B)}] $\mathfrak{a}$ has \emph{solenoidally injective generalised $X$-ray transform} on $1$-forms with values in $\End(\E)$, if roughly speaking, the transform of integration along closed geodesics is invertible on the kernel of the adjoint $(\nabla^{\End(\E)})^*$. More precisely, if we set $\X^{\End} := (\pi^*\nabla^{\End(\E)})_X$, this means that for any $u \in C^\infty(SM, \End(\E))$ and $f \in C^\infty(M, T^*M \otimes \End(\E))$ we have
	\[\X^{\End} u = \pi_1^*f \implies \exists p \in C^\infty(M, \End(\E)), \quad f = \nabla^{\End(\E)}p.\]
	Here we denote by $\pi_1^*f \in C^\infty(SM, \End(\E))$ the section defined by $\pi_1^*f(x, v) := f(x)(v)$.
	\end{enumerate} 
	
	By the authors' prior work, when $n \geq 3$ both conditions {\bf(A)} and {\bf(B)} are valid on an \emph{open and dense set} in the moduli space $\mathbb{A}_{\E}^N$ for $N \gg 1$ sufficiently large (depending only on $n$); this was proved in \cite[Theorem 1.6]{Cekic-Lefeuvre-22a} and \cite[Theorem 1.8]{Cekic-Lefeuvre-21-2}, respectively. We note that {\bf{(A)}} and {\bf{(B)}} are stated for smooth connections, and the statements are adapted to the regularity $C_*^N$ in a straightforward way.

Injectivity of $\mathbf{W}$ was proved locally near connections that satisfy {\bf(A)} and {\bf(B)} in \cite{Cekic-Lefeuvre-22}. We now quantify this result:

\begin{theorem}
\label{theorem:main}
Let $(M, g)$ be a smooth Anosov $n$-manifold. There exists $N \gg 1$ sufficiently large (depending only on $n$) such that the following holds. Let $\mathfrak{a}_0 \in \mathbb{A}_{\E}^{N}$ be a point satisfying assumptions {\bf{(A)}} and {\bf{(B)}}, and let $\eta > 0$ arbitrary. There exist $\tau > 0$ (depending only on $(M, g)$ and $\eta$) and $\delta, C > 0$ such that for all $\mathfrak{a}_1, \mathfrak{a}_2 \in \mathbb{A}_{\E}^N$ with $d_{C_*^N}(\mathfrak{a}_i, \mathfrak{a}_0) < \delta$ for $i = 1, 2$, we have:
\begin{equation}\label{eq:main-estimate}
	d_{C^{N - \eta}_*}(\mathfrak{a}_1, \mathfrak{a}_2) \leq C \left( \sup_{\gamma \in \mc{G}^\sharp} \ell(\gamma)^{-1} \left|\W(\mathfrak{a}_1, \gamma) - \W(\mathfrak{a}_2, \gamma)\right| \right)^\tau.
\end{equation}
\end{theorem}

Note that if we write $\varepsilon$ (assumed positive) for the supremum in \eqref{eq:main-estimate}, then it is achieved for closed geodesics of length $\ell(\gamma) < 2\varepsilon^{-1} \rk(\E)$ (since the parallel transports are unitary), i.e. it always suffices to consider only \emph{finitely many geodesics}. In fact, we will see from the proof that, for any $\delta > 0$, it only suffices to consider the supremum over closed geodesics of length $\ell(\gamma) = \mc{O}_\delta(\eps^{-1/2 - \delta})$.

In the case of line bundles, i.e. $\rk(\E) = 1$, injectivity of $\mathbf{W}$ goes back to Paternain \cite{Paternain-09}; for sums of line bundles it was proved by the authors in \cite[Theorem 5.8]{Cekic-Lefeuvre-22}. The following quantifies the former result by providing a \emph{global} stability result:

\begin{theorem}
\label{theorem:stability-line-bundles-new}
	Let $(M, g)$ be an Anosov $n$-manifold and let $\pi_{\mc{L}}: (\mc{L}, h) \to M$ be a Hermitian line bundle. There exists $N \gg 1$ sufficiently large (depending on $n$) such that the following holds. Let $\eta > 0$ be arbitrary. Then, there exists $\tau > 0$ (depending only on $(M, g)$ and $\eta$) such that for all $K > 0$, there exists a constant $C > 0$, such that for all unitary connections $\nabla_1^{\mc{L}}$ and $\nabla_2^{\mc{L}}$ with $\|\nabla_1^{\mc{L}}-\nabla_2^{\mc{L}}\|_{C^N_*} \leq K$,
\[
	d_{C_*^{N - \eta}}([\nabla_1^{\mc{L}}], [\nabla_2^{\mc{L}}]) \leq C \left( \sup_{\gamma \in \mc{G}^\sharp}   \ell(\gamma)^{-1}\left|\W([\nabla_1^{\mc{L}}], \gamma) - \W([\nabla_2^{\mc{L}}], \gamma)\right| \right)^\tau.
\] 
\end{theorem}

Global injectivity of $\mathbf{W}$ under a certain low-rank assumption and assuming $(M, g)$ is negatively curved was proved in \cite{Cekic-Lefeuvre-22b}. More precisely, for each $m \in \mathbb{Z}_{\geq 1}$ we define $q(m)$ to be the least integer $r \in \mathbb{Z}_{\geq 1}$ such that there exists a \emph{non-constant} polynomial map between unit spheres $\mathbb{S}^m \to \mathbb{S}^r$. Here by polynomial map we mean the restriction of a polynomial map $\mathbb{R}^{m + 1} \to \mathbb{R}^{r + 1}$ to $\mathbb{S}^m$. It is known by results of Wood \cite{Wood-68} that 
\[m/2 < q(m) \leq m, \quad q(2^k) = 2^k,\quad k \in \mathbb{Z}_{\geq 0}.\]
We introduce $q_{\mathbb{C}}(m) := \frac{1}{2}q(m)$; note that $q(m)$ is even for $m \geq 2$. We refer to \cite[Section 2]{Cekic-Lefeuvre-22b} for more details on $q(m)$. Using this global uniqueness result and local stability proved in Theorem \ref{theorem:main}, we get:

\begin{corollary}
\label{corollary}
	Let $(M, g)$ be a closed smooth $n$-manifold with negative sectional curvature and let $(\E, h) \to M$ be a Hermitian vector bundle with $\rk(\E) \leq q_{\mathbb{C}}(n)$, equipped with an arbitrary smooth unitary connection $\mathfrak{a}_0$. There exists $N \gg 1$ sufficiently large (depending on $n$) such that the following holds. Let $\eta \in (0, 1)$ and $K > 0$ arbitrary, and let $\mc{U} \subset \mathbb{A}_{\E}^{N - \eta}$ be an open neighbourhood of connections which do not satisfy either {\bf{(A)}} or {\bf{(B)}}. Then, there exist $\tau > 0$ (depending only on $(M, g)$ and $\eta$) and $C > 0$, such that for all $\mathfrak{a}_1, \mathfrak{a}_2 \in \mathbb{A}_{\E}^N \setminus \mc{U}$ with $d_{C^N_*}(\mathfrak{a}_i, \mathfrak{a}_0) < K$, we have: 
	\[
		d_{C^{N - \eta}_*}(\mathfrak{a}_1, \mathfrak{a}_2) \leq C \left( \sup_{\gamma \in \mc{G}^\sharp} \ell(\gamma)^{-1} \left|\W(\mathfrak{a}_1, \gamma) - \W(\mathfrak{a}_2, \gamma)\right| \right)^\tau.
	\]
\end{corollary}

Let us discuss the conditions {\bf(A)} and {\bf(B)} when $(M, g)$ has negative curvature, under the low-rank assumption $\rk(\E) \leq q_{\mathbb{C}}(n)$. Non-opacity of $[\nabla^{\E}]$ is equivalent to \emph{reducibility} in the sense that we can write $(\E, \nabla^{\E}) = (\mc{F}, \nabla^{\mc{F}}) \oplus (\mc{F}^\perp, \nabla^{\mc{F}^\perp})$ as an orthogonal sum, where $\mc{F} \subset \E$ is non-trivial. This follows immediately from the proof of \cite[Theorem 4.5]{Cekic-Lefeuvre-22b}, as well as the claim about Grassmanians in \cite[Lemma 2.3]{Cekic-Lefeuvre-22b}, and \cite[Lemma 5.3]{Cekic-Lefeuvre-22a}. In particular, if the bundle $\E$ is topologically irreducible (i.e. it has no non-trivial subbundles), then \emph{all} connections in $\mathbb{A}_{\E}$ satisfy condition {\bf(A)}. The condition {\bf(B)} is much less understood, but for instance it is satisfied if the curvature of $\nabla^{\E}$ is small enough (see \cite[Theorem 1.6]{Guillarmou-Paternain-Salo-Uhlmann-16}). However, since the holonomy inverse problem (the corresponding non-linear problem) is as mentioned known to be injective in this situation, we expect that {\bf(B)} is always satisfied under the low-rank assumption. (In fact, in the analogous case of Riemannian metrics, it is known that the injectivity of the non-linear problem (boundary rigidity) implies the injectivity of the linear problem ($X$-ray transform) on two-dimensional manifolds, see \cite{Sharafutdinov-07}.)

As far as global injectivity of $\mathbf{W}$ is concerned, it is known to fail on some vector bundles of high enough rank, that is, there exist distinct $\mathfrak{a}_1, \mathfrak{a}_2 \in \mathbb{A}_{\E}$ such that $\mathbf{W}(\mathfrak{a}_1)=\mathbf{W}(\mathfrak{a}_2)$, see \cite[Proposition 4.2]{Cekic-Lefeuvre-22b}. Nevertheless, these connections are ``far apart'' in the moduli space, and it is very likely that injectivity of $\mathbf{W}$ still holds \emph{locally} around $\mathfrak{a}_1$ or $\mathfrak{a}_2$. However, a local stability estimate is also expected in this case, and more generally, as in Corollary \ref{corollary} we get global stability estimates once we restrict to a suitable subset of $\mathbb{A}_{\E}$ where injectivity of $\mathbf{W}$ and conditions {\bf(A)} and {\bf(B)} hold.

We also point out that a similar problem to the recovery of a connection from the trace of its holonomies along closed geodesics, is to recover simultaneously a connection and a Higgs field (i.e. endomorphism-valued potential) from similar data, see \cite{Paternain-Salo-Uhlmann-12,Guillarmou-Paternain-Salo-Uhlmann-16} for the related case of manifolds with boundary. It can be checked that the injectivity of the Wilson loop operator $\mathbf{W}$ for connections and Higgs fields in the low-rank setting (and negative curvature) immediately follows from \cite[Theorem 4.5]{Cekic-Lefeuvre-22b}. We expect identical stability estimates as Theorem \ref{theorem:main} and Corollary \ref{corollary} to hold.

\subsection{Approximate non-Abelian Liv\v{s}ic Theorem}

One of the main new ingredients in the above theorems is an approximate version of the Liv\v{s}ic theorem for \emph{unitary cocycles}, which may also be viewed as a stability result. We will take the infinitesimal point of view on cocycles, that is, they will be generated by differential operators $\X$; the more common parallel transport version of a cocycle is obtained by integrating (solving) the ordinary differential equation $\X u = 0$ along flow lines. 

Let $(\varphi_t)_{t \in \mathbb{R}}$ be an Anosov flow on a smooth closed manifold $\M$ generated by a smooth vector field $X$. We say that $(\varphi_t)_{t \in \mathbb{R}}$ is \emph{transitive} if it has a dense orbit in $\M$. Denote by $\mc{B}$ the set of all triples $(\E, h, \X)$, where $(\E, h) \to \M$ is a smooth Hermitian vector bundle equipped with a linear differential operator $\X$ of order $1$ satisfying the following Leibniz rule
\begin{equation*}
	\X(fu) = (Xf) u + f \X u, \quad  f \in C^\infty(\M),\,\, u \in C^\infty(\M,\E),
\end{equation*}
and such that the parallel transport induced by $\X$ is unitary with respect to $h$, or equivalently, we have:
\begin{equation*}
	X h(u_1, u_2) = h(\X u_1, u_2) + h(u_1, \X u_2), \quad u_1, u_2 \in C^\infty(\M, \E). 
\end{equation*}
We call such a triple a \emph{linear unitary extension} of $(\varphi_t)_{t \in \mathbb{R}}$. When the emphasis is on regularity, we will write $\mc{B}^1$ for the space of cocycles where $\X$ has coefficients in $C^1$. Given $(\E_j, h_j, \X_j) \in \mc{B}$ for $j = 1, 2$, it is possible to introduce a first order differential operator $\X^{\Hom}$ on the homomorphism vector bundle $\Hom(\E_1, \E_2)$ by the Leibniz formula
\begin{equation*}
	\X_2(H u) =\X^{\Hom}(H) u + H \X_1 u, \quad \forall H \in C^\infty(\M, \Hom(\E_1, \E_2)),\,\, u \in C^\infty(\M, \E_1).
\end{equation*}
Clearly $\X^{\Hom}$ depends on $\X_1$ and $\X_2$; we suppress these from the notation for simplicity and note that this dependence will be clear from context.

We now introduce a notion of uniform boundedness for linear unitary cocycles that is required to state the theorem. Let $K_0 > 0$ and $R_0 \in \mathbb{Z}_{\geq 1}$. By standard results from algebraic topology, there is an $L > 0$, such that any Hermitian vector bundle over $\M$ of rank $\leq R_0$ can be unitarily (smoothly) embedded by some $p: (\E, h) \xhookrightarrow{} \M \times (\mathbb{C}^L, h_{\mathrm{std}})$, where the latter is equipped with the standard Hermitian metric $h_{\mathrm{std}}$ on $\mathbb{C}^L$. Let $\pi_{\E}$ denote the orthogonal projection onto $p(\E)$, and consider the zero order term 
\[
	Q(\X) := p_*\X - \pi_{\E} \circ (X \times \id_{p(\E)}) \in C^1(\M, \End(p(\E))),
\] 
where $p_*\X(\bullet) := p\X(p^{-1}\bullet)$ and $\pi_{\E} \circ(X \times \id_{p(\E)})$ acts by differentiating coordinate-wise by $X$ and then projecting onto $p(\E)$. We say that $(\E, h, \X)$ is \emph{uniformly bounded by $(K_0, R_0)$} if there exists an embedding $p$ as above such that
\[
	\rk(\E) \leq R_0, \quad \|\pi_{\E}\|_{C^1} \leq K_0, \quad \|\pi_{\E}Q(\X) \pi_{\E}\|_{C^1} \leq K_0.
\]
See \S \ref{ssec:linear-unitary-extensions} below for more details on this concept.

Next, we fix an arbitrary periodic orbit $\gamma_\star$ of $(\varphi_t)_{t \in \mathbb{R}}$ together with a reference point $x_\star \in \gamma_\star$, and denote by $\mc{H}$ the set of homoclinic orbits to $\gamma_\star$ (that is, orbits that accumulate in both past and future to $\gamma_\star$). Define $\mathbf{G}$, \emph{Parry's free monoid}, to be the free monoid generated by $\mc{H}$. For $(\E, h, \X) \in \mc{B}$, introduce \emph{Parry's representation} as
\begin{equation}\label{eq:parry-rep-def}
	\rho: \mathbf{G} \to \operatorname{U}(\E(x_\star)), \quad \rho(\gamma) := \mathbf{H}^s_{x_s(\gamma) \to x_\star} \mathbf{H}^c_{x_u(\gamma) \to x_s(\gamma)} \mathbf{H}^u_{x_\star \to x_u(\gamma)}, \quad \gamma \in \mc{H},
\end{equation}
and extend multiplicatively to $\mathbf{G}$. Here $x_{s/u}(\gamma) \in \gamma$ are such that they belong to stable/unstable leaves of $x_\star$, and $\mathbf{H}^{s/c/u}$ denote \emph{stable/central/unstable holonomies}, see \S \ref{sssec:parry-monoid-anosov} below for more details. 

Similarly to the previous section, write $\mc{G}^\sharp$ for the set of primitive closed orbits of $(\varphi_t)_{t \in \mathbb{R}}$ and $\ell(\gamma)$ for the period of $\gamma \in \mc{G}^\sharp$. For $b = (\E, h, \X) \in \mc{B}$, define $\mathbf{W}(b, \gamma)$ to be the trace of the parallel transport induced by $\X$ along $\gamma \in \mc{G}^\sharp$. 

We are now in shape to state the theorem:

\begin{theorem}
\label{theorem:approximate-livsic}
Assume that $\M$ is endowed with a smooth transitive Anosov flow generated by the vector field $X$. Let $b_0 = (\E_0, h_0, \X_0) \in \mc{B}^1$ such that its Parry's representation $\rho_0$ is irreducible, and let $K_0 > 0$, $R_0 \in \mathbb{Z}_{\geq 1}$. Let $b = (\E, h, \X) \in \mc{B}^1$ such that $b$ is uniformly bounded by $(K_0, R_0)$. Then, there exist $\tau > 0$, $\alpha > 0$ (depending only on $X$) and $C > 0$, $\eps_0 > 0$ (depending on $b_0$, $K_0$, and $R_0$) such that the following holds. Assume that there exists $\eps \in (0, \eps_0)$ such that
\begin{equation}
\label{equation:wilson-small}
\left| \W(b_0,\gamma) - \W(b,\gamma) \right| < \eps \ell(\gamma), \quad \forall \gamma \in \mc{G}^\sharp.
\end{equation}
Then, there exists $p \in C^\infty(\M,\mathrm{U}(\E_0, \E))$ such that $\X^{\Hom} p \in C^\alpha(\M,\mathrm{U}(\E_0,\E))$ and
\[
	\|\X^{\Hom} p\|_{C^\alpha} \leq C \eps^\tau, \quad \|p\|_{C^\alpha} \leq C.
\]
\end{theorem}

Similarly to the remark after Theorem \ref{theorem:main}, we note that it only suffices to assume the condition \eqref{equation:wilson-small} for finitely closed orbits; more precisely, for any $\delta > 0$, the proof shows that the conidition \eqref{equation:wilson-small} is required only for $\ell(\gamma) = \mc{O}_{\delta}(\eps^{-1/2 - \delta})$. Observe that if we have the equality $\mathbf{W}(b_0, \gamma) = \mathbf{W}(b, \gamma)$ for all $\gamma \in \mc{G}^\sharp$, by applying the theorem for all $\varepsilon \in (0, \varepsilon_0)$ and using compactness, we may obtain $p \in C^{\frac{\alpha}{2}}(\M, \operatorname{U}(\E_0, \E))$ such that $\X^{\Hom}p = 0$. If $\X_0$ and $\X$ are of regularity $C^\infty$, then regularity results (see \cite{Bonthonneau-Lefeuvre-20}) show that $p$ is smooth and we retrieve the exact non-Abelian Liv\v{s}ic Theorem in \cite[Theorem 3.2]{Cekic-Lefeuvre-22}.

In the \emph{Abelian case}, a version of Theorem \ref{theorem:approximate-livsic} was first proved by Gou\"ezel and the second author in \cite{Gouezel-Lefeuvre-19}: it asserts that a function $f \in C^1(\M)$ whose (normalized) integral along every closed geodesic is $< \eps$ is a \emph{quasi-coboundary} in the sense that there exists $u, h \in C^\infty(\M)$ such that $f = Xu + h$ and $\|h\|_{C^\alpha} \leq C \eps^{\tau}\|f\|_{C^1}^{1-\tau}$. We note that it is possible to retrieve the Abelian approximate Liv\v sic Theorem of \cite{Gouezel-Lefeuvre-19} using Theorem \ref{theorem:approximate-livsic} (at least when the flow is \emph{homologically full}).

\subsection{Previous results} In the closed case, there are not many results on the holonomy inverse problem. Apart from the above mentioned references \cite{Cekic-Lefeuvre-22, Cekic-Lefeuvre-22b}, there are results of Paternain \cite{Paternain-09, Paternain-11, Paternain-12}, who classified \emph{transparent connections} on negatively curved surfaces, that is, connections whose parallel transport along every closed geodesic is the identity map; see also \cite[Section 9]{Guillarmou-Paternain-Salo-Uhlmann-16}.

In the boundary case, one considers the problem of recovering the connection (and a potential) from parallel transports between geodesics connecting boundary points. Paternain-Salo-Uhlmann-Zhou \cite{Paternain-Salo-Uhlmann-Zhou-19} show injectivity under a certain convex foliation condition and prove the injectivity of the related non-Abelian $X$-ray transform (based on the method of Vasy-Uhlmann \cite{Uhlmann-Vasy-16}); $L^2$-stability for potentials in the same setting was shown by Bohr \cite{Bohr-21}. Under the assumption that $(M, g)$ is \emph{simple}, generic local stability for both linear and non-linear problems was obtained by Zhou \cite{Zhou-17}, see also \cite{Paternain-Salo-Uhlmann-12}.

\subsection{Proof strategy} As mentioned above, the main new ingredient in the paper is the non-Abelian Liv\v{s}ic Theorem \ref{theorem:approximate-livsic}, as well an analysis of Pollicott-Ruelle resonances in low regularity. For the former result, we start by proving a \emph{stability result in representation theory}: a careful analysis shows that two unitary representations with nearly equal characters are nearly equivalent. This generalises the usual statement that two unitary representations are equivalent if and only if they have equal characters. We proceed to apply this to our setting: we show that the assumption of near equality of Wilson loop operators implies that Parry's representations acting on the fibres above $x_\star \in \gamma_\star$ have nearly equal characters, so our representation theoretic stability result applies. Then we proceed to construct a suitably dense, but also separated homoclinic orbit (to $\gamma_\star$) and using stable/central/unstable holonomies we extend this near equivalence to this good orbit and eventually in a controlled H\"older fashion to the whole space, completing the proof.

For the applications in \S \ref{ssec:geodesic-wilson}, we then need a careful perturbation theoretic argument in low regularity (we need to use the paradifferential calculus at this stage), as well as a stability estimate established in our previous work \cite{Cekic-Lefeuvre-22}.

\subsection{Perspectives} To the best of our knowledge, all known results (in the closed case) regarding injectivity of the holonomy inverse problem are concerned with \emph{unitary} connections. It would be interesting to generalize them to the non-unitary setting, and also to consider non skew-Hermitian Higgs fields (the latter is what typically appears in SPECT). This, as well as proving suitable stability estimates, is left for future investigation.

The \emph{connection} (or \emph{Bochner}) \emph{Laplacian} is the second order elliptic operator $\Delta_{\E} := (\nabla^{\E})^* \nabla^{\E}$; by the Duistermaat-Guillemin trace formula, the spectrum of this operator determines the Wilson loop operator $\mathbf{W}([\nabla^{\E}])$ (at least if all geodesics have distinct lengths). This observation and injectivity results for $\mathbf{W}$ were used in \cite{Cekic-Lefeuvre-22b} to prove that the spectrum of $\Delta_{\E}$ determines $[\nabla^{\E}]$. In view of the stability results of this paper, it would be interesting to quantify this inverse spectral result and to estimate the distance between the two spectra by the distance in $\mathbb{A}_{\E}$.

Chen, Erchenko, and Gogolev \cite{Chen-Erchenko-Gogolev-23} prove that simple manifolds with boundary $(M,g_M)$ can be isometrically embedded into closed Anosov manifolds $(N, g_N)$. Assuming {\bf (A)} and {\bf (B)} are true in this extension $(N, g_N)$ (which is true at least generically), our stability results also give certain H\"older type stability estimates for the parallel transport problem on $(M, g_M)$.

\subsection{Organisation of the paper} In Section \ref{sec:preliminaries}, we introduce the tools that are used throughout the paper: (linear) extensions of Anosov flows and Parry's representation, moduli space of connections, and the theory of Pollicott-Ruelle resonances. In Section \ref{section:low-regularity}, we establish some technical results for Pollicott-Ruelle resonances with low-regularity potentials that are needed in the proof of the main Theorem \ref{theorem:main}. Next, in Section \ref{section:representation-theory} we prove a stability estimate in representation theory. In Section \ref{sec:approximate-livsic} we prove the approximate non-Abelian Liv\v{s}ic Theorem \ref{theorem:approximate-livsic}. Finally, in Section \ref{sec:main-proof} we prove the main stability estimates, Theorems \ref{theorem:main} and \ref{theorem:stability-line-bundles-new}, as well as Corollary \ref{corollary}.
\medskip

\noindent \textbf{Acknowledgements:} M.C. has received funding from an Ambizione grant (project number 201806) from the Swiss National Science Foundation. During the course of writing this paper, he was also supported by the European Research Council (ERC) under the European Union’s Horizon 2020 research and innovation programme (grant agreement No. 725967). We thank the anonymous referee for their comments.

\section{Preliminaries}\label{sec:preliminaries}

In this section we introduce the tools needed in the main proofs: fundamental domains for homoclinic orbits of Anosov flows, connections in limited regularity, linear unitary cocycles, we study the \emph{uniform boundedness} assumption for cocycles, and we introduce Parry's representation.

\subsection{Anosov flows, homoclinic orbits, and fundamental domains}\label{sssec:parry-monoid-anosov} 
\label{ssec:fundamental-domain} 

Let $\M$ be a smooth manifold equipped with a transitive Anosov flow $(\varphi_t)_{t \in \R}$ generated by a vector field $X \in C^\infty(\M,T\M)$. More precisely, we assume there exists a $d\varphi_t$-invariant continuous splitting $T\M = \mathbb{R}X \oplus E_u \oplus E_s$ and there exist $C, \nu > 0$ such that for all $x \in \M$
\begin{align}\label{eq:anosov}
\begin{split}
	|d\varphi_t(x)(v)| \leq C e^{-\nu |t|} |v|, \qquad \begin{cases} v \in E_s(x),\, t \geq 0,\\
	v \in E_u(x),\, t \leq 0,
	\end{cases}
\end{split}
\end{align}
where $|\bullet|$ is an arbitrary smooth norm on the fibres of $T\M$. We call $E_{s/u}$ \emph{stable/unstable} bundles. A principal example of volume preserving Anosov flows are geodesic flows on Riemannian manifolds of negative curvature. Write $d(\bullet, \bullet)$ for a distance induced by a smooth fixed Riemannian metric. 

Anosov flows admit stable and unstable foliations tangent to stable and unstable bundles; for details of the following see \cite[Theorem 6.1.1]{Fisher-Hasselblatt-19}. \emph{Global} leaves at $x \in \M$ are defined as:
\begin{align*}
	W^{s}(x) &:= \{y \in \M \mid d(\varphi_{t}x, \varphi_{t} y) \to_{t \to \infty} 0\},\\
	W^{u}(x) &:= \{y \in \M \mid d(\varphi_{-t}x, \varphi_{-t} y) \to_{t \to \infty} 0\},
\end{align*}
which are injectively immersed smooth submanifolds tangent to $E^{u}(x)$ and $E^s(x)$, respectively. Moreover, for every $t_0 > 0$ and $x \in \M$ there are \emph{local} leaves $W^{s/u}_{\loc}(x)$ diffeomorphic to a disk and tangent to $E_{s/u}(x)$ at $x$, and the dependence on $x$ is continuous. They satisfy the following properties: there exist $C > 0$ and $\mu \in (0, \nu)$ such that for $t \geq 0$
\begin{align}\label{eq:speed-of-convergence}
	d(\varphi_t x, \varphi_t y) \leq C e^{-\mu t},\quad y \in W^s_{\loc}(x), \qquad d(\varphi_{-t} x, \varphi_{-t} y) \leq C e^{-\mu t},\quad y \in W^u_{\loc}(x).
\end{align}
Moreover, for all $t \geq t_0$ it holds that
\begin{align}\label{eq:invariance-flow-t0}
	\varphi_{t}(W^s_{\loc}(x)) \subset W^s_{\loc}(\varphi_tx), \quad \varphi_{-t}(W^u_{\loc}(x)) \subset W^u_{\loc}(\varphi_{-t}x).
\end{align}
For all $\varepsilon > 0$ small enough we have $W^{s/u}_{\eps}(x) \subset W^{s/u}_{\loc}(x)$ where
\begin{align*}
	W^{s}_{\varepsilon}(x) &:= \{y \in W^s_{\loc}(x) \mid \forall t \geq 0,\, d(\varphi_t x, \varphi_ty) < \varepsilon\},\\
	W^{u}_{\varepsilon}(x) &:= \{y \in W^u_{\loc}(x) \mid \forall t \geq 0,\, d(\varphi_{-t} x, \varphi_{-t}y) < \varepsilon\}.
\end{align*}
There exists $\varepsilon_0 > 0$ small enough such that for all $\varepsilon \in (0, \varepsilon_0)$, if $d(x, y) < \varepsilon$ there is a unique point
\begin{equation}\label{eq:bowen-bracket}
	z := [x, y] \in W^{cu}_{\varepsilon}(x) \cap W^s_{\varepsilon}(y)
\end{equation}
called the \emph{Bowen bracket}, where $W^{cu}_{\eps}(x) = \cup_{t \in (-\eps, \eps)} W^u(\varphi_t x)$. Moreover, there exists a $t = t(x, y) \in (-\varepsilon, \varepsilon)$ such that $z \in W^u_{\varepsilon}(\varphi_t(x))$, and a uniform $C > 0$ such that
\[d(z, \varphi_t(x)) < C d(x, y), \quad d(z, y) < Cd(x, y).\]
For details about the Bowen bracket, see \cite[Proposition 6.2.2]{Fisher-Hasselblatt-19}. By considering \emph{adapted norms}, defined for $x \in \M$ and $v \in T_xM$ as
\begin{equation}\label{eq:adapted-norms}
	|v|_s^2 := \int_{0}^T e^{\mu t} |d\varphi_t(x)(v)|^2\, dt, \quad |v|_u^2 := \int_{0}^T e^{\mu t} |d\varphi_{-t}(x)(v)|^2\, dt,
\end{equation}
where $\mu \in (0, \nu)$ and $T = T(\mu) > 0$ is large enough, we may assume for $t \geq 0$ that
\[|d\varphi_t(x)(v)|_s \leq e^{-\mu t} |v|_s, \quad v \in E_s(x), \quad |d\varphi_{-t}(x)(v)|_u \leq e^{-\mu t} |v|_u, \quad v \in E_u(x).\]
Indeed, this is standard and the proof is similar to \cite[Proposition 5.1.5]{Fisher-Hasselblatt-19}. 

We fix an arbitrary periodic point $x_\star$ of period $T_\star$, whose orbit is denoted by $\gamma_\star$. Denote by $\mc{H}$ the set of homoclinic orbits to $x_\star$, namely, orbits which accumulate both in the past and in the future to $\gamma_\star$. More precisely, $\gamma \in \mc{H}$ if there exist points $x_\pm \in \gamma$ such that $d(\varphi_{\pm t} x_\pm, \varphi_{\pm t} x_\star) \to 0$ as $t \to \infty$.

It will be convenient to associate a notion \emph{length} to every homoclinic orbit. The flow in time $T_\star$, $\Psi := \varphi_{T_\star}$ acts on $W^{s/u}_{\loc}(x_\star)$ by \eqref{eq:invariance-flow-t0} by taking $t_0 < T_\star$. In fact, we have:

\begin{lemma}\label{lemma:fundamental-domain}
	With the notation as above, there are further $\Psi$-invariant open subsets $U_{s/u} \subset W^{s/u}_{\loc}(x_\star)$ homeomorphic to disks for which there are fundamental domains $\mathbf{D}_{s/u} \subset U_{s/u}$ homeomorphic to annuli. Every homoclinic orbit $\gamma \in \mc{H}$, except $\gamma_\star$, crosses $\mathbf{D}_{s/u}$ exactly once.
\end{lemma}
\begin{proof}
	Let us prove the claim for the stable foliation; the case of the unstable foliation follows by considering the flow $(\varphi_{-t})_{t \in \mathbb{R}}$. Note that $x_\star \in W^s_{\loc}(x_\star)$ is a hyperbolic fixed point for $\Psi$ since $d\Psi(x_\star): E_s(x_\star) \to E_s(x_\star)$ is contracting for the norm $|\bullet|_s$ by \eqref{eq:adapted-norms} and hence $x_\star$ is a hyperbolic fixed point for $\Psi$. By the Hartman-Grobman Theorem (see \cite[Theorem B.4.14]{Fisher-Hasselblatt-19}) there exist an open neighbourhood of zero $V_s = \{v \in E_s(x_\star) \mid |v|_s < \delta\} \subset E_s(x_\star)$ for some $\delta > 0$ small enough, an open set $U_s \subset W^s_{\loc}(x_\star)$ containing $x_\star$ invariant under $\Psi$ and a homeomorphism $h: U_s \to V_s$ such that $\Psi|_{U_s} = h^{-1} \circ d\Psi(x_\star) \circ h|_{U_s}$. Therefore $\Psi$ is locally topologically conjugate to the contractive linear map $d\Psi(x_\star)$ and the claim follows upon taking $\mathbf{D}_s := U_s \setminus \Psi(U_s)$.
	
	For the last claim, assume $\gamma \in \mc{H}$ intersects $\mathbf{D}_{s}$ at $x$ and $y = \varphi_{t_0}x$ for $t_0 \geq 0$, which by definition means that $d(\varphi_{kT_*}x, x_\star)\to 0$ and $d(\varphi_{kT_* + t_0}x, x_\star) \to 0$ as $k \to \infty$. This is impossible unless $t_0$ is an integer multiple of $T_\star$. But $\mathbf{D}_s$ was constructed so that each orbit of $\Psi$ in $U_s$ intersects $\mathbf{D}_s$ only once showing $t_0 = 0$ and so $x = y$. Note that each homoclinic orbit intersects $W^s_{\loc}(x_\star)$ since $W^s_{\varepsilon}(x_\star) \subset W^s_{\loc}(x_\star)$ for $\varepsilon > 0$ small enough, and so eventually also intersects the neighbourhood $U_s$ and so $\mathbf{D}_s$.
\end{proof}

Note that the choice of the fundamental domains $\mathbf{D}_{s/u}$ in Lemma \ref{lemma:fundamental-domain} is arbitrary and one can obtain another fundamental domain $\mathbf{D}'_{s/u}$ from $\mathbf{D}_{s/u}$ by simply applying the map $\Psi$ a finite number of times. Fundamental domains are connected unless the ranks of $E_{s/u}$ are equal to $1$ in which case they have two components (homeomorphic to intervals). We define the points $x_{s/u}(\gamma) \in \gamma$ such that
\[
	\left\{x_{s/u}(\gamma)\right\} := \mathbf{D}_{s/u} \cap \gamma, \quad \gamma \in \mc{H}.
\]

\begin{definition}
\label{definition:length}
The \emph{length} of $\gamma$ is defined to be the unique time $T(\gamma)$ such that $x_s(\gamma) = \varphi_{T(\gamma)}(x_u(\gamma))$. The segment of orbit $\mc{T}_\gamma := [x_u(\gamma);x_s(\gamma)]$ will be called the \emph{trunk} of $\gamma$.
\end{definition}

The following lemma shows that we can choose the fundamental domains to make the length function positive.

\begin{lemma}\label{lemma:length-positive}
	The length function is bounded from below. Therefore, by choosing the fundamental domains $\mathbf{D}_{s/u}$ suitably, we can assume that $T(\gamma)$ is positive for any $\gamma \in \mc{H}$.
\end{lemma}
\begin{proof}
	
Observe that the definition of length depends on the choice of fundamental domains $\mathbf{D}_{s/u}$, and changing the domains (under iteration by $\Psi^{\pm 1}$) will affect uniformly all the lengths by $\pm kT_\star$, $k \in \Z_{\geq 0}$. In particular, once we show the length function is bounded from below we can modify $\mathbf{D}_s$ such that the length function is positive.

By Lemma \ref{lemma:fundamental-domain}, the fundamental domains are at some positive distance from $x_\star$. Using \eqref{eq:speed-of-convergence} we know that there are $C > 0$ and $\mu > 0$ such that for $x \in \mathbf{D}_s \subset W^s_{\loc}(x_\star)$, and $k \in \mathbb{Z}_{> 0}$, we have $d(\varphi_{kT_\star}x, x_\star) \leq C e^{-\mu k T_\star}$. Therefore, for $k$ sufficiently large $\varphi_{kT_\star}(\mathbf{D}_s)$ will be very close to $x_\star$ and will miss $\mathbf{D}_u$, showing that the length function is bounded from below. This completes the proof.
\end{proof}

From now on we make a standing assumption that $\mathbf{D}_{s/u}$ are chosen such that the length function is positive (possible by Lemma \ref{lemma:length-positive}). For later use, for each $n \in \mathbb{Z}_{\geq 0}$, we introduce the following translates of $x_{s/u}(\gamma)$:
\begin{equation}\label{eq:fundamental-domain-translate}
	x_{s}(\gamma; n) := \varphi_{nT_\star}(x_s(\gamma)), \quad x_{u}(\gamma; n) := \varphi_{-nT_\star}(x_u(\gamma)).
\end{equation}
Thanks to \eqref{eq:speed-of-convergence}, and since $\mathbf{D}_{s/u}$ are fixed, we have:
\begin{equation}\label{eq:estimate-translate}
	d(x_{s/u}(\gamma; n), x_\star) = \mc{O}(e^{-nT_\star \mu}), \quad n \to \infty.
\end{equation}

\begin{figure}
\centering
\includegraphics[scale=1.2]{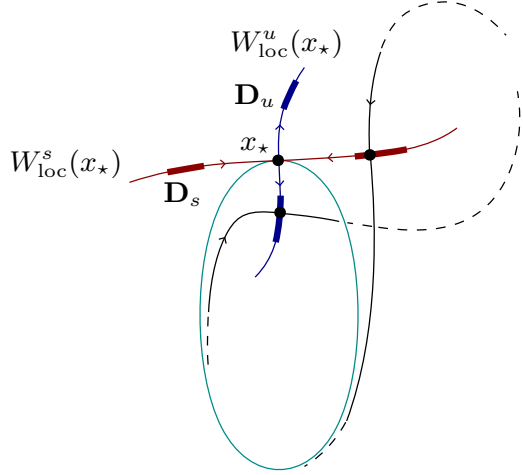}
\caption{The orbit of $x_\star$ and the fundamental domains $\mathbf{D}_{s/u}$.}
\end{figure}

\subsection{Moduli space of connections}

\label{sssection:distance-connections}

Let $\nabla^{\mc{E}}$ be a connection on the vector bundle $\mc{E}$. The connection $\nabla^{\mc{E}}$ induces a canonical connection $\nabla^{\End(\E)}$ on $\End(\mc{E})$ defined by the following Leibniz identity:
\begin{equation}
\label{equation:connection-end}
\nabla^{\mc{E}}\left(u(f)\right) = (\nabla^{\End(\E)}u)(f) + u(\nabla^{\mc{E}}f), \quad \forall u \in C^\infty(\M,\mathrm{End}(\mc{E})), \quad \forall f \in C^\infty(\M,\mc{E}).
\end{equation}
In particular, if ${\nabla'}^{\mc{E}} = \nabla^{\mc{E}} + A$ with $A \in C^\infty(\M,\mathrm{End}(\mc{E}))$ is another connection on $\mc{E}$, then:
\begin{equation}
\label{equation:diff-connection-end}
{\nabla'}^{\End(\E)}(\bullet) = \nabla^{\End(\E)}(\bullet) + [A, \bullet].
\end{equation}

Connections $\nabla_1^{\E}$ and $\nabla_2^{\E}$ are said to be \emph{gauge-equivalent}, denoted by $\nabla_1^{\E} \sim \nabla_2^{\E}$, if there exists a $p \in C^\infty(\M, \mathrm{Aut}(\E))$ such that $\nabla_1^{\E}(\bullet) = p^{-1}\nabla_2^{\mc{E}}(p\bullet) =: p^* \nabla_2^{\E}(\bullet)$ ($p^*$ stands for \emph{pullback}). Here $C^\infty(\M, \mathrm{Aut}(\E))$ denotes invertible sections of $\End(\E)$. Using the natural connection $\nabla^{\End(\E)}$, the pullback action can be rewritten as 
\begin{equation}\label{eq:gauge-equivalent-end}	
	p^*\nabla^{\E} = p^{-1} \nabla^{\End(\E)}p+ \nabla^{\mc{E}}.
\end{equation}
In the particular case where $\mc{E} = \mc{L}$ is a line bundle, this boils down to the existence of a nowhere vanishing function $p \in C^\infty(\M,\C)$ such that $p^*\nabla^{\E} = p^{-1}dp + \nabla^{\E}$.

Assume now $\mc{E}$ is endowed with a Hermitian inner product $h$ and $\nabla^{\mc{E}}$ is \emph{unitary}, that is, we have 
\[Y(h(s_1, s_2)) = h(\nabla_Y^{\E}s_1, s_2) + h(s_1, \nabla_Y^{\E} s_2), \quad \forall s_1, s_2 \in C^\infty(\M, \E), \quad \forall Y \in C^\infty(\M, T\M).\]
In this case, the set of gauge-equivalent connections is obtained by considering only sections $p \in C^\infty(M,\mathrm{U}(\E))$, that is, with values in the unitary group. Denote by $\mc{A}^{\mc{E}}_{\mathrm{unit}}$ the space of all unitary connections on $(\E, h)$ and by $\sim_{\mathrm{unit}}$ the equivalence relation obtained. Introduce the \emph{moduli space} of unitary connections on $\E$ by $\mathbb{A}_{\E}^{\mathrm{unit}} := \mc{A}_{\mc{E}}^{\mathrm{unit}}/\sim_{\mathrm{unit}}$. In order not to burden the notations, we will drop the subscript $\mathrm{unit}$ notation and \emph{from now on, we will always assume that the connections are unitary.} 

We can also work in limited regularity and consider $C^s_*$ regular connections for $s > 0$, instead of $C^\infty$; denote the space of such connections by $\mathcal{A}^s_{\E}$. Here $C^s_*$ denotes the scale of H\"older-Zygmund regularity, see \cite[Appendix A]{Taylor-91}. In this case, the gauge-equivalence is considered with $p \in C^{s+1}_*(\M, \mathrm{U}(\E))$ and the moduli space of such connections is denoted by $\mathbb{A}_{\E}^s$. For any $s \in (1, \infty]$, we introduce the following natural metric on $\mathcal{A}_{\E}^s$:
\[
\begin{split}
d(\nabla_1^{\mc{E}}, \nabla_2^{\mc{E}})& := \inf_{p \in C^\infty(\M,\mathrm{U}(\E))} \|p^{-1}\nabla_1^{\mathrm{End}(\E)}p + \nabla_1^{\mc{E}}-\nabla_2^{\mc{E}}\|_{L^\infty(\M, T^*\M \otimes \mathrm{End}(\E))}\\
&  =  \inf_{p \in C^\infty(\M,\mathrm{U}(\E))} \|p^* \nabla_1^{\E}  - \nabla_2^{\E}\|_{L^\infty(\M,T^* \M \otimes \mathrm{End}(\E))}.
\end{split}
\]

We have the following classical statement:

\begin{lemma}
\label{lemma:distance-connections}
The metric $d$ descends to a metric on $\mathbb{A}^s_\mc{E}$, for $1 < s \leq \infty$.
\end{lemma}

\begin{proof}
Since the pullback action by a section $p \in C^\infty(\M, \mathrm{U}(\E))$ is an isometry of the fibres of $\End(\E)$, we know that $d$ descends to $\mathbb{A}_{\E}$. The non-negativity of $d$ is immediate, as well as the triangle inequality and symmetry. It remains to check that $d(\nabla_1^{\mc{E}},\nabla_2^{\mc{E}}) = 0$ (where $\nabla_1^{\E}$ and $\nabla_2^{\E}$ are $C_*^s$ regular) if and only if the two connections are gauge-equivalent. 

Consider a sequence $p_j \in C^\infty(\M,\mathrm{U}(\E))$ such that $p_j^{-1}\nabla_1^{\mathrm{End}(\E)}p_j + \nabla_1^{\mc{E}}-\nabla_2^{\mc{E}} \rightarrow_{j \rightarrow \infty} 0$ in $L^\infty$, which is equivalent to $\nabla_1^{\mathrm{End}(\E)}p_j +p_j( \nabla_1^{\mc{E}}-\nabla_2^{\mc{E}}) \rightarrow_{j \rightarrow \infty} 0$. Write $s=1+\eps$ for $\eps > 0$. Since $\|p_j\|_{L^\infty} = 1$ and $L^\infty(\M, \End(\E)) \hookrightarrow C_*^{-\eps/2}(\M, \End(\E))$ is compact, we can always assume that $p_j \rightarrow p$ in $C_*^{-\eps/2}(\M, \End(\E))$. Therefore, 
\[
\nabla_1^{\mathrm{End}(\E)}p_j +p_j( \nabla_1^{\mc{E}}-\nabla_2^{\mc{E}}) \rightarrow_{j \rightarrow \infty} \nabla_1^{\mathrm{End}(\E)}p+p( \nabla_1^{\mc{E}}-\nabla_2^{\mc{E}}) = 0,
\]
in $C_*^{-\eps/2-1}(\M, T^*\M \otimes \End(\E))$ (here we use the continuity of the multiplication $C_*^{-\eps/2} \times C_*^{1 + \eps} \to C_*^{-\eps/2}$). Since $\nabla_1^{\mc{E}}-\nabla_2^{\mc{E}} \in C^s_*(\M, T^*\M \otimes \End(\E))$, the operator 
\[U \mapsto \nabla_1^{\mathrm{End}(\E)} U+U( \nabla_1^{\mc{E}} - \nabla_2^{\mc{E}})\] 
is a differential operator which is elliptic of order $1$ with regularity $C^s_*$. Indeed, its principal symbol is that of $\nabla_1^{\mathrm{End}(\E)}$ and
\[
\sigma_{\mathrm{princ}}(\nabla_1^{\mathrm{End}(\E)})(x,\xi) : \mathrm{End}(\E_x) \rightarrow T^*_x \M \otimes \mathrm{End}(\E_x), \quad \sigma_{\mathrm{princ}}(\nabla_1^{\mathrm{End}(\E)})(x,\xi)u = \xi \otimes u,
\]
for $(x,\xi) \in T^*\M \setminus \left\{ 0 \right\}$ and $u \in \mathrm{End}(\E_x)$. This is clearly injective for $\xi \neq 0$. As a consequence, by rough elliptic theory (see \cite[Theorem 2.2.C p.55]{Taylor-91}), and using that $p \in C^{-\eps/2}_*(\M, \End(\E))$, we obtain that $p \in C^{s+1}_*(\M, \mathrm{U}(\E))$ and $\nabla_1^{\End(\E)}p+p( \nabla_1^{\mc{E}}-\nabla_2^{\mc{E}}) = 0$. In other words, the two connections are gauge equivalent (see \eqref{eq:gauge-equivalent-end}).
\end{proof}

\subsection{Linear unitary extensions and uniform boundedness}\label{ssec:linear-unitary-extensions}

Let $(\varphi_t)_{t \in \mathbb{R}}$ be a smooth Anosov flow on $\M$ generated by a smooth vector field $X$. Denote by $\mc{B}$ the set of all triples $(\E, h, \X)$, where $(\E, h) \to \M$ is a Hermitian vector bundle equipped with a linear differential operator $\X$ of order $1$ satisfying the following Leibniz rule
\begin{equation}
\label{equation:leibniz}
\X(fu) = (Xf) u + f \X u, \quad  \forall f \in C^\infty(\M),\,\, u \in C^\infty(\M,\E),
\end{equation}
and such that the norm of the propagator $e^{t\X}$ on $C^0(\M,\E)$ is constant equal to $1$ for all $t$. In other words, this last condition is equivalent to the fact that the parallel transport induced by $\X$ is unitary (with respect to $h$), or still equivalently we have
\begin{equation}\label{eq:unitary-operator}
	X h(u_1, u_2) = h(\X u_1, u_2) + h(u_1, \X u_2), \quad u_1, u_2 \in C^\infty(\M, \E). 
\end{equation}
We call such a triple a \emph{linear unitary extension} of $(\varphi_t)_{t \in \mathbb{R}}$; in the case we have only a vector bundle $\E$ equipped with $\X$ satisfying \eqref{equation:leibniz} we refer to the pair $(\E, \X)$ as a \emph{linear extension}. Then, define the moduli space $\B$ of linear unitary extensions as the quotient $\B := \mc{B}/\sim$ of $\mc{B}$ by the equivalence relation $\sim$, where $(\E_1, h_1, \X_1) \sim (\E_2, h_2, \X_2)$ if and only if there is a fibrewise unitary map $p \in C^\infty(\M,\mathrm{U}(\E_1, \E_2))$ such that $p^*\X_2(\bullet) := p^{-1}\X_2(p\bullet) = \X_1(\bullet)$.

Given $(\E_j, h_j, \X_j) \in \mc{B}$ for $j = 1, 2$, it is possible to introduce a first order differential operator $\X^{\Hom}$ on the homomorphism vector bundle $\Hom(\E_1, \E_2)$ (equipped with the natural Hermitian metric $\Hom(h_1, h_2)$ induced by $h_1$ and $h_2$) by the following formula
\begin{equation}\label{eq:hom-operator}
	\X_2(H u) =\X^{\Hom}(H) u + H \X_1 u, \quad \forall H \in C^\infty(\M, \Hom(\E_1, \E_2)),\,\, u \in C^\infty(\M, \E_1).
\end{equation}
It follows that the operator $\X^{\Hom}$ satisfies \eqref{equation:leibniz} and hence generates a linear unitary extension $(\Hom(\E_1, \E_2), \Hom(h_1, h_2), \X^{\Hom})$. Clearly $\X^{\Hom}$ depends on $\X_1$ and $\X_2$; we suppress these from the notation for simplicity and note that this dependence will be clear from context.

Finally, to obtain stability in the holonomy inverse problem we will impose certain uniformity bounds on triples in $\mc{B}$. It will be convenient to embed vector bundles in a large trivial vector bundle, which is standard but we give its proof for completeness.

\begin{proposition}\label{prop:embedding}
	Let $r \in \mathbb{Z}_{\geq 1}$ and let $(U_i)_{i = 1}^k$ be a cover of $\M$ by geodesically convex balls. Assume $b := (\E, h, \X) \in \mc{B}$ satisfies $\rk(\E) \leq r$. Then there exists a unitary embedding of vector bundles $p: \E \xhookrightarrow{} \M \times (\mathbb{C}^{rk}, h_{\mathrm{std}})$, where $h_{\mathrm{std}}$ is the canonical Hermitian metric on $\mathbb{C}^{rk}$. In particular, points $b \in \mc{B}$ as above determine linear unitary extensions $(p_*\E, h_\mathrm{std}, p_*\X)$ and consequently pairs $(\pi_{p_*\E}, Q(p_*\X))$, where $\pi_{p_*\E}$ is the orthogonal projector onto $p_*\E$ and $Q(p_*\X)$ is a skew-Hermitian section of $\End(p(\E))$.
\end{proposition}
\begin{proof}
	Assume without loss of generality that $\rk(\E) = r$; in fact the proof gives a unitary embedding into $\M \times \mathbb{C}^{k \rk(\E)}$ which embeds unitarily into $\M \times \mathbb{C}^{rk}$. Let $(\rho_i)_{i = 1}^k$ be a quadratic partition of unity subordinate to the cover $(U_i)_{i = 1}^k$, i.e. satisfying $\sum_{i = 1}^k \rho_i^2 \equiv 1$ and $\supp(\rho_i) \subset U_i$. For each $i =1, \dotsc, k$, there exists a smooth orthonormal frame $(\e_{ij})_{j = 1}^{r}$ of $\E|_{U_i}$; define
	\[
		p: \E \to \M \times \mathbb{C}^{k \times r}, \quad p(x, v) := \big(x, \big(h(v, \rho_i(x) \e_{ij}(x))\big)_{ij}\big),\quad i = 1,\dotsc, k, \quad j = 1,\dotsc, r.
	\]
	Then we have for arbitrary $x \in \M$, $u, v \in \E(x)$
	\begin{align*}
		h_{\mathrm{std}}(p(x, v), p(x, u)) &= \sum_{i, j} h(v, \rho_i(x) \e_{ij}(x)) h(\rho_j(x) \e_{ij}(x), u)\\ 
		&= \sum_{i, j} \rho_i(x)^2 h(v, \e_{ij}(x) h(u, \e_{ij}(x))) = h(v, u),
	\end{align*}
	where $h_{\mathrm{std}}$ is the standard Hermitian metric on $\mathbb{C}^{kr}$. Therefore $p$ is a unitary embedding, and the triple $(p(\E), h_{\mathrm{std}}, p_*\X)$ is a linear unitary extension equivalent to $(\E, h, \X)$, and in what follows we identify the two triples.
	
	 Denote by $\pi_{\E}$ the orthogonal projector onto $\E \subset \M \times \mathbb{C}^{kr}$. Using the Leibniz property \eqref{equation:leibniz} and the property \eqref{eq:unitary-operator}, we conclude that \[	\X = \pi_{\E} X \times \id_{\E} + Q(\X),\] 
	 where $Q(\X)$ is a skew-Hermitian endomorphism of $\E$ and $X \times \id_{\E}$ is the operator obtained by differentiating a section of $\E$ in each coordinate using the vector field $X$. Conversely, every pair $(\pi_{\E}, Q(\X))$ defines a linear unitary extension $(\E, h_{\mathrm{std}}, \X)$ by reversing this procedure, which completes the proof.
\end{proof}
\begin{remark}\rm
	Any two Hermitian metrics on a vector bundles are unitarily equivalent. Indeed, given two Hermitian metrics $h_1$ and $h_2$ on $\E$ there exists a canonical unitary isomorphism $p: (\E, h_1) \to (\E, h_2)$ defined as follows. At $x \in \M$, set $p(x): (\E(x), h_1) \to (\E(x), h_2)$ to be the unique unitary linear map which is Hermitian and positive definite with respect to $h_2$. The smoothness of such a map follows from a straightforward argument in a local trivialisation and the fact that $A \mapsto \sqrt{A}$ is smooth on the manifold of Hermitian and positive definite matrices.
	
	Also, we note that the map $p$ in Proposition \ref{prop:embedding} is \emph{not} canonical: it depends on the choices of orthonormal frames $(\e_{ij})_{j = 1}^r$ of $\E|_{U_i}$.
\end{remark}

Proposition \ref{prop:embedding} says that any triple $b \in \mc{B}$ is unitarily equivalent to a triple $(\mc{E}, h_{\mathrm{std}}, \X)$ and is in turn characterised by the pair $(\pi_{\E}, Q(\X))$. We say that $b$ is \emph{uniformly bounded by $(K, r)$} where $K > 0$ and $r \in \mathbb{Z}_{\geq 1}$ if there exists a unitary embedding $p$ as in Proposition \ref{prop:embedding}, such that:
\begin{equation}\label{eq:uniform-bound}
	\rk(\E) \leq r, \quad \|\pi_{\E}\|_{C^1(\M, \End(\mathbb{C}^{rk}))} \leq K, \quad \|\pi_{\E} Q(\X) \pi_{\E}\|_{C^1(\M, \End(\mathbb{C}^{rk}))} \leq K.
\end{equation}
We introduce the notation 
\[\widetilde{Q}(\X) := \pi_{\E} Q(\X) \pi_{\E} \in C^\infty(\M, \End(\mathbb{C}^{rk}))\] for future reference and note that $\widetilde{Q}(\X)$ is skew-Hermitian. For a vector field $Y$ we denote by $Y^\flat$ the $1$-form defined by the musical isomorphism through the fixed background metric $g$ on $\M$.

We finish this section with a lemma associating a unitary connection to every linear unitary extension, and transforming the uniform bound \eqref{eq:uniform-bound} to a uniform bound on \emph{curvature} of this connection.

\begin{lemma}\label{lemma:uniform-bound-curvature}
	For each triple $(\E, h_{\mathrm{std}}, \X)$ with $\rk(\E) \leq r$, there exists an associated unitary connection $\nabla = \nabla(\X)$ on $(\E, h_{\mathrm{std}})$ such that $\nabla_X = \X$. Moreover, there exists a constant $C > 0$ (depending only on $\M$, $X$, and $g$), such that the curvature $F_{\nabla}$ of the connection $\nabla$ satisfies
	\[\|F_{\nabla}\|_{C^0} \leq C(\|\pi_{\E}\|_{C^1}^2 + \|\widetilde{Q}(\X)\|_{C^1}).\]
\end{lemma}
\begin{proof}
	Define the unitary connection $\nabla$ by the covariant derivative formula
	\[\nabla = \pi_{\E} d + \widetilde{Q}(\X) X^{\flat}\wedge,\]
	where the unitarity follows from the fact that $\widetilde{Q}(\X)$ is skew-Hermitian. (This is the restriction of the connection $d + \widetilde{Q}(\X) X^{\flat} \wedge$ from $\M \times \mathbb{C}^{rk}$ to $\E$.) It follows that $\nabla_X = \pi_{\E} X \times \id_{\E} + Q(\X) = \X$.
	
	Next, we compute the curvature of $\nabla$, which is defined as the two form with values in $\End(\E)$ by the formula $F_{\nabla} = \nabla^2$ (the covariant derivative is naturally extended to form valued sections of $\E$). Then for a section $s$ of $\E \subset \M \times \mathbb{C}^{rk}$ (for simplicity we drop the argument in $\widetilde{Q}(\X)$)
	\begin{align*}
		F_{\nabla}s &= \nabla^2s = (\pi_{\E} d + \widetilde{Q} X^{\flat}\wedge)(\pi_{\E} ds + \widetilde{Q} X^{\flat}s)\\
		&= \pi_{\E} d\pi_{\E} \wedge ds + \pi_{\E} d\widetilde{Q} \wedge X^\flat s + \widetilde{Q} (dX^{\flat})s + \widetilde{Q} X^{\flat} \wedge (- ds + \pi_{\E} ds) + \widetilde{Q} X^{\flat} \wedge \widetilde{Q} X^{\flat} s\\
		&= (d\pi_{\E} - (d\pi_{\E}) \pi_{\E}) \wedge ds + \pi_{\E} d\widetilde{Q} \wedge X^\flat s + \widetilde{Q} (dX^{\flat})s - \widetilde{Q} X^{\flat} \wedge d\pi_{\E} s\\
		&= d\pi_{\E} \wedge d\pi_{\E} s + \pi_{\E} d\widetilde{Q} \wedge X^\flat s + \widetilde{Q} (dX^{\flat})s + d\widetilde{Q} (\id - \pi_{\E}) \wedge X^{\flat} s.
	\end{align*}
	where in the second line we used $d^2 = 0$, in the third line we used $X^{\flat} \wedge X^{\flat} = 0$, the fact that $d\pi_{\E} s + \pi_{\E} ds = ds$ (coming from the fact that $\pi_{\E}s = s$), and $\pi_{\E}^2 = \pi_{\E}$, and in the fourth lines we used again that $\pi_{\E} ds = ds - d\pi_{\E}s$ and $\widetilde{Q} d\pi_{\E} = d\widetilde{Q}(\id - \pi_{\E})$ (coming from the fact that $\widetilde{Q} = \widetilde{Q} \pi_{\E}$). This shows that
	\[F_{\nabla} = d\pi_{\E} \wedge d\pi_{\E} + \widetilde{Q} (dX^{\flat}) + (d\widetilde{Q} + [\pi_{\E}, d\widetilde{Q}]) \wedge X^{\flat},\]
	and the main estimate follows directly.
\end{proof}
\begin{remark}\rm
	Compared to the connection constructed in Lemma \ref{lemma:uniform-bound-curvature}, a more intrinsic choice of a connection $\nabla$ satisfying $\X = \nabla_X$ is the dynamical connection (see Remark \ref{remark:dynamical-connection} below). However, this connection has the downside that it is only H\"older regular, and hence its curvature is only distributionally defined which is not well suited for our repeated applications of the Ambrose-Singer formula in Section \ref{sec:approximate-livsic}.
\end{remark}

\subsection{Parry's representation}\label{ssec:parry-rep}

Let $(\E, h, \X) \in \mc{B}$ be a linear unitary of the Anosov flow $(\varphi_t)_{t \in \mathbb{R}}$; we assume the notation of \S \ref{ssec:fundamental-domain} and \S \ref{ssec:linear-unitary-extensions}. Define \emph{Parry's free monoid} as the formal free monoid generated by $\mc{H}$-words, namely,
\[
	\mathbf{G} := \left\{ \gamma_{1}^{k_1} \dotsb \gamma_n^{k_n} ~|~ n, k_i \in \Z_{\geq 0},\, \gamma_i \in \mc{H},\, i = 1, \dotsc, n\right\}.
\]
The empty word corresponds to the identity element $\mathbf{1}_{\G}$. Also observe that $\gamma_\star \in \G$ is a specific element.
 
Write $(\Phi_t)_{t \in \mathbb{R}}: \E \to \E$ for the parallel transport flow induced by the cocycle $(\E, h, \X)$ on $\E$. If $\pi: \E \to \M$ is the projection, it satisfies $\pi \circ \Phi_t = \varphi_t \circ \pi$, and $\Phi_t: \E(x) \to \E(\varphi_t x)$ is an isometry for for every $t \in \mathbb{R}$ and $x \in \M$; in other words, $(\Phi_t)_{t \in \mathbb{R}}$ is an isometric extension of $(\varphi_t)_{t \in \mathbb{R}}$. Moreover, let $\nabla$ be a unitary connection on $(\E, h)$ such that $\nabla_X = \X$ (for instance, take the connection provided by Lemma \ref{lemma:uniform-bound-curvature}). If $x, y \in \M$ are at distance smaller than the injectivity radius of the fixed backround metric, write $C_{x \to y}: \E(x) \to \E(y)$ for the parallel transport with respect to $\nabla$ along the unique unit speed short geodesic connecting $x$ to $y$. Also, write $C(x, t)e := \Phi_t(x, e)$ for the parallel along the flow trajectories, where $x \in \M$, $e\in \E(x)$, and $t \in \mathbb{R}$.

We are now in shape to define \emph{stable/central/unstable holonomies}. If $x, y \in \M$ and $e \in \E(x)$ we set:
\begin{itemize}
	\item if $y = \varphi_tx$ for some $t \in \mathbb{R}$, then: $\mathbf{H}_{x \to y}e := \Phi_t(x, e)$;
	\item if $y \in W^{s}(x)$, then $f := \mathbf{H}^{s}_{x \to y}e \in \E(y)$ is defined to be the unique vector such that $d_{\E}(\Phi_t(x, f), \Phi_t(x, e)) \to 0$ as $t \to \infty$, where $d_{\E}(\bullet, \bullet)$ is the distance on $\E$. This is well-defined by the Ambrose-Singer lemma, see \cite[Lemma 3.14]{Cekic-Lefeuvre-22} and in fact $\mathbf{H}^{s}_{x \to y}: \E(x) \to \E(y)$ is a unitary isomorphism. More precisely, we have
	\begin{equation}\label{eq:stable-holonomy}
		\mathbf{H}_{x \to y}^s := \lim_{t \to \infty} C(\varphi_t y, -t) C_{\varphi_t x\to \varphi_t y}C(x, t).
	\end{equation}
	\item if $y \in W^{y}(x)$, then $f := \mathbf{H}^{u}_{x \to y}e \in \E(y)$ is defined to be the unique vector such that $d_{\E}(\Phi_{-t}(x, f), \Phi_{-t}(x, e)) \to 0$ as $t \to \infty$, where this is as above well-defined by the Ambrose-Singer lemma, and in fact $\mathbf{H}^{u}_{x \to y}: \E(x) \to \E(y)$ is a unitary isomorphism.
\end{itemize}
These definitions are independent of the choice of $\nabla$, as follows once more by an application of the Ambrose-Singer lemma. From the defining properties it follows that when $y, z \in W^{s/u}(x)$, one has the property $\mathbf{H}^{s/u}_{y \to z} \mathbf{H}^{s/u}_{x \to y} = \mathbf{H}^{s/u}_{x \to z}$.

\begin{remark}\rm\label{remark:dynamical-connection}
	The \emph{dynamical connection} is the connection whose parallel transports along curves in stable/unstable foliations and the flow direction, are given by stable/unstable and the central holonomies, respectively. It can be shown that this connection is H\"older regular, flat when restricted to stable/unstable leaves, and is an interesting object on its own for further studies.
\end{remark}

Finally, \emph{Parry's representation} $\rho: \mathbf{G} \to \operatorname{U}(\E(x_\star))$ is defined as
\begin{equation}\label{eq:parry-rep-def}
	\rho(\gamma) := \mathbf{H}^s_{x_s(\gamma) \to x_\star} \mathbf{H}^c_{x_u(\gamma) \to x_s(\gamma)} \mathbf{H}^u_{x_\star \to x_u(\gamma)}, \quad \gamma \in \mc{H},
\end{equation}
and extending multiplicatively to $\mathbf{G}$. It will be convenient for us to approximate $\rho(\gamma)$ by parallel transports over long orbits, i.e. to take one step back and use the limit definition as in \eqref{eq:stable-holonomy}. To this end, for $m, n \in \mathbb{Z}_{\geq 0}$ we define
\[
	\rho(\gamma; m, n) := C(x_\star, - nT_\star) C_{x_s(\gamma; n) \to x_\star} C(x_u(\gamma; m), x_s(\gamma; n)) C_{x_\star \to x_u(\gamma; m)}C(x_\star, - mT_\star).
\]
Another application of the Ambrose-Singer lemma (see \cite[Lemma 3.14]{Cekic-Lefeuvre-22}) together with \eqref{eq:estimate-translate} shows that
\begin{equation}\label{eq:approximate-parry}
	\rho(\gamma) = \rho(\gamma; m, n) + \|F_\nabla\|_{C^0}\mc{O}(e^{-\min(m, n) T_\star \mu}), \quad m, n \to \infty,
\end{equation}
where $F_\nabla$ is the curvature of $\nabla$. We will revisit the notion of Parry's representation in Section \ref{sec:approximate-livsic}.

\subsection{Pollicott-Ruelle resonances}\label{ssec:P-R-resonances}
In the past twenty years there was a significant progress in the theory of anisotropic Banach or Sobolev spaces for hyperbolic dynamics, which significantly increased the understanding of Pollicott-Ruelle resonances -- the eigenvalues of the flow generator acting in these spaces. For more details, see Baladi \cite{Baladi-05}, Baladi-Tsujii \cite{Baladi-Tsujii-07}, Blank-Keller-Liverani \cite{Blank-Keller-Liverani-02}, Gou\"ezel-Liverani \cite{Gouezel-Liverani-06}, and Liverani \cite{Liverani-04}. In this paper, we take the viewpoint of microlocal methods for dynamical resonances, introduced by Faure-Sj\"ostrand \cite{Faure-Sjostrand-11} and Dyatlov-Zworski \cite{Dyatlov-Zworski-16}.

Let $(\E, \X)$ be a linear extension of the Anosov flow $(\varphi_t)_{t\in \mathbb{R}}$ generated by a smooth vector field $X$. Let $m \in C^\infty(T^*\M, [-1, 1])$ be a function which is $0$-homogeneous in $\xi$ for $|\xi| \geq 1$, and satisfies
\begin{equation}\label{eq:order-function}
	m = -1\,\, \mathrm{near}\,\, E_u^*, \quad m = 1\,\, \mathrm{near}\,\, E_s^*, \quad \mathbb{X}m \leq 0 \,\, \mathrm{for}\,\, |\xi| \geq 1,
\end{equation}
where $\mathbb{X}$ is the Hamiltonian lift of $X$, that is, the generator of $(x, \xi) \mapsto (\varphi_tx, \xi \circ d\varphi_{t}^{-1}(x))$, and define $G(x, \xi) := m(x, \xi) \log(1 + |\xi|)$. Denote by $\Op(\bullet)$ the quantisation procedure on $\E \to \M$ and for $s \geq 0$ introduce the anisotropic Sobolev spaces
\[\mc{H}^{s}_\pm(\M, \E) := \mc{H}_{\pm}^s := \Op(e^{\mp sG} \times \id_{\E}) L^2(\M, \E),\]
where the space $L^2(\M, \E)$ is with defined with respect to a smooth measure on $\M$ and using $h$, and $\Op(e^{\mp sG} \times \id_{\E})$ is possibly modifed by a smoothing (finite rank) operator to make it invertible. This construction is carried out in \cite{Faure-Sjostrand-11} and \cite{Dyatlov-Zworski-16}, who show there exist $C, c > 0$ (depending only on $X$) such that the resolvent $(\pm \X + z)^{-1}: \mc{H}^s_{\pm} \to \mc{H}^s_{\pm}$, defined initially for $\re(z) \gg 1$, extends meromorphically as a map 
\begin{equation}\label{eq:mero-ext}
	\{z \in \mathbb{C} \mid \re(z) > C - cs\} \ni z \mapsto (\pm \X + z)^{-1} \in \mc{L}(\mc{H}^s_{\pm},  \mc{H}^s_{\pm}).
\end{equation}

If the flow $(\varphi_t)_{t \in \mathbb{R}}$ preserves a smooth volume $\dd \mu$, and $(\E, h, \X)$ is a unitary linear extension, then we may take $C = 0$ in the above, and the resolvent is holomorphic as a map
\begin{equation}\label{eq:holo-ext}
	\{z \in \mathbb{C} \mid \re(z) > 0\} \ni z \mapsto (\pm \X + z)^{-1} \in \mc{L}(L^2(\M, \E), L^2(\M, \E)),
\end{equation}
since $e^{t\X}$ acts by unitary isomorphisms on $L^2(\M, \E)$.

Finally, note that when working with Hermitian bundles $(\E, h)$, we may choose the anistropic spaces such that $\mc{H}^s_\pm(\M, \End(\E))$ is invariant under taking adjoints. Indeed, this can be achieved by replacing the order function $m(x, \xi)$ by $\frac{1}{2}(m(x, \xi) + m(x, -\xi))$ and observing it also satisfies the required conditions in \eqref{eq:order-function} (this works because $E_{u/s}$ have a vector bundle structure). The invariance under the adjoint then follows from a straightforward computation in local coordinates.
%%%
%%%
%%%

\section{Pollicott-Ruelle resonances in low regularity}

\label{section:low-regularity}

In this section we prove some properties of the resolvent for Anosov flows and potentials in low regularity. These results will be used in the proof of the Main Theorem later in Section \ref{sec:main-proof}.

\subsection{A multiplication lemma}

Let $X$ generate a smooth Anosov flow on an $n$-manifold $\M$, preserving a smooth measure $\dd\mu$, and let $(\E, h, \X) \rightarrow \M$ be a linear unitary extension of this flow. As explained in \S\ref{ssec:P-R-resonances}, the resolvent map $z \mapsto (\X + z)^{-1}$ is a holomorphic linear map on $L^2(\M, \E)$ in the right half-plane, and meromorphic as acting on \emph{anisotropic Sobolev spaces} $\mc{H}^s_+$ for $\re(z) > -cs$, where $c > 0$ depends only on the flow and $s > 0$ is arbitrary.

We will need a first technical lemma:

\begin{lemma}
\label{lemma:multiplication}
For any $\beta > 2s$, the follow map is continuous:
\[
	C^\beta(\M, \End(\E)) \times \mc{H}_+^s(\M, \E) \ni (f_1, f_2) \mapsto f_1 f_2 \in \mc{H}_+^s(\M, \E).
\]
\end{lemma}

\begin{proof}
	First of all, observe that this is only a local statement. Hence, up to taking a cover by local charts and using local trivialisations, it is sufficient to prove the lemma in $\R^n$ for functions. The proof relies on \emph{paradifferential calculus}, see \cite{Bahouri-Chemin-Danchin-11} or \cite[Chapter 10]{Hormander-97} for further references. We will however use the convenient global approach which was recently set up by Guedes-Bonthonneau--Guillarmou--de Poyferre \cite{Guillarmou-de-Poyferre-22}. The proof is essentially contained in \cite[Proposition 4.1]{Guillarmou-de-Poyferre-22} but we give it here for completeness; we will freely use the notation of \cite{Guillarmou-de-Poyferre-22}.
	
For notational simplicity, in what follows we give the proof in the case of $\E = \M \times \mathbb{C}$; the transition to the general case is given by the verbatim same argument replacing functions by sections (see also the discussion on \cite[p. 697]{Guillarmou-de-Poyferre-22} for the case of vector bundles). By \cite[Lemma 2.14]{Guillarmou-de-Poyferre-22}, there is a splitting 
\begin{equation}\label{eq:splitting}
	f = f^{\sharp} + f^{\flat}, \quad f^{\sharp} \in {}^{\beta} \widetilde{S}^0_{1, 1}(\M), \quad f^{\flat} \in \cap_{\alpha \in (0, \beta)} C^{\beta - \alpha} S^{-\alpha}(\M).
\end{equation}
Then $T_f := \Op(f^{\sharp}) \in {}^\beta \widetilde{\Psi}^0_{1, 1}(\M)$ is called the associated \emph{paradifferential operator}, and we have
\begin{equation}\label{eq:flat-bounded}
	\forall \alpha \in (0, \beta), \quad \Op(f^\flat): H^{\alpha - \beta}(\M) \to H^{\alpha}(\M),
\end{equation}
is bounded as a map on Sobolev spaces. We want to prove that the following map is bounded (where $M_f$ is multiplication by $f$):
\[\Op(e^{-sG})^{-1} M_f \Op(e^{-sG}): L^2(\M) \to L^2(\M).\]
According to the splitting \eqref{eq:splitting}, it suffices to show that that the terms
\[\Op(e^{-sG})^{-1} T_f \Op(e^{-sG}), \quad \Op(e^{-sG})^{-1} \Op(f^{\flat}) \Op(e^{-sG}): L^2(\M) \to L^2(\M)\]
are bounded. For the latter term, we simply observe that $\Op(e^{sG})^{-1} = \Op(e^{-sG}) + \mc{O}(\Psi^{(s - 1) +}(\M))$, as well as that $\Op(e^{\pm sG}) \in \Psi^{s+}(\M)$. Here and throughout the proof, we write $\Psi^{\bullet+}(\M) := \cup_{k > 0} \Psi^{\bullet + k}(\M)$. Therefore, the composition is bounded as:
\[\Op(e^{-sG})^{-1} \Op(f^{\flat}) \Op(e^{-sG}): L^2(\M) \to H^{-2s + \beta}(\M) \subset L^2(\M),\]
where we also used \eqref{eq:flat-bounded} and the assumption that $\beta > 2s$, showing that this contribution is moreover compact.

For the former term, we argue in two steps as follows. Observe first that
\begin{equation}\label{eq:composition-1}
	T_f \Op(e^{-sG}) = \Op(f_1) + \mc{O}(\Psi^{-\infty}) \in {}^\beta \widetilde{\Psi}^{s +}_{1, 1}(\M), \quad f_1 = f^{\sharp} e^{-s G} + \mc{O}(S^{(s - \beta) + }_{1, 1}(\M)),
\end{equation}
where we used \cite[Proposition 2.8 and 2.9]{Guillarmou-de-Poyferre-22} (note that $\Op(e^{-sG}) \in \Psi^{s+}(\M) \subset {}^\beta \widetilde{\Psi}^{s+}_{1, 1}(\M)$). Next, we note that by the parametrix consutrction we may write 
\[
	\Op(e^{-sG})^{-1} = \Op(F e^{sG}) + \mc{O}(\Psi^{-\infty}), \quad F \in S^0(\M).
\]
We next compose on the left to get, again using \cite[Proposition 2.8 and 2.9]{Guillarmou-de-Poyferre-22}:
\begin{equation}\label{eq:composition-2}
	\Op(e^{-sG})^{-1} T_f \Op(e^{-sG}) =  \Op(f_2) + \mc{O}(\Psi^{-\infty}), \quad f_2 = f_1 F e^{sG} + \mc{O}(S^{(2s - \beta)+}_{1, 1}(\M)).
\end{equation}
Using \eqref{eq:composition-1} and \eqref{eq:composition-2}, we conclude that
\[\Op(e^{-sG})^{-1} T_f \Op(e^{-sG}) = \underbrace{\Op(F f^{\sharp})}_{\in {}^\beta\widetilde{\Psi}^{0}_{1,1}(\M)} + \mc{O}(\Psi^{(2s - \beta)+}_{1, 1}(\M)),\]
which concludes the proof since both summands are bounded as maps $L^2(\M) \to L^2(\M)$ by \cite[Proposition 2.6]{Guillarmou-de-Poyferre-22} (in fact the sum on the right hand side is equal to $T_f$ plus compact terms).
\end{proof}

Recall now that $\mc{H}^s_- = \Op(e^{sG} \times \id_{\E}) L^2(\M, \E)$. Note that the $L^2$-pairing $C^\infty(\M, \E) \times C^\infty(\M, \E) \to \mathbb{C}$, $(u, v) \mapsto \langle{u, v}\rangle_{L^2}$, extends continuously to $\mc{H}^s_+ \times \mc{H}^s_- \to \mathbb{C}$ by the fact that $\Op(e^{sG} \times \id_{\E})^* \Op(e^{-sG} \times \id_{\E}) \in \Psi^0(\M, \E)$.

\begin{lemma}
\label{lemma:multiplication2}
Let $\beta > 0$ such that $\beta > 2s$. Then the following multiplication map is continuous:
\[
	\mc{H}^s_+(\M, \End(\E)) \times \mc{H}^s_-(\M, \End(\E)) \ni (f_1, f_2) \mapsto f_1 f_2 \in H^{-(n/2 + \beta)}(\M, \End(\E)).
\]
\end{lemma}

\begin{proof}
Let $N := n/2 + \beta$ and $\phi \in H^{N}(\M, \End(\E)) \subset C^\beta(\M, \End(\E))$ (the last inclusion follows by the Sobolev embedding theorem). Hence by the multiplication Lemma \ref{lemma:multiplication} (more precisely, its version for multiplication of $\End(\E)$ sections) we get:
\[ 
	\|f_2 \phi\|_{\mc{H}_-^s} \lesssim \|f_2\|_{\mc{H}_-^s} \|\phi\|_{C^{\beta}} \lesssim \|f_2\|_{\mc{H}_-} \|\phi\|_{H^N},
\]
and moreover using this bound, we have that:
\[
	\langle f_1 f_2, \phi \rangle_{L^2} = \langle f_1, \overline{f_2} \phi \rangle_{L^2} \lesssim \|f_1\|_{\mc{H}_+} \|\overline{f_2} \phi\|_{\mc{H}_-} \lesssim \|f_1\|_{\mc{H}_+}  \|f_2\|_{\mc{H}_-} \|\phi\|_{H^N}.
\]
By duality, this proves the announced result.
\end{proof}

\subsection{Meromorphic extension in low regularity}

\label{appendix:b}

Fix $s > 0$ and let $\beta > 2s$ (so that the multiplication Lemma \ref{lemma:multiplication} applies) and $V \in C^\beta(\M,\End(\E))$. We want to look at the meromorphic extension of $z \mapsto(\X + V + z)^{-1}$ in the rectangle
\[\Omega_R := \left\{z \in \mathbb{C} \mid \re(z) > -cs/2,\,\, |\im(z)|<R \right\}, \quad R > 0.\]

\begin{lemma}\label{lemma:mero-perturbation}
Fix $R > 0$. There exists $\eps = \eps(R) > 0$ such that for all $V \in C^\beta(\M, \End(\E))$ such that $\|V\|_{C^\beta} < \eps$, the following map is meromorphic:
\[
	\Omega_R \ni z \mapsto (\X + V + z)^{-1} \in \mc{L}(\mc{H}^s_+,\mc{H}^s_+).
\]
\end{lemma}

\begin{proof}
Using the condition \eqref{eq:order-function} and the Anosov condition \eqref{eq:anosov}, it is possible to show that the propagators $(e^{-t\X})_{t \geq 0}: \mc{H}_+^s \to \mc{H}_+^s$ are a locally bounded family of operators (see for instance \cite[Sections 2 and 3]{Faure-Sjostrand-11} for the related case of flow pullback $(\varphi_{-t}^*)_{t \geq 0}$). Therefore, there exists $C_0 > 0$ such that $\|e^{-t \X}\|_{\mc{H}^s_+ \to \mc{H}_+^s} \leq C_0 e^{C_0t}$ for $t > 0$. By Lemma \ref{lemma:multiplication}, the propagator $e^{\int_0^t e^{-p\X} V\, dp}$ is also bounded as a map $\mc{H}^{s}_+ \to \mc{H}^s_+$ by $C_V e^{C_V t}$ for some other $C_V > 0$. We conclude that for $C := C_0 + C_V$ we have
\[
	(\X + V + z)^{-1} := \int_0^\infty e^{-\int_0^t e^{-p \X} V\, dp} e^{-zt} e^{-t\X}\, dt :\mc{H}^s_+ \to \mc{H}^s_+,
\]
is a holomorphic family of operators for $\re(z) > C + \varepsilon$. Using the fact that $\X + z: \mc{D}^s_+ \to \mc{H}^s_+$ is Fredholm for $\re(z) > -cs$ (which follows from \eqref{eq:mero-ext}, and where $\mc{D}^s_+ = \{u \in \mc{H}^s_+ \mid \X u \in \mc{H}^s_+\}$), and openness of the set of Fredholm operators in $\mc{L}(\mc{D}^s_+, \mc{H}^s_+)$, as well as compactness of $\Omega_R$, we conclude that for $\varepsilon > 0$ small enough and $z \in \Omega_R$ the family of maps $\X + V + z: \mc{D}^s_+ \to \mc{H}^s_+$ is Fredholm. The main claim then follows by the analytic Fredholm theorem (see \cite[Theorem D.4]{Zworski-12}).
\end{proof}

Let $\gamma$ be a small oriented contour around $0$ such that it encloses only the (potential) resonance $0$ of $\X$. By Lemma \ref{lemma:mero-perturbation}, if $\|V\|_{C^\beta}$ is sufficiently small we can speak of the inverse of the perturbed operator $\X + V + z$ on $\gamma$.

\begin{lemma}\label{lemma:perturbed-resolvent}
There exist $C, \eps > 0$ such that for all $\|V\|_{C^\beta} < \eps$, for all $z \in \gamma$ we have:
\begin{equation*}
	\|(\X + V + z)^{-1}\|_{\mc{H}^s_+\rightarrow \mc{H}^s_+} \leq C.
\end{equation*}
\end{lemma}
\begin{proof}
By the very definition of $\gamma$, $(\X + z)^{-1}$ is uniformly bounded as a map $\mc{H}^s_+ \to \mc{H}^s_+$ for $z \in \gamma$. We then write:
\begin{equation}\label{eq:identity}
	\X + V + z = (\X + z)(\id + (\X + z)^{-1}V),
\end{equation}
as maps $\mc{D}^s_+ \to \mc{H}^s_+$, where $\mc{D}^s_+ = \{u \in \mc{H}^s_+ \mid \X u \in \mc{H}^s_+\}$. Observe that by the multiplication Lemma \ref{lemma:multiplication}, we obtain:
\[
	\|(\X + z)^{-1}V\|_{\mc{H}^s_+ \rightarrow \mc{H}^s_+} \leq C \|V\|_{C^\beta} < 1/2,
\]
for $\eps > 0$ small enough. Similarly, we may assume that $\|(\X + z)^{-1}V\|_{\mc{D}^s_+ \rightarrow \mc{D}^s_+} < 1/2$ for $\varepsilon > 0$ sufficiently small (here $\mc{D}^s_+$ is equipped with the graph norm $\|u\|^2_{\mc{D}^s_+} = \|u\|^2_{\mc{H}^s_+} + \|\X u\|^2_{\mc{H}^s_+}$) and hence $\id + (\X + z)^{-1}V$ is invertible on both $\mc{H}^s_+$ and $\mc{D}^s_+$, uniformly in $z \in \gamma$. This concludes the proof by inverting \eqref{eq:identity}.
\end{proof}

By the previous lemma, for each $V$ with $\|V\|_{C^\beta} < \varepsilon$ sufficiently small we can define the following finite rank spectral projector:
\[
	\Pi_V := \dfrac{1}{2 \pi i} \int_{\gamma} (z+\X+V)^{-1}\, \dd z,
\]
onto the resonant spaces contained inside $\gamma$.

\begin{lemma}
\label{lemma:projector-v}
The following map is smooth:
\[
	C^\beta(\M,\End(\E)) \cap \left\{\|V\|_{C^\beta} < \eps\right\} \ni V \mapsto \Pi_V \in \mc{L}(\mc{H}^s_+,\mc{H}^s_+).
\]
Moreover, if we consider the sum of all resonances of $\X + V$ enclosed by $\gamma$, that is,
\[
	\lambda_V := \Tr((\X + V)\Pi_V),
\]
then the following map is smooth:
\[
	C^\beta(\M,\End(\E)) \cap \left\{\|V\|_{C^\beta} < \eps\right\} \ni V \mapsto \lambda_V \in \C.
\]
\end{lemma}

\begin{proof}
Let us prove that the first map is $C^1$. Fix $V \in C^\beta(\M,\End(\E))$ such that $\|V\|_{C^\beta} < \eps$ and consider $V + sW$, where $W \in C^\beta(\M,\End(\E))$ and $s \in \R$ is small enough. Then:
\[
	\partial_s \Pi_{V+sW}|_{s=0} = -\dfrac{1}{2 \pi i} \int_{\gamma} (z+\X+V)^{-1} W (z+\X+V)^{-1} \dd z,
\]
and combining Lemma \ref{lemma:perturbed-resolvent} together with the multiplication Lemma \ref{lemma:multiplication}, we obtain:
\[
	\|(z+\X+V)^{-1} W (z+\X+V)^{-1}\|_{\mc{H}^s_+ \rightarrow \mc{H}^s_+} \leq C \|W\|_{C^\beta},
\]
for some uniform constant $C > 0$, i.e. independent of $z \in \gamma$ and $V$ with $\|V\|_{C^\beta} < \eps$. Hence $\|\partial_s \Pi_{V+sW}|_{s=0}\|_{\mc{H}^s_+ \rightarrow \mc{H}^s_+} \leq C \|W\|_{C^\beta}$, proving the $C^1$ claim. Arguing similarly, we obtain that for each $k \in \mathbb{Z}_{\geq 2}$, there exists another uniform $C > 0$ such that $\|\partial_s^k \Pi_{V+sW}|_{s=0}\|_{\mc{H}^s_+ \rightarrow \mc{H}^s_+} \leq C \|W\|_{C^\beta}^k$, which shows that the map is in fact smooth.

The smoothness of the trace map then follows: let $(\e_i)_{i = 1}^k \subset \mc{H}^s_+$ be a basis of resonant states at $V = 0$ (i.e. a basis $\ran(\Pi_{V = 0})$). Note that $(\Pi_V\e_i)_{i = 1}^k$ form a basis of $\ran \Pi_V$ by the fact that the rank of $\Pi_V$ is locally constant (see \cite[Lemma 6.2]{Cekic-Paternain-20}). Since $\ran((\X + V) \Pi_V) \subset \ran \Pi_V$ it is possible to express the required trace using the basis $(\Pi_V\e_i)_{i = 1}^k$ which is smooth by the first part, which completes the proof.
\end{proof}
%%%
%%%
%%%

\section{Stability estimate in representation theory}

\label{section:representation-theory}

Recall that a monoid $\mathbf{G}$ is a set endowed with a neutral element $\mathbf{1}_{\mathbf{G}}$ together with an associative product but, unlike groups, it does not have inverses. If $\rho : \mathbf{G} \to \mathrm{Aut}(\E)$ is a representation of $\mathbf{G}$ in a finite-dimensional vector space $\E$, its \emph{character} is defined as $\chi_\rho(\bullet) := \Tr(\rho(\bullet))$. The well-known character theorem asserts that two (finite-dimensional) unitary representations of a monoid with same character are isomorphic:

\begin{theorem}
Let $\rho_0 : \G \to \mathrm{U}(\E_0)$ and $\rho : \G \to \mathrm{U}(\E)$ be two finite-dimensional unitary representations of the monoid $\mathbf{G}$ such that $\chi_\rho = \chi_{\rho_0}$. Then there exists $\mathbf{P} \in \mathrm{U}(\E, \E_0)$ such that:
\[
\rho_0(g) = \mathbf{P} \rho(g) \mathbf{P}^*, \qquad \forall g \in \mathbf{G},
\]
that is, the representations are unitarily equivalent.
\end{theorem}

We refer to \cite[Corollary 3.8]{Lang-02} for instance. In this section, we provide a stability estimate version of the previous theorem.

\subsection{Representations with nearly-equal characters}

From now on, $\mathbf{G}$ is a free monoid. It will be convenient to work with the \emph{monoid algebra} $\C[\mathbf{G}]$ generated $\mathbf{G}$ over $\C$. We will write $\E$ for a complex vector space equipped with an inner product. If $\rho : \C[\mathbf{G}] \rightarrow \mathrm{U}(\E)$ is a finite-dimensional unitary representation, then $\E$ splits as the orthogonal sum (up to unitary isomorphism):
\[
\E = \oplus_{i=1}^K V_i^{n_i},
\]
where $n_i \in \mathbb{Z}_{> 0}$. Each $V_i$ is $\mathbf{G}$-invariant and irreducible, where for $i \neq j$, the induced representations $\rho|_{V_i} : \C[\mathbf{G}] \rightarrow \End(V_i)$ and $\rho|_{V_j}$ are \emph{not} isomorphic and $V_i^{n_i} = V_i \times \dotsb \times V_i$ denotes $n_i$ copies of the same representation.

Define $\Delta_{n_i} : \End(V_i) \rightarrow \End(V_i^{n_i})$ by $\Delta_{n_i}(u) := u \oplus u \oplus \dotsb \oplus u$, where the sum is performed $n_i$-times. Burnside's theorem (see \cite[Corollary 3.4]{Lang-02}) characterizes the image $\rho(\C[\mathbf{G}]) \subset \End(\E)$ as:
\begin{equation}
\label{equation:burnside}
\rho(\C[\mathbf{G}]) = \oplus_{i=1}^K \Delta_{n_i} \End(V_i),
\end{equation}
In particular, it follows from \eqref{equation:burnside} that $\rho(\C[\mathbf{G}])$ is stable by the adjoint operator ${}^*$. We let $\chi_\rho(\bullet) := \Tr(\rho(\bullet))$ be the character of the representation $\rho$.

Fix an arbitrary unitary representation $\rho_{0} : \mathbf{G} \rightarrow \mathrm{U}(\E_0)$ and $g_1,\dotso,g_{N_0} \in \G$ such that 
\[\spann\big(\rho_0(g_1), \dotso, \rho_0(g_{N_0})\big) = \rho_0(\C[\G]),\] 
where $N_0 := \dim \rho_0(\C[\G])$. We will say that a subset $\G_0 \subset \G$ is \emph{sufficiently large} if 
\begin{equation}\label{eq:sufficiently-large}
	g_i, g_i g_j, g_i g_j g_k \in \mathbf{G}_0, \quad \forall i, j, k \in \{1, \dotso, N_0\}.
\end{equation}
Note that this depends both on the representation $\rho_0$ and the chosen basis of $\rho_0(\C[\G])$. It will be convenient to introduce 
\begin{equation}\label{eq:prime-notation}
	\mathbf{G}'_0 := \left\{ g \in \mathbf{G}_0 \mid g_i g \in \mathbf{G}_0, \, \forall i = 1, \dotso, N\right\}.
\end{equation}
Note that by the assumption \eqref{eq:sufficiently-large}, we have that $g_i, g_{i} g_j \in \mathbf{G}_0'$ for all $i, j \in \{1, \dotso, N_0\}$. 

We will prove that the following estimate holds:

\begin{lemma}
\label{lemma:stability-representation}
Let $\rho_0 : \G \to \mathrm{U}(\E_0)$ be irreducible, and let $\rho : \G \to \mathrm{U}(\E)$ be another finite-dimensional unitary representation of the monoid $\mathbf{G}$ such that $\dim \E = \dim \E_0$. There exist $\eps_0, C > 0$, depending only on $\rho_0$, such that that the following holds: if $\mathbf{G}_0 \subset \G$ is sufficiently large (see \eqref{eq:sufficiently-large}) and there exists $\eps \in [0, \eps_0)$ such that
\begin{equation}
\label{equation:trace}
	\|\chi_{\rho_0}-\chi_{\rho}\|_{L^\infty(\mathbf{G}_0)} \leq \eps,
\end{equation}
then there exists $\mathbf{P} \in \mathrm{U}(\E, \E_0)$ such that for all $g \in \mathbf{G}_{0}'$:
\[
	\|\rho_0(g)-\mathbf{P} \rho(g) \mathbf{P}^* \| \leq C \eps.
\]
\end{lemma}

The notation $\rho_0$ is used to insist on the fact that this representation is fixed and that we vary $\rho$ close to $\rho_0$. Here $\|\bullet\|$ denotes the operator norm of a linear map, and in what follows we will sometimes use the Frobenius norm $\|\bullet\|_{\mathrm{Fr}}^2 = \Tr(\bullet^* \cdot \bullet)$. According to Burnside's Theorem (see \eqref{equation:burnside}), the irreducibility assumption implies that $N_0 = (\dim \E_0)^2$.

\begin{remark}\rm \label{remark:necessary}
	We will see below in Lemma \ref{lemma:A-injective} that for $\eps_0 > 0$ small enough (depending on $\rho_0$), $\dim \rho(\mathbb{C}[\mathbf{G}]) = \dim \rho_0(\mathbb{C}[\mathbf{G}])$ which is a crucial step in the argument. Let us briefly comment on the necessity of the assumptions that $\rho_0$ is irreducible and $\dim \E = \dim \E_0$. If we dropped the former assumption (but still keep $\dim \E = \dim \E_0$) it is easy to find examples such that for arbitrarily small $\eps > 0$ there are $\rho$ irreducible and hence $\dim \rho(\mathbb{C}[\mathbf{G}]) > \dim \rho_0(\mathbb{C}[\mathbf{G}])$. To see this, for instance take $\mathbf{G}$ to be generated by two elements and $\rho_0$ to be defined by two $2\times 2$ commuting unitary matrices on these; then $\rho_0$ splits into two irreducible components, but there are sufficiently small perturbations for which these matrices do not commute and which is irreducible. Similarly, there are examples where $\chi_{\rho_0} = \chi_{\rho}$ on $\mathbf{G}_0$ but $\dim \E > \dim \E_0$: for instance, take $\rho_0$ as before, $\mathbf{G}_0$ to be the generating set, and define the representation $\rho$ on $\mathbf{G}_0$ as $\rho := \rho_0 \oplus \begin{pmatrix}
	0 & a\\
	b & 0
	\end{pmatrix}$, where $a, b \in \mathbb{S}^1$ are fixed. Then clearly $\dim \E = \dim \E_0 + 2 = 4$ and $\dim \rho(\mathbb{C}[\mathbf{G}]) > \dim \rho_0(\mathbb{C}[\mathbf{G}])$. In summary, the assumptions that $\rho_0$ is irreducible and $\dim \E = \dim \E_0$ seem to be natural ones to ensure that $\dim \rho(\mathbb{C}[\mathbf{G}]) = \dim \rho_0(\mathbb{C}[\mathbf{G}])$.
	
	Finally, observe that \emph{if} we had taken $\mathbf{1}_{\mathbf{G}}$ inside the generating set $\{g_1, \dotsc, g_{N_0}\}$, then the assumption \eqref{equation:trace} would guarantee (for $\eps < 1$) that $\dim \E = \dim \E_0$. However, in the application we have in mind we cannot make this assumption and hence we work in this more general setting, from which $\dim \E = \dim \E_0$ \emph{cannot be derived} as a consequence (see the previous paragraph for a counterexample).
\end{remark}

\subsection{Proof of the stability estimate}
In this section we prove Lemma \ref{lemma:stability-representation} and to this end we assume its hypotheses. We form the matrices $M_{0}$ and $M$ whose entries are given by $(M_{0})_{i,j=1}^{N_0} := \chi_{\rho_{0}}(g_i g_j)$ and $(M)_{i,j=1}^{N_0} := \chi_{\rho}(g_i g_j)$. We start with a few auxiliary statements.

\begin{lemma}\label{lemma:M-invertible}
The matrix $M_0$ is invertible. Moreover, if $\eps_0 > 0$ is chosen small enough (depending only on $\rho_0$), $M$ is also invertible.
\end{lemma}

\begin{proof}
If $M_0$ is invertible, it is clear that $M$ is also invertible for $\eps_0$ small enough since $|(M_0)_{ij}-M_{ij}| < \eps < \eps_0$. The invertibility of $M_0$ follows from the Burnside theorem stated in \eqref{equation:burnside}. Indeed, assume $X = (x_1,\dotso,x_{N_0})^\top \in \mathbb{C}^{N_0}$ is such that $M_0X = 0$. Then for all $i = 1, \dotso, N_0$:
\begin{equation}
\label{equation:annule}
\sum_{j=1}^{N_0} \Tr\big(x_j \rho_0(g_ig_j)\big) = \Tr\big(\rho_0(g_i) \cdot Y\big) = 0,
\end{equation}
where $Y := \sum_{j=1}^{N_0} x_j \rho_0(g_j)$. By assumption, the $\rho_0(g_i)$'s span $\rho_0(\C[\mathbf{G}])$ (for $i = 1, \dotso, N_0$) and thus \eqref{equation:annule} holds by replacing $\rho_0(g_i)$ by an arbitrary $Z \in \rho_0(\C[\mathbf{G}])$. In particular, since the latter is stable by taking the adjoint, we have that $Y^* \in \rho_0(\C[\mathbf{G}])$ and thus $\Tr(Y^*Y) = 0$, that is, $Y=0$ and $X=0$.
\end{proof}

Define the linear map $A : \rho_0(\C[\mathbf{G}]) \rightarrow \rho(\C[\mathbf{G}])$ by setting 
\[A(\rho_0(g_i)) := \rho(g_i), \quad i = 1, \dotso, N_0,\] 
and extending linearly. Set $N := \dim\left(\rho(\C[\mathbf{G}]) \right)$.

\begin{lemma}\label{lemma:A-injective}
The map $A$ is an isomorphism, $N=N_0$, and $\rho$ is irreducible.
\end{lemma}

\begin{proof}
Injectivity of $A$ follows directly from invertibility of $M$ proved in Lemma \ref{lemma:M-invertible}; to show $A$ is an isomorphism, it suffices to prove that $N = N_0$. For the latter, we use that $\dim \E = \dim \E_0$, so $\dim \rho(\mathbb{C}[\mathbf{G}]) \leq (\dim \E_0)^2$, and so using injectivity of $A$ we get $N = N_0$. The irreducibility of $\rho$ then follows from Burnside's theorem.
\end{proof}

\begin{remark}\rm
Without assuming irreducibility of $\rho_0$ or equality of dimensions $\dim \E = \dim \E_0$, it still holds that $A$ is injective. However, in general there is no reason for $A$ to be an isomorphism since the dimension of $\rho(\C[\mathbf{G}])$ tends to jump when perturbing the representation $\rho_0$ (see Remark \ref{remark:necessary} above). 
\end{remark}

In order to simplify notation, we will sometimes write $\mc{O}(\eps)$ for a quantity bounded by $\leq C\eps$, where $C > 0$ is some uniform constant depending only on $\rho_0$ (and not on $\eps,\rho$). Write $d_{0, j}, d_j:  \C[\mathbf{G}] \to \mathbb{C}$ for the linear maps defined by
\[\rho_0(\alpha) = \sum_{j=1}^N d_{0, j}(\alpha) \rho_0(g_j), \quad \rho(\alpha) = \sum_{j=1}^N d_{j}(\alpha) \rho(g_j), \qquad \alpha \in \C[\mathbf{G}].\] 
We have:

\begin{lemma}
\label{lemma:presque}
There exists a constant $C > 0$ (depending only on $\rho_0$) such that for all $g \in \mathbf{G}'_0$:
\begin{equation}\label{eq:A-ineq}
	\|A(\rho_0(g)) - \rho(g)\| \leq C \eps.
\end{equation}
\end{lemma}

In particular, plugging in $g = \mathbf{1}_{\mathbf{G}}$ into \eqref{eq:A-ineq} we get $A(\mathbf{1}_{\E_1}) = \mathbf{1}_{\E_2} + \mc{O}(\eps)$. Moreover, if $\alpha := \sum_{i = 1}^k c_i h_i$ with $c_i \in \C$, $h_i \in \mathbf{G}'_0$ and $k \in \mathbb{Z}_{> 0}$, we obtain:
\[
	\|A(\rho_0(\alpha)) - \rho(\alpha)\| \leq \left(\max_{i \in \{1, \dotso, k\}} |c_i|\right) \cdot Ck\varepsilon.
\]

\begin{proof}
Let $g \in \mathbf{G}_0'$ and define $D_{(0)}(g) := \big(d_{(0), 1}(g), \dotso, d_{(0), N}(g)\big)^\top \in \mathbb{C}^N$. (By putting the index inside brackets we mean that we may both include it or not include it.) Observe that for $i=1, \dotso,N$, we have:
\begin{equation}
\label{equation:comput}
\begin{split}
%\Tr(\rho_0(g_ig)) &  = \sum_{j = 1}^N d_{0, j}(g) \Tr(\rho_0(g_i g_j)) = \left[M_0 D_0(g)\right]_i \\
%& = \Tr(\rho(g_ig)) + \mc{O}(\eps)   = \left[M D(g)\right]_i + \mc{O}(\eps)  \\
%& = \left[M_0 D(g)\right]_i + \mc{O}\big(\eps+ \|M - M_0\| \|D(g)\|\big),
\mc{O}(\varepsilon) = \Tr(\rho_0(g_ig)) - \Tr(\rho_2(g_ig)) &  = \sum_{j = 1}^N \Big(d_{0, j}(g) \Tr(\rho_0(g_i g_j)) - d_{j}(g) \Tr(\rho(g_i g_j))\Big)\\ 
&= \left[M_0 D_0(g)\right]_i - \left[M D (g)\right]_i\\ 
&= \left[M_0 (D_0(g) - D(g))\right]_i + \mc{O}\big(\|M - M_0\| \|D(g)\|\big),
\end{split}
\end{equation}
where $\|D(g)\|^2 := \sum_{j=1}^N |d_j(g)|^2$, and in the first equality we used \eqref{equation:trace} together with the fact that $g \in \mathbf{G}_0' \subset \mathbf{G}_0$. Next, from \eqref{equation:trace} we get $\|M - M_0\| \leq N \eps$. We claim that there exists a constant $B > 0$ depending only on $\rho_0$, such that:
\begin{equation}
\label{equation:bound-uniform}
	\|D(g)\|^2  \leq B, \quad g \in \mathbf{G}_0'.
\end{equation}
Indeed, we have $\rho_0(g) = \sum_{j} d_{0, j}(g) \rho_0(g_j)$ and $\|\rho_0(g)\|_{\Fr}^2 = \rk(\E_0)$ is constant. Hence, by equivalence of norms in $\E_0$, there exists a constant $B_0 > 0$ (depending only on $\rho_0$) such that $\|D_0(g)\|^2 = \sum_j |d_{0, j}(g)|^2 \leq B_0$ for all $g \in \mathbf{G}$. Therefore, using \eqref{equation:comput}, the invertibility of $M_0$, and \eqref{equation:bound-uniform}, we obtain:
\[
C_1 \|D(g)\| \leq \|M_0 D(g)\| \leq \|M_0 D_0 (g)\| + C_2\eps \left(1+ \|D(g)\|\right) \leq C_3 + \eps C_4 \|D(g)\|
\]
for some constants $C_1, C_2, C_3, C_4 >0$ depending only on $\rho_0$. Taking $\eps < \eps_0$ and $\eps_0 < \frac{C_1}{C_4}$, we get \eqref{equation:bound-uniform}. Again using \eqref{equation:comput}, we thus get $\|M_0(D_0(g) - D(g))\| = \mc{O}(\varepsilon)$, which by invertibility of $M_0$ implies $\|D_0(g) - D(g)\| = \mc{O}(\varepsilon)$. Finally this implies that
\[\|A \rho_0(g) - \rho(g)\| = \left\|\sum_j (d_{0, j}(g) - d_j(g)) \rho(g_j)\right\| \leq \sqrt{N} \|D_0(g) - D(g)\| = \mc{O}(\varepsilon),\]
where in the first estimate we used triangle inequality and that $\|\rho(g_j)\| = 1$ for each $j$, as well as the Cauchy-Schwarz inequality. This completes the proof.
\end{proof}

The previous lemma yields:
\begin{lemma}
\label{lemma:mult}
There exists a constant $C > 0$, depending only on $\rho_1$, such that for all $u,v \in \rho_0(\C[\mathbf{G}])$,
\[
\|A(uv)-A(u)A(v)\| \leq C \eps \|u\| \|v\|
\]
\end{lemma}

\begin{proof}
Assume without of generality that $u$ and $v$ have unit operator norm. Write $u = \sum_{i} u_i \rho_0(g_i)$ and $v = \sum_i v_i \rho_0(g_i)$; by equivalence of norms then $\sum_i u_i^2$ and $\sum_i v_i^2$ are bounded by a constant $C_1 > 0$ which depends only on $\rho_0$. Then:
\[\|A(uv) - A(u)A(v)\| \leq \left\|\sum_{i, j} u_i v_j (A(\rho_0(g_i g_j)) - \rho(g_i g_j))\right\| \leq N C_1 C \varepsilon,\]
where in the last inequality we used the arithmetic-geometric mean inequality and Lemma \ref{lemma:presque} (note that $g_i g_j \in \mathbf{G}_0'$ since $\mathbf{G}_0$ is sufficiently large). This proves the claim.
\end{proof}

Next, we show that $A$ nearly commutes with the adjoint (with a uniform bound on the remainder).

\begin{lemma}
\label{lemma:commute}
There exists a constant $C > 0$, depending only on $\rho_0$, such that for all $u \in  \rho_0(\C[\mathbf{G}])$,
\[
\|A(u^*)-(Au)^*\| \leq C \eps \|u\|.
\]
\end{lemma}

\begin{proof}
Write $u = \sum_{i} u_i \rho_0(g_i)$ for some $u_i \in \mathbb{C}$. As in the previous proofs, we may assume that $u$ has unit operator norm, and consequently $u_i$ are uniformly bounded depending only on $\rho_0$. By Burnside's theorem (see \eqref{equation:burnside}), we can write 
\begin{equation}\label{eq:adjoint-basis}
	\rho_0(g_i)^* = \rho_0\Big(\underbrace{\sum_{j=1}^N d_{0, ij}g_j}_{=:\alpha_i}\Big), \quad d_{0, ij} \in \mathbb{C}, \quad i \in \{1, \dotso, N\}.
\end{equation}
Therefore
\[A(u^*) - (Au)^* = \sum_i \bar{u}_i \big(\rho(\alpha_i) - \rho(g_i)^*\big),\]
and it suffices to show for each $i$ that $\rho(\alpha_i) - \rho(g_i)^* = \mc{O}(\eps)$. Again by equivalence of norms, $d_{0, ij}$ are uniformly bounded and hence
\begin{align*}
	\rho(g_i) \rho(\alpha_i) &= \rho(g_i \alpha_i) = \sum_j d_{0, ij} \rho(g_i g_j)\\
	&= \sum_j d_{0, ij} A\big(\rho_0(g_i g_j)\big) + \mc{O}(\varepsilon) = A \big(\rho_0(g_i) \rho_0(\alpha_i)\big) + \mc{O}(\varepsilon)\\
	&= A(\mathbf{1}_{\E_0}) + \mc{O}(\varepsilon) = \mathbf{1}_{\E} + \mc{O}(\varepsilon).
\end{align*}
Here in the third and sixth equalities we used Lemma \ref{lemma:presque}. Thus, multiplying by $\rho(g_i)^*$ on the left we get $\rho(\alpha_i) = \rho(g_i)^* + \mc{O}(\eps)$ (here we use that $\rho(g_i)$ is unitary), proving the claim and completing the proof of the lemma.
\end{proof}

The following lemma shows that $A$ is nearly unitary.

\begin{lemma}
\label{lemma:norm}
There exists a constant $C > 0$, depending only on $\rho_0$ such that
\[
\|Au\|_{\mathrm{Fr}}^2 \leq (1 + C\eps)\|u\|_{\mathrm{Fr}}^2,\, \forall u \in \rho_0(\mathbb{C}[\mathbf{G}]), \qquad \|A^{-1}v\|_{\mathrm{Fr}}^2 \leq (1+C\eps)\|v\|_{\mathrm{Fr}}^2,\, \forall v \in \rho(\mathbb{C}[\mathbf{G}]).
\]
\end{lemma}

\begin{proof}
Write $u = \sum_i u_i \rho_0(g_i) \in \rho_0(\C[\mathbf{G}])$. Then we have:
\begin{align*}
\|Au\|^2_{\Fr} & = \sum_{i, j} \Tr(\rho(g_i)\rho(g_j)^*)u_i \bar{u}_j  = \sum_{i, j} \Tr(\rho(g_i) A (\rho_0(g_j)^*)) u_i \bar{u}_j + \mc{O}(\varepsilon) \|u\|_{\Fr}^2\\
& = \sum_{i, j, k} d_{0, jk} \Tr(\rho(g_i g_k))u_i \bar{u}_j + \mc{O}(\eps)\|u\|_{\Fr}^2\\ 
&= \sum_{i, j, k}  d_{0, jk} \Tr(\rho_0(g_i) \rho_0(g_k))u_i \bar{u}_j + \mc{O}(\varepsilon) \|u\|_{\Fr}^2 = \|u\|^2_{\Fr}(1+\mc{O}(\eps)).
\end{align*}
Here in the second equality we used Lemma \ref{lemma:commute} (and the facts that $\rho(g_i)$ is unitary as well as the equivalence of norms in $\rho_0(\mathbb{C}[\mathbf{G}])$), in the third equality we used the notation of \eqref{eq:adjoint-basis}, in the fourth we used \eqref{equation:trace} and the fact that $d_{0, jk}$ are uniformly bounded, while in the fifth we used \eqref{eq:adjoint-basis} again. This completes the proof.
\end{proof}

We next show how to go from the nearly unitary isomorphism $A$ on the level of algebras, to a nearly unitary equivalence of representations.

\begin{lemma}\label{lemma:stability-final-lemma}
There exist a constant $C >0$ (depending only on $\rho_0$) and $\mathbf{P} \in \mathrm{U}(\E, \E_0)$ such that for all $g \in \mathbf{G}'_{0}$:
\[
	\|\rho_0(g) - \mathbf{P} \rho(g) \mathbf{P}^* \| \leq C \eps.
\]
\end{lemma}

\begin{proof}
Assume that we have constructed a $\mathbf{P} \in \mathrm{U}(\E, \E_0)$ such that for all $i=1, \dotso,N$:
\begin{equation}
\label{equation:todo}
\|A(\rho_0(g_i)) - \mathbf{P}^* \rho_0(g_i) \mathbf{P}\| \leq C \eps.
\end{equation}
We claim that the result then follows. Indeed, we have, for $g \in \mathbf{G}_0'$:
\begin{align*}
	\rho(g) &= A(\rho_0(g)) + \mc{O}(\eps)  = A \left(\sum_{j=1}^N d_{0, j}(g) \rho_0(g_j) \right) +\mc{O}(\eps)\\
	&= \mathbf{P}^* \sum_{j=1}^N d_{0, j}(g)  \rho_0(g_i) \mathbf{P} + \mc{O}(\eps) = \mathbf{P}^* \rho_0(g) \mathbf{P}  + \mc{O}(\eps),
\end{align*}
where in the first equality we used Lemma \ref{lemma:presque} and in the third equality we used \eqref{equation:todo} and the fact that $|d_{0, j}(g)|$ are uniformly bounded (see \eqref{equation:bound-uniform}).

Let us now prove \eqref{equation:todo}; we will write $r := \dim \E$. Pick a unit vector $\omega \in \E_0$ and denote by $\omega^\flat \in (\E_0)^*$ the linear form defined by $\omega^\flat(\bullet) := \langle \bullet, \omega \rangle_{\E_0}$. We have $\|\omega \otimes \omega^\flat\|_{\Fr} = 1$ and so we get (recall that $\rho_0(\mathbb{C}[\mathbf{G}]) = \End(\E_0)$ in this case by irreducibility of $\rho_0$ and \eqref{equation:burnside}):
\[
	\|A(\omega\otimes\omega^\flat)\|_{\Fr}^2 = 1 +\mc{O}(\eps) \leq r \|A(\omega \otimes \omega^\flat)\|^2,
\]
where we used Lemma \ref{lemma:norm} in the first equality and the elementary estimate $\|\bullet\|_{\Fr} \leq \sqrt{r} \|\bullet\|$ in the second one. Hence, we can find a unit vector $z \in \E$ such that 
\begin{equation}\label{eq:A-lower-bound}
	\|A(\omega \otimes \omega^\flat)z\|^2 > r^{-1} + \mc{O}(\eps).
\end{equation} 
From \eqref{eq:A-lower-bound} we can also derive the following estimate: 
\begin{equation*}
	\langle{z, A(\omega \otimes \omega^\flat)z}\rangle_{\E} = \langle{z, \left[A(\omega \otimes \omega^\flat)\right]^2 z}\rangle_{\E} + \mc{O}(\varepsilon) = \|A(\omega \otimes \omega^\flat)z\|^2_{\E} + \mc{O}(\varepsilon),
\end{equation*}
where in the first equality we used Lemma \ref{lemma:mult} (together with the fact that $\omega \otimes \omega^\flat$ is a projector), in the second one we used Lemma \ref{lemma:commute}, while in the last one we applied \eqref{eq:A-lower-bound}. Note that $\langle{z, A(\omega \otimes \omega^\flat)z}\rangle_{\E}$ might have non-zero angular part, which is $\mc{O}(\varepsilon)$ close to $1$ by the previous estimate and so we get 
\begin{equation}\label{eq:A-lower-bound-2}
	|\langle{z, A(\omega \otimes \omega^\flat)z}\rangle_{\E}| >  r^{-1} + \mc{O}(\eps).
\end{equation}
Define $\mathbf{Q} : \E_0 \to \E$ by 
\[\mathbf{Q}x := A(x \otimes \omega^\flat)z \in \E.\]
Next, we show that $\mathbf{Q}$ is nearly unitary (up to constant), that is, we claim that
\begin{equation}\label{eq:Q-nearly-unitary} 
	\|\mathbf{Q}^*\mathbf{Q} - \langle{z, A(\omega \otimes \omega^\flat) z}\rangle_{\E_0} \mathbf{1}_{\E_0}\| = \mc{O}(\varepsilon).
\end{equation}
(Later we will divide $\mathbf{Q}$ by the constant factor to obtain something nearly unitary.) In particular, it follows that $\|\mathbf{Q}\|_{\E_0 \to \E} = \mc{O}(1)$ is uniformly bounded. To see this, take arbitrary $x, y \in \E_0$. Then
\begin{align*}
	\langle{\mathbf{Q}x, \mathbf{Q}y}\rangle_{\E} &= \langle{A(x \otimes \omega^\flat) z, A(y \otimes \omega^\flat)z}\rangle_{\E} = \left\langle{z, \left[A(x \otimes \omega^\flat)\right]^* A(y \otimes \omega^\flat) z}\right\rangle_{\E_0}\\
	&= \langle{z, A(\omega \otimes x^\flat) A(y \otimes \omega^\flat) z}\rangle_{\E_0} + \mc{O}(\varepsilon) \|x\|_{\E_0} \|y\|_{\E_0}\\
	&= \langle{x, y}\rangle_{\E_0} \langle{z, A(\omega \otimes \omega^\flat) z}\rangle_{\E_0} + \mc{O}(\varepsilon) \|x\|_{\E_0} \|y\|_{\E_0},
\end{align*}
where in the second line we used Lemma \ref{lemma:commute} (and that the adjoint of $x \otimes \omega^\flat$ is $\omega \otimes x^\flat$) and in the final line we used Lemma \ref{lemma:mult} (as well as $\omega \otimes x^{\flat} \circ y \otimes \omega^\flat = \langle{y, x}\rangle_{\E_0} \omega \otimes \omega^\flat$). The claim follows by taking the maximum over all unit $y = x \in \E_0$.

For $u \in \rho_0(\C[\mathbf{G}])$ and $x \in \E_0$, using Lemma \ref{lemma:mult}, we get:
\begin{equation}\label{eq:Q-maps}
	\mathbf{Q}(ux)  = A(ux \otimes \omega^\flat)z = A(u)A(x \otimes \omega^\flat)z + \mc{O}(\eps \|u\|_{\Fr} \|x\|_{\E_0}) = A(u) \mathbf{Q}x +  \mc{O}(\eps \|u\|_{\Fr} \|x\|_{\E_0}).
\end{equation}

It follows from \eqref{eq:Q-nearly-unitary} and \eqref{eq:A-lower-bound-2} that for $\varepsilon \in [0, \varepsilon_0)$ small enough (depending only on $\rho_0$), we have that $\mathbf{Q}: \E_0 \to \E$ is an isomorphism. At this stage, we \emph{re-define} 
\[
	\mathbf{Q} := \frac{\mathbf{Q}}{\sqrt{|\langle{z, A(\omega \otimes \omega^\flat) z}\rangle_{\E_0}}|},\]
where this is well-defined thanks to \eqref{eq:A-lower-bound-2} and satisfies thanks to \eqref{eq:Q-maps} and \eqref{eq:Q-nearly-unitary} that:
\begin{equation}\label{eq:Q-re-define}
	\|\mathbf{Q}^*\mathbf{Q} - \mathbf{1}_{\E_0}\| = \mc{O}(\varepsilon), \quad \mathbf{Q}(ux) = A(u) \mathbf{Q}x +  \mc{O}(\eps) \|u\|_{\Fr} \|x\|_{\E_0},\quad  \forall x \in \E_0, u \in \rho_0(\mathbb{C}[\mathbf{G}]).
\end{equation}
Define $\mathbf{P} \in \operatorname{U}(\E, \E_0)$ setting its inverse to be the unitarisation of $\mathbf{Q}$, that is, $\mathbf{P}^{-1} := \mathbf{Q}(\mathbf{Q}^* \mathbf{Q})^{-1/2}$. Then for each $i = 1, \dotso, N$
\begin{align*}
	\|\mathbf{P} A(\rho_0(g_i)) \mathbf{P}^* - \rho(g_i)\| &= \|(\mathbf{Q}^* \mathbf{Q})^{\frac{1}{2}} \mathbf{Q}^{-1} A(\rho_0(g_i)) \mathbf{Q} (\mathbf{Q}^*\mathbf{Q})^{-\frac{1}{2}} - \rho(g_i)\|\\
	&= \|(\mathbf{Q}^* \mathbf{Q})^{\frac{1}{2}} \rho_0(g_i) (\mathbf{Q}^*\mathbf{Q})^{-\frac{1}{2}} - \rho(g_i)\| + \mc{O}(\varepsilon) = \mc{O}(\varepsilon),
\end{align*}
where in the second and third equalities we used the second and the first identities in \eqref{eq:Q-re-define}, respectively. This proves \eqref{equation:todo} and completes the proof.
\end{proof}

\begin{remark}\rm
	In the previous proof we were inspired by the approach of \v{S}emrl (see \cite{Semrl-06}) to the Noether-Skolem theorem.
\end{remark}

We now simply record that the above lemmas prove the claimed stability of representations result.

\begin{proof}[Proof of Lemma \ref{lemma:stability-representation}]
	The proof follows immediately from Lemma \ref{lemma:stability-final-lemma}, which in turn depends on Lemmas \ref{lemma:M-invertible}, \ref{lemma:A-injective}, \ref{lemma:presque}, \ref{lemma:mult}, \ref{lemma:commute}, and \ref{lemma:norm}.
\end{proof}

\section{Approximate non-Abelian Liv\v{s}ic cocycle Theorem}\label{sec:approximate-livsic}

In this section, we prove the approximate non-Abelian Liv\v{s}ic cocycle Theorem \ref{theorem:approximate-livsic}. Throughout the section, $\M$ is a smooth closed manifold equipped with a transitive Anosov flow $(\varphi_t)_{t \in \R}$, generated by a smooth vector field $X$. We divide the proof into four steps.

\subsection{Existence of a good orbit}

\label{sssection:good}

We introduce the notion of \emph{transversal separation}:

\begin{definition}
Let $\delta > 0$. We say that a subset $S \subset \M$ is $\delta$-\emph{transversally separated} if for all $x \in S$, we have:
\[
	\cup_{y \in W^s_\delta(x)} W^u_\delta(y)  \cap S = \left\{ x \right\}.
\]
\end{definition}

The following lemma is one of the key ingredients in the proof of Theorem \ref{theorem:approximate-livsic}. A similar statement was proved in \cite[Lemma 3.4]{Gouezel-Lefeuvre-19} for the purpose of proving a stability estimate for the $X$-ray transform. For brevity, in the proof we will often only refer to the mentioned lemma and to the Shadowing Lemma (see \cite[Theorem 6.2.4]{Fisher-Hasselblatt-19}).

\begin{lemma}\label{lemma:separation+dense}
There exist $\beta_s, \beta_d, \eps_0 > 0$ such that the following holds. Assume $\varepsilon < \varepsilon_0$. There exists a homoclinic orbit $\gamma_\eps \in \mc{H}$ of length $T(\gamma_\eps) \leq \eps^{-1/2}$ whose trunk $\mc{T}_\eps$ is $\eps^{\beta_s}$-transversally separated and $\eps^{\beta_d}$-dense in $\M$. 
\end{lemma}

Let us remark that the exponents can be made explicit: although they are not optimal in \cite[Lemma 3.4]{Gouezel-Lefeuvre-19}, one has $\beta_d = \tfrac{1}{3 \dim(\M)}$ and $\beta_s$ depends on the Lyapunov exponents of the flow (in a less explicit fashion).

\begin{proof}
Following the proof of \cite[Lemma 3.4]{Gouezel-Lefeuvre-19} we let $p_ 1 := x_\star$ and set the periodic point $p_2 \neq p_1$ to be arbitrary. The proof is exactly the same as in \cite[Lemma 3.4]{Gouezel-Lefeuvre-19} up to the sentence ``we can glue the sequence ..." to produce points $x_i \in \M$ and orbit segments $\gamma_{x_i}$ (starting and ending within $2\eps$ of $x_\star$), where it has to be adapted in the following way: one glues instead the sequence
\[
\dotsc, \gamma_\star, \gamma_\star, \gamma_{x_1}, \gamma_{x_2}, \dotsc, \gamma_{x_N}, \gamma_\star, \gamma_\star, \dotsc
\]
using the Shadowing Lemma in order to produce a homoclinic orbit $\gamma_\varepsilon$. The orbit segments $\gamma_{x_i}$ make excursions to some fixed points $z_0$ outside a neighbourhood of $\gamma_{\star}$. The homoclinic orbit $\gamma_\eps$ is $\mc{O}(\eps)$-shadowing the previous concatenation of segments and so (taking $\varepsilon_0 > 0$ small enough) the trunk $\mc{T}_\eps$ of $\gamma_\eps$ is contained in a segment shadowing the concatenation $\gamma_{x_1}, \gamma_{x_2}, \dotsc, \gamma_{x_N}$. More precisely, since $\mathbf{D}_{u/s}$ are at a positive finite distance from $x_\star$ the trunk $\mc{T}_\eps$ could only possibly miss small segments of $\gamma_{x_1}$ or $\gamma_{x_N}$ of size $\mc{O}(|\log \eps|)$ corresponding, in the construction and notation of \cite{Gouezel-Lefeuvre-19}, to orbit pieces $\gamma'_\pm$ associated to the points $x_1$ or $x_N$. It is straightforward to check that the same arguments as in \cite[Lemma 3.4]{Gouezel-Lefeuvre-19} apply to show that $\mc{T}_\eps$ is $\eps^{\beta_s}$-separated, $\eps^{\beta_d}$-dense, and of length $\leq \eps^{-1/2}$. 
\end{proof}

\subsection{Application of representation theory}
We assume the notation of Theorem \ref{theorem:approximate-livsic}. The length function of Definition \ref{definition:length} can be extended naturally to $\G$ by setting:
\[
T(\gamma_1^{k_1} \dotsm \gamma_n^{k_n}) := k_1 T(\gamma_1) + \dotsb + k_n T(\gamma_n).
\]
Let $g_1, \dotsc, g_{N_0} \in \G$ such that $\mathrm{span}\left\{\rho_0(g_i) ~|~ i=1, \dotsc, N_0 \right\} = \rho_0(\C[\G])$ and define $N_0 = \dim(\rho_0(\C[G]))$. Define $\mc{A}_\eps \subset \G$ as:
\[
\mc{A}_\eps := \left\{ \gamma_1^{k_1} \dotsm \gamma_n^{k_n} \in \G \mid n \leq |\log(\eps)|, \, \forall i = 1, \dotsc, n, \, T(\gamma_i) \leq \eps^{-1/2},\, k_i \leq |\log(\eps)|\right\}.
\]
We will also assume that $\eps < \eps_0$, where $\eps_0$ is chosen small enough (depending on $b_0$) so that for $i=1, \dotsc ,N_1$, $g_i \in \mc{A}_\eps$. Also, write $r := \rk(\E)$ and $r_0 := \rk(\E_0)$.

\begin{lemma}\label{lemma:from-trace-to-parry}
Under the assumptions of Theorem \ref{theorem:approximate-livsic} (we do not need irreducibility of $\rho_0$ at this stage), for each $\tau_0 < \frac{1}{2}$ there exists $\eps_0 > 0$ such that the following holds for all $0 < \varepsilon < \varepsilon_0$
\begin{equation}\label{eq:eps/6}
	|\chi_{\rho_0}(g)-\chi_{\rho}(g)| \leq \eps^{\tau_0}, \quad \forall g \in \mc{A}_{\eps/6},
\end{equation}
and also $\rk(\E) = \rk(\E_0)$.
\end{lemma}
Note that we ask the character condition in \eqref{eq:eps/6} to hold for $g \in \mc{A}_{\eps/6}$ (and not, say, for $g \in \mc{A}_{\eps}$). We do this for two reasons: firstly, to apply Lemma \ref{lemma:stability-representation}, we need some space to deal with the fact that its conclusion holds on a slightly smaller set $\mathbf{G}_0'$, and secondly, because at some point below (see \eqref{eq:new-homoclinic}) we might need to slightly leave the set $\mc{A}_{\eps} \subset \mc{A}_{\eps/2}$ and we need some extra space for this as well.

We can thus apply Lemma \ref{lemma:stability-representation}, which provides a unitary isomorphism $\mathbf{P}_\star \in \mathrm{U}(\E,\E_0)(x_\star)$ over the point $x_\star$ and a constant $C > 0$ (depending only on $b_0$) such that: 
\begin{equation}
\label{equation:equality-approx}
	\|\rho_0(g) - \mathbf{P}_\star \rho(g)\mathbf{P}^{-1}_\star\| \leq C \eps^{\tau_0}, \quad \forall g \in \mc{A}_{\eps/2}.
\end{equation}
Here we used the inclusion $\mc{A}_{\eps/2} \subset \mc{A}_{\eps/6}'$, where we recall the prime notation was introduced in \eqref{eq:prime-notation} (we use here that $\log 3 > 1$).

\begin{proof}
We will follow the proof of \cite[Proposition 3.19]{Cekic-Lefeuvre-22}. Let us first deal with a single homoclinic orbit $\gamma \in \mc{H}$. We will use the notation from \S \ref{sssec:parry-monoid-anosov} and \S \ref{ssec:parry-rep}. The points $x_{s/u}(\gamma; N)$ are $\mc{O}(e^{- N T_\star \mu })$-close to $x_\star$, hence $\mc{O}(e^{- n T_\star \mu })$-close to $x_\star$ for $N \geq n$; therefore the segment $[x_{u}(\gamma; N), x_s(\gamma; n)]$ is shadowed by a genuine periodic point $y$ of period 
\begin{equation}\label{eq:shadow-period}
	T_{y} = T(\gamma) + (n + N) T_\star + \mc{O}(e^{-n T_\star \mu}).
\end{equation}
The argument from \cite[Proposition 3.19]{Cekic-Lefeuvre-22} ensures that as soon as $N$ is large enough so that 
\begin{equation}\label{eq:shadow-primitive-condition}
	NT_\star > T(\gamma) + 2nT_\star +  \mc{O}(e^{-nT_\star \mu}) \geq (N - 1)T_\star,
\end{equation}
the closed orbit $\gamma$ is primitive (more precisely, under the condition \eqref{eq:shadow-primitive-condition} the segment of length $(N - n) T_\star$ on the left is larger than its complement). Next, by the shadowing property and by the Ambrose-Singer formula \cite[Lemma 3.14, Items (1) and (2)]{Cekic-Lefeuvre-22} we have
\begin{equation}\label{eq:shadow-parry}
	\rho(\gamma; N, n) = C_{y \rightarrow x_*} C(y, T_{y}) C_{x_* \rightarrow y} + \mc{O}(e^{-n T_\star \mu}).
\end{equation}
Note that here we implicitly used Lemma \ref{lemma:uniform-bound-curvature} to guarantee the uniformity of the remainder. Combining the estimates above we conclude that:
\begin{align}\label{eq:shadow-chain}
\begin{split}
\Tr(\rho(\gamma)) & = \Tr(\rho(\gamma; N, n)) + \mc{O}(e^{-nT_\star \mu}) \\
& = \Tr(C(y, T_y)) + \mc{O}(e^{-nT_\star \mu})\\
& = \Tr(C_0(y, T_y) + \mc{O}(e^{-nT_\star \mu}) + \mc{O}(\eps T_{y})\\
& = \Tr(\rho_0(\gamma; N, n)) + \mc{O}(e^{-nT_\star \mu}) + \mc{O}(\eps T_{y})\\ 
&= \Tr(\rho_0(\gamma)) + \mc{O}(e^{-nT_\star \mu}) + \mc{O}(\eps T_{y}).
\end{split}
\end{align}
Here we used \eqref{eq:approximate-parry} in the first line, \eqref{eq:shadow-parry} in the second one, while in the third line we used the assumption of Theorem \ref{theorem:approximate-livsic}, in the fourth we used the version of \eqref{eq:shadow-parry} for $C_0$ and $\X_0$, while in the fifth line we used \eqref{eq:approximate-parry} again. Taking $n = \lceil{\varepsilon^{-1/2}}\rceil$ in \eqref{eq:shadow-chain}, we conclude that the remainder is of size $\mc{O}(\varepsilon^{1/2})$ (here we use that $T(\gamma) \leq \sqrt{6}\varepsilon^{-1/2}$), which proves the claim in this case (note that here the excess of powers of $\varepsilon$ is used to swallow the constant in front).

For the general case of a concatenation $g \in \mc{A}_{\varepsilon/6}$ of several $\gamma \in \mc{H}$ we argue analogously as in \eqref{eq:shadow-chain} and omit the argument, which produces the estimate
\[|\chi_{\rho}(g) - \chi_{\rho}(g)| = \mc{O}(|\log \varepsilon|^2 \varepsilon^{1/2}).\]
This completes the proof of the first claim.

For the second claim, fix an arbitrary $\gamma \in \mc{A}_{\eps/6}$. We would like to use a quantitative version of the fact that $\rho(\gamma)^{k_n} \to \id_{\E(x_\star)}$ as $n \to \infty$ for some sequence $k_n \to \infty$. Using that $\rho(\gamma)$ and $\rho_0(\gamma)$ are unitary, hence diagonalisable, using the pigeon-hole principle we get that for each $N \in \mathbb{Z}_{>0}$ with $N > N_0$ (where $N_0$ is a uniform constant) there exists $j = j(N) \in \{1, \dotsc, N^{r_0 + r}\}$ such that
\begin{equation}\label{eq:pigeonhole}
	\|\rho(\gamma)^j - \id_{\E(x_\star)}\| \leq \frac{3\pi}{N}, \quad \|\rho_0(\gamma)^j - \id_{\E_0(x_\star)}\| \leq \frac{3\pi}{N}.
\end{equation}
If we assume that $|\log(\varepsilon)| > N^{r_0 + r}$ (valid for $\eps_0 > 0$ small enough), then
\[
	|r - r_0| \leq \frac{3\pi (r_0 + R_0)}{N} + |\Tr (\rho(\gamma)^j) - \Tr(\rho_0(\gamma)^j)| \leq \frac{3\pi (r_0 + R_0)}{N} + \varepsilon^{\tau_0},
\]
where in the first estimate we used triangle inequality and \eqref{eq:pigeonhole}, as well as that $r \leq R_0$, and in the second estimate we used the first part of the lemma. It follows that as soon as $\varepsilon$ is so small that $\varepsilon^{\tau_0} < \frac{1}{2}$, $N$ is so large that $N > 6\pi (r_0 + R_0)$, and in turn $\varepsilon$ small enough so that $|\log(\varepsilon)| > N^{r_0 + R_0}$, we get that $|r - r_0| < 1$ and so $r \equiv r_0$, completing the proof.
\end{proof}

Having defined $\mathbf{P}_\star$ in \eqref{equation:equality-approx}, for $\gamma \in \mc{H}$ and $x \in \gamma$ we set:
\begin{equation}\label{eq:p_-def}
p_-(x) := \mathbf{H}^c_{x_u(\gamma) \to x}\mathbf{H}^u_{x_\star \to x_u(\gamma)} \mathbf{P}_\star, \quad 
p_+(x) := \mathbf{H}^c_{x_s(\gamma) \to x}\mathbf{H}^s_{x_\star \to x_s(\gamma)} \mathbf{P}_\star,
\end{equation}
where $\mathbf{H}^{s/u/c}$ denotes the stable/unstable/central holonomy of the linear unitary extension $(\Hom(\E, \E_0), \Hom(h, h_0), \X^{\Hom})$ defined in \eqref{eq:hom-operator}; its Parry's representation $\rho^{\Hom}$ then acts in the following way
\[
	\rho^{\Hom}(g)p = \rho_0(g)p\rho^{-1}(g), \quad p \in \mathrm{Hom}(\E, \E_0)(x_\star), \quad g \in \mathbf{G}.
\]
This is because the parallel transport $\mathbf{P}(\bullet, \bullet)$ of $\X^{\Hom}$ along the flowlines acts as 
\[
	\mathbf{P}(x, t) p(x) = P_{0}(x, t) p (x) P(x, -t) \in \Hom(\E, \E_0)(\varphi_t x),
\] 
where $p \in \Hom(\E, \E_0)(x)$; here $P_{0}(\bullet, \bullet)$ and $P(\bullet, \bullet)$ denote parallel transports of $\X_0$ and $\X$ along the flowlines, respectively (the notation is similar to \S \ref{ssec:parry-rep}). In what follows, for all $x, y \in \M$ sufficiently close we will denote by $\mathbf{P}_{x \to y}$ the parallel transport with respect to the homomorphism connection on $\Hom(\E, \E_0)$ (where we recall that the connections on $\E$ and $\E_0$ are the ones given by Lemma \ref{lemma:uniform-bound-curvature}) along the unit speed geodesic between $x$ and $y$. The following holds similarly to \cite[Lemma 3.23]{Cekic-Lefeuvre-22}:

\begin{lemma}
\label{lemma:small-trunk}
There exists $C > 0$ (depending only on $b_0$) such that the following holds. For all $\gamma \in \mc{H}$ such that $T(\gamma) \leq \sqrt{2} \eps^{-1/2}$, for all $x \in \gamma$, $\|p_-(x)-p_+(x)\| \leq C \eps^{\tau_0}$.
\end{lemma}

\begin{proof}
The stable/unstable/centre holonomies are unitary, so we have by \eqref{equation:equality-approx}:
\begin{align*}
\|p_-(x)-p_+(x)\| & = \|\mathbf{P}_\star - \mathbf{H}^s_{x_s(\gamma) \to x_\star}\mathbf{H}^c_{x_u(\gamma) \to x_s(\gamma)}\mathbf{H}^u_{x_\star \to x_u(\gamma)}\mathbf{P}_\star \| \\
& = \|\mathbf{P}_\star - \rho^{\Hom}(g)\mathbf{P}_\star\| \\
& = \|\mathbf{P}_\star - \rho_0(g)\mathbf{P}_\star \rho^{-1}(g)\|  = \|\rho_0(g)-\mathbf{P}_\star\rho(g)\mathbf{P}_\star^{-1}\| = \mc{O}(\eps^{\tau_0}).
\end{align*}
\end{proof}

We now consider $\gamma_\eps \in \mc{H}$, the homoclinic orbit constructed in Lemma \ref{lemma:separation+dense} and denote by $\mc{T}_\eps$ its trunk. The following statement is similar to \cite[Lemmas 3.22 and 3.24]{Cekic-Lefeuvre-22}.

\begin{lemma}\label{lemma:uniform-bound-p_-}
Let $\alpha := \tau_0/\beta_s < 1$. There exists a constant $C > 0$ such that: 
\[\|p_-|_{\mc{T}_\eps}\|_{C^\alpha} \leq C.\]
\end{lemma}

\begin{proof}
Consider $x,y \in \mc{T}_\eps$ with $d(x, y) < \delta$ where $\delta > 0$ is sufficiently small, such that $y$ is in the future of $x$, and let $z := [x,y] \in W^{cu}_{\delta}(x) \cap W^{s}_{\delta}(y)$ (see \eqref{eq:bowen-bracket} for the definition of $[\bullet, \bullet]$). Define $\tau$ such that $\varphi_\tau(z) \in W^u_{\loc}(x)$ and note that $|\tau| \leq Cd(x,y)$ for some uniform constant $C > 0$. The orbit $\gamma$ of the point $z$ is also a homoclinic orbit; if $x$ and $y$ are sufficiently close on the local central leaf, then $z = y$ and $\gamma = \gamma_{\eps}$. By construction, the length of its trunk satisfies 
\begin{equation}\label{eq:new-homoclinic}
	T(\gamma) \leq T(\gamma_\eps) + \mc{O}(1) \leq \eps^{-1/2} + \mc{O}(1) \leq \sqrt{2}\varepsilon^{-1/2},
\end{equation}
and so we may apply Lemma \ref{lemma:uniform-bound-p_-} to $\gamma$. Indeed, the trunk gets smaller by the length of the segment between $x$ and $y$ plus $\mc{O}(\delta)$; the extra $\mc{O}(1)$ is to account for the possible addition of $2T_\star$ coming from the fact that the endpoints $x_s(\gamma)$ and $x_u(\gamma)$ of the trunk of $\gamma$ need to make another loop around $\gamma_\star$ to hit the fundamental domain. We have:
\begin{align}\label{eq:estimate}
\begin{split}
& \|p_-(x) - \mathbf{P}_{y \rightarrow x}p_-(y)\| \\
& \leq \|p_-(x) - \mathbf{P}_{z \rightarrow x}p_-(z)\| + \|p_-(z) - p_+(z)\| + \|\mathbf{P}_{z \rightarrow x}p_+(z) - \mathbf{P}_{y \rightarrow x}p_+(y)\|\\ 
&\hspace{325pt} + \|p_+(y) - p_-(y)\| \\
& \leq \|p_-(x) - \mathbf{P}_{z \rightarrow x}p_-(z)\| + \|\mathbf{P}_{x \rightarrow y}\mathbf{P}_{z \rightarrow x}p_+(z) - p_+(y)\| + \mc{O}(\varepsilon^{\tau_0})\\
& \leq  \|p_-(x) - \mathbf{P}_{\varphi_{\tau}(z) \rightarrow x} p_-(\varphi_{\tau}(z))\| + \|\mathbf{P}_{\varphi_{\tau}(z) \rightarrow x} p_-(\varphi_{\tau}(z)) - \mathbf{P}_{z \rightarrow x}p_-(z)\|\\ 
&\hspace{150pt}+ \|\mathbf{P}_{x \rightarrow y} \mathbf{P}_{z \rightarrow x}p_+(z) - p_+(y)\| + \mc{O}(d(x,y)^{\tau_0/\beta_s}),
\end{split}
\end{align}
where we used the triangle inequality in the first and third inequalities, Lemma \ref{lemma:small-trunk} in the second one, and also Lemma \ref{lemma:separation+dense} the third one. More precisely, in the last estimate we used that there exists $C > 0$ uniform such that $\varepsilon^{\beta_s} \leq C d(x, y)$ (otherwise we would contradict the transversal separation of $\mc{T}_\varepsilon$). Next, the first, second, and third terms of the last line in \eqref{eq:estimate} are bounded by $\mc{O}(d(x, y))$ by the Ambrose-Singer formula (see \cite[Lemma 3.14, Item (1) and (2)]{Cekic-Lefeuvre-22}); alternatively one say that the same estimate follows directly from the proof of \cite[Lemmas 3.22 and 3.24]{Cekic-Lefeuvre-22}. This completes the proof.
\end{proof}

\subsection{Extension of the conjugacy}

The third step is to extend $p_- \in C^{\infty}(\mc{T}_\eps, \Hom(\E, \E_0))$ from $\mc{T}_\eps$ to $\M$. For that, we consider a collection of transverse sections $(\Sigma_j)_{j=1}^N$ and open neighbourhoods $(U_j)_{j = 1}^N$ covering $\M$ such that
\begin{equation}\label{eq:sigma_jdef}
	\Sigma_j := \cup_{y \in W^s_{\eta_1}(x_j)} W^u_{\eta_1}(y), \quad U_j := \cup_{-\eta_2 < t < \eta_2} \varphi_t(\Sigma_j),
\end{equation}
for some $\eta_1, \eta_2 > 0$ and $x_j \in \M$. Taking $\eta_1$ and $\eta_2$ small enough depending on the metric, we may assume that $U_j \subset B(x_j, \imath/2)$, where $\imath$ is the injectivity radius of $g$. We consider a partition of unity $\sum_{j=1}^N \chi_j = \mathbf{1}$ subordinated to this cover (i.e. $\mathrm{supp}(\chi_j) \subset U_j$). In what follows, we extend $p_-$ in three steps: first to $\Sigma_j$ for each $j$, from $\Sigma_j$ to $U_j$, and finally using the partition of unity to $\M$.

\begin{lemma}\label{lemma:p_j-holder-1}
	There exists a uniform $C > 0$ such that the following holds for $\eta_1$ small enough. There exists an extension $p_j \in C^\alpha(\Sigma_j, \Hom(\E, \E_0))$ of $p_-$, i.e. such that $p_-(y) = p_j(y)$ for any $y \in \mc{T}_\eps \cap \Sigma_j$, and we have for $j = 1, \dotsc, N$:
	\begin{align*}
		\|p_j\|_{C^\alpha(\Sigma_j, \mathrm{Hom}(\E, \E_0))} \leq C\|p_-\|_{C^\alpha(\mc{T}_\eps \cap \Sigma_j, \mathrm{Hom}(\E, \E_0))}.
	\end{align*}
\end{lemma}

\begin{proof}
	Fix some $1\leq j \leq N$ and recall that $\rk(\E) = \rk(\E_0) = r$ by Lemma \ref{lemma:from-trace-to-parry}. On each $\Sigma_j$, we extend $p_-$ to $p_j$ by the following procedure. Let $f_{j1}, \dotso, f_{jr^2} \in \mathrm{Hom}({\E}, {\E_0})(x_j)$ be an arbitrary orthonormal basis. Then for $y \in \mc{T}_\eps \cap U_j$, we may write
\[
	\mathbf{P}_{y \to x_j} p_{-}(y) = \sum_{k = 1}^{r^2} a_{jk}(y) \cdot f_{jk}, \quad a_{jk}(y) \in \mathbb{C}.
\]
We may estimate the H\"older norm of $a_{jk}$ as follows. First we estimate the homogeneous part of the H\"older norm; for $x \neq y \in \mc{T}_\eps \cap U_j$, we have
\[
	\sum_{k = 1}^{r^2} \frac{|a_{jk}(y) - a_{jk}(x)|^2}{d(x, y)^{2\alpha}} = \frac{\|\mathbf{P}_{y \to x_j} p_-(y) - \mathbf{P}_{x \to x_j} p_-(x)\|^2}{d(x, y)^{2\alpha}} \leq \mc{O}(d(x, y)^{2(1 - \alpha)}) + \|p_-|_{\mc{T}_\eps}\|_{C^\alpha}^2,
\] 
where in the estimate we used that $\mathbf{P}_{x_j \to x} \mathbf{P}_{y \to x_j} = \mathbf{P}_{y \to x} + \mc{O}(d(x, y))$ by the Ambrose-Singer formula, and therefore since $\|p_-|_{\mc{T}_\eps}\|_{C^0} = 1$. For the non-homogeneous part of the norm observe that $\sum_k |a_{jk}(y)|^2 = \|p_{-}(y)\|^2 = 1$, so we obtain
\begin{equation}\label{eq:a_jk-bound}
	\|a_{jk}|_{\mc{T}_\eps}\|_{C^\alpha} \leq \mc{O}(d(x, y)^{1 - \alpha}) + \|p_-|_{\mc{T}_\eps}\|_{C^\alpha} + \|p_-|_{\mc{T}_{\eps}}\|_{C^0} \leq 2\|p_{-}|_{\mc{T}_\eps}\|_{C^\alpha}, \quad \forall k = 1, \dotsc, r^2,
\end{equation}
for $\eta_1$ small enough, and for $x$ and $y$ close enough (in a uniform way), using that $\|p_-|_{\mc{T}_\eps}\|_{C^0} = 1$ to absorb the term $\mc{O}(d(x, y)^{1 - \alpha})$. Next, we define $b_{jk} \in C^\alpha(\Sigma_j)$ by setting, for $x \in \Sigma_j$:
\begin{align*}
	\re(b_{jk})(x) &:= \min_{y \in \mc{T}_\eps \cap \Sigma_j} \big(\re (a_{jk}(y)) + 2 d(x, y)^\alpha \|p_-\|_{C^\alpha(\mc{T}_\eps \cap \Sigma_j, \mathrm{Hom}(\E, \E_0))}\big),\\
	\im(b_{jk})(x) &:= \min_{y \in \mc{T}_\eps \cap \Sigma_j} \big(\im (a_{jk}(y)) + 2 d(x, y)^\alpha \|p_-\|_{C^\alpha(\mc{T}_\eps \cap \Sigma_j, \mathrm{Hom}(\E, \E_0))}\big).
\end{align*}
By construction, the functions $\re(b_{jk})$ and $\im(b_{jk})$ extend respectively $\re(a_{jk})$ and $\im(a_{jk})$ in the sense that they coincide in restriction to $\mc{T}_\eps \cap \Sigma_j$. That $b_{jk}$ are H\"older regular follows from the elementary fact that $x \mapsto d(x, y)^{\alpha}$ is H\"older regular and that
\[\|\min(q_1, q_2)\|_{C^\alpha(\M)} \leq \max(\|q_1\|_{C^\alpha(\M)}, \|q_2\|_{C^\alpha(\M)}), \quad \forall q_1, q_2 \in C^\alpha(\M, \mathbb{R}).\] 
Indeed, this is applied one less times than the cardinality of the finite set $\mc{T}_\eps \cap \Sigma_j$. From this and \eqref{eq:a_jk-bound} it also follows that 
\begin{equation}\label{eq:b_jk-bound}
	\|b_{jk}\|_{C^\alpha} \leq 8\|p_-|_{\mc{T}_\eps}\|_{C^\alpha}, \quad \forall k = 1, \dotsc, r^2.
\end{equation}
Finally, define:
\begin{equation}\label{eq:def-p_j}
	p_j(x) := \mathbf{P}_{x_j \to x} \left(\sum_{k = 1}^{r^2} b_{jk}(x) f_{jk}\right), \quad x \in \Sigma_j.
\end{equation}

It remains to give a H\"older bound on $p_j$; this follows similarly to the bound for $a_{jk}$ above. Indeed, for $x \neq y \in \Sigma_j$ we have
\[\frac{\|\mathbf{P}_{y \to x} p_j(y) - p_j(x)\|}{d(x, y)^{\alpha}} = \frac{\|\sum_{k} (b_{jk}(y) - b_{jk}(x)) f_{jk}\|}{d(x, y)^{\alpha}} + \mc{O}(d(x, y)^{1 - \alpha}) \leq 9r^{2} \|p_-|_{\mc{T}_\eps}\|_{C^\alpha},\]
for $\eta_1$ small enough, where we used that $\mathbf{P}_{x \to x_j} \mathbf{P}_{y \to x} \mathbf{P}_{x_j \to x} = \id_{\Hom(\E, \E_0)(x)} + \mc{O}(d(x, y))$ in the first equality, which follows from the Ambrose-Singer formula, and we used \eqref{eq:b_jk-bound} in the inequality. This shows the required estimate and completes the proof.
\end{proof}

The following lemma extends $p_j$ constructed in the preceding lemma to $U_j$ and further to $\M$, and concludes the discussion in this section.

\begin{lemma}\label{lemma:p_j-holder-2}
	For $\eta_2 > 0$ small enough, for each $j = 1, \dotsc, N$, there exists an extension $p_j'$ of $p_j|_{\Sigma_j}$ constructed in the previous lemma, such that for some uniform $C > 0$:
	\begin{equation}
	\label{eq:p_j-holder-2}
		\|p_j'\|_{C^\alpha(U_j, \Hom(\E, \E_0))} \leq C\|p_j\|_{C^\alpha(\Sigma_j, \mathrm{Hom}(\E,\E_0))}.
	\end{equation}
	Let us set $\widetilde{p} := \sum_{j = 1}^N \chi_j p_j' \in C^{\alpha}(\M, \Hom(\E, \E_0))$. Then $\widetilde{p}|_{\mc{T}_\eps} = p_-$ and we have
	\begin{equation}
	\label{eq:p_j-holder-3}
		\|\widetilde{p}\|_{C^\alpha} = \mc{O}(1), \quad \|\X^{\Hom} \widetilde{p}\|_{C^\alpha} = \mc{O}(1), \quad \|\X^{\Hom} \widetilde{p}\|_{C^0} \leq \mc{O}(1) \varepsilon^{\beta_d \alpha}.
	\end{equation}	
	Moreover, for $\varepsilon$ small enough we may assume that $\|\widetilde{p}\|_{C^0(\M, \Hom(\E, \E_0))} \leq 2$ and that $\widetilde{p}$ is invertible at each point. In particular, $\E$ and $\E_0$ are isomorphic as vector bundles.
\end{lemma}
\begin{proof}
	We extend $p_j$ to $U_j$ by setting  
	\[
		p_j'(\varphi_t x) := \mathbf{P}(x, t) p_j(x), \quad x \in \Sigma_j, \quad |t| \leq \eta_2.
	\] 
	Using that $d(\varphi_{t_1}x, \varphi_{t_2}y) \sim d(x, y) + |t_1 - t_2|$ for $\eta_2$ small enough, $|t_i| \leq \eta_2$ for $i = 1, 2$, and $x, y \in \Sigma_j$, and using also the Ambrose-Singer formula (similarly as in the proof of Lemma \ref{lemma:p_j-holder-1}) we obtain the H\"older estimate \eqref{eq:p_j-holder-2}.
	
	Next, using \eqref{eq:p_j-holder-2} and the fact that $C^\alpha$ is an algebra, we directly get the estimate \eqref{eq:p_j-holder-3}. Since $p_j$ agrees with $p_-$ on $\mc{T}_{\eps} \cap U_j$, we have that $p_j'$ agrees with $p_-$ on $\mc{T}_\eps$ by construction. Therefore $\widetilde{p}|_{\mc{T}_\eps} = p_-$. To see the uniform derivative we recall that $\X^{\Hom} p_j' = 0$ by definition and hence
	\[\X^{\Hom} \widetilde{p} = \sum_{j = 1}^N X\chi_j \cdot p_j',\]
	and so the second estimate of \eqref{eq:p_j-holder-3} follows from \eqref{eq:p_j-holder-2}. For the final estimate of \eqref{eq:p_j-holder-3}, observe that $\X^{\Hom} \widetilde{p} = 0$ on $\mc{T}_\eps$ since $\widetilde{p} = p_-$ there. The estimate follows from the facts that the H\"older norm of $\X^{\Hom} \widetilde{p}$ is uniformly bounded and that $\mc{T}_\eps$ is $\varepsilon^{\beta_d}$-dense.
	
	Finally, a straightforward consequence of \eqref{eq:p_j-holder-3} and $\varepsilon^{\beta_d}$-density of $\mc{T}_\eps$ is that, since $p_-$ is unitary on $\mc{T}_\eps$, we may assume $\|\widetilde{p}\|_{C^0(\M, \mathrm{Hom}(\E, \E_0))} \leq 2$ for $\varepsilon$ small enough, and also that we may assume $\widetilde{p}$ is invertible at every point of $\M$. This completes the proof.
	\end{proof}
	
\subsection{Unitarisation of the conjugacy}

Let $\widetilde{p} \in C^\alpha(\M, \Hom(\E, \E_0))$ be the conjugacy constructed in Lemma \ref{lemma:p_j-holder-2}; it is invertible, and satisfies
\begin{align}
\begin{split}\label{eq:p-tilde-corollary}
	&\|\widetilde{p}\|_{C^\alpha} \leq \mc{O}(1), \quad \|\widetilde{p}\|_{C^0} \leq 2, \quad \widetilde{p}^* \widetilde{p}|_{\mc{T}_\eps} = \id_{\E},\\ 
	&\|\X^{\Hom} \widetilde{p}\|_{C^\alpha} = \mc{O}(1), \quad \|\X^{\Hom} \widetilde{p}\|_{C^0} \leq \mc{O}(1) \varepsilon^{\beta_d \alpha}.
\end{split}
\end{align}
Let us split $\widetilde{p}$ into unitary and radial parts:
\[
	\widetilde{p} = \underbrace{\widetilde{p} (\widetilde{p}^* \widetilde{p})^{-1/2}}_{p := } \underbrace{(\widetilde{p}^* \widetilde{p})^{1/2}}_{R :=},
\]
where for each $x \in \M$, $p(x)$ and $R(x)$ belong to $\operatorname{U}(\E, \E_0)(x)$ and $\End(\E)(x)$, respectively. 

\begin{lemma}\label{lemma:unitarisation}
	There exists $C > 0$ depending only on $b_0$ such that
	\begin{equation}\label{eq:p-holder-final}
		\|p\|_{C^{\alpha/4}(\M, \Hom(\E, \E_0))} \leq C, \quad \|\X^{\Hom} p\|_{C^{\alpha/4}(\M, \Hom(\E, \E_0))} \leq C \varepsilon^{\frac{\beta_d \alpha}{4}}.
	\end{equation}
\end{lemma}
\begin{proof}
	Let us first study the radial part $R$ of $\widetilde{p}$. Since the square root is $1/2$-H\"older continuous on the space of positive self-adjoint matrices with a uniform H\"older constant (this follows by elementary arguments), by an argument involving a local orthonormal frame for $\E$ and the Ambrose-Singer formula we get that $R \in C^{\alpha/2}(\M, \End(\E))$ with the estimate
	%it's actually tricky! see https://math.stackexchange.com/questions/1934184/is-the-matrix-square-root-uniformly-continuous
	\begin{equation}\label{eq:r-holder}
		\|R\|_{C^{\alpha/2}} \leq \mc{O}(1) \|\widetilde{p}^* \widetilde{p}\|_{C^{\alpha}}^{\frac{1}{2}} \leq \mc{O}(1) \|\widetilde{p}\|_{C^\alpha} = \mc{O}(1),
	\end{equation}
	where in the second inequality we used the fact that $C^{\alpha}(\M, \End(\E))$ is an algebra (with continuous multiplication), and in the last estimate we used \eqref{eq:p-tilde-corollary}. Next, we observe that $R = \id_{\E}$ on $\mc{T}_\eps$ by \eqref{eq:p-tilde-corollary}. Hence, using that $\mc{T}_\eps$ is $\varepsilon^{\beta_d}$-dense and the uniform H\"older bound \eqref{eq:r-holder}, we obtain
	\[\|R - \id_{\E}\|_{C^0} \leq \mc{O}(1) \varepsilon^{\beta_d \alpha}.\]
	Interpolating between $C^0$ and $C^{\alpha/2}$, and using the two previous estimates we get
	\begin{equation}\label{eq:r-interpolation-bound}
		\|R - \id_{\E}\|_{C^{\alpha/4}} \leq \mc{O}(1) \eps^{\frac{\beta_d \alpha}{4}}.
	\end{equation}
	By the continuity of multiplication in $C^{\alpha/4}$, there is a constant $D > 0$ such that
	\begin{align}\label{eq:multiplication-algebra}
	\begin{split}
		&\|Q_1 Q_2 \|_{C^{\alpha/4}(\M, \End(\E))} \leq D \|Q_1\|_{C^{\alpha/4}(\M, \End(\E))} \|Q_2\|_{C^{\alpha/4}(\M, \End(\E))},\\ 
		&\hspace{250pt} \forall Q_1, Q_2 \in C^{\alpha/4}(\M, \End(\E)),
	\end{split}
	\end{align}
	and hence using Neumann series, i.e. that $R^{-1} = (\id_{\E} - (\id_{\E} - R))^{-1}$ we get that
	\begin{equation}\label{eq:r-1-holder}
		\|R^{-1}\|_{C^{\alpha/4}} \leq \sum_{i = 0}^\infty D^i \|\id_{\E} - R\|^i_{C^{\alpha/4}} \leq 2,
	\end{equation}
	for $\varepsilon > 0$ small enough, where in the last estimate we used \eqref{eq:r-interpolation-bound}. In particular from \eqref{eq:r-1-holder} it follows that
	\begin{equation*}
		\|p\|_{C^{\alpha/4}} = \|\widetilde{p}R^{-1}\|_{C^\alpha/4} \leq D \|\widetilde{p}\|_{C^{\alpha/4}} \|R^{-1}\|_{C^{\alpha/4}} = \mc{O}(1),
	\end{equation*}
	where in the first inequality we used \eqref{eq:multiplication-algebra}, and in the second one we used \eqref{eq:p-tilde-corollary}. This shows the first estimate in \eqref{eq:p-holder-final}.
	
	Next, we estimate $\X^{\Hom}p$, and to this end we may write (using the Leibniz rule)
	\begin{equation}\label{eq:X-hom-p}
		\X^{\Hom}p = (\X^{\Hom} \widetilde{p}) R^{-1} - \widetilde{p} R^{-1}(\X^{\End}r) R^{-1},
	\end{equation}
	where $\X^{\End}$ is the operator induced by $\X$ on the vector bundle $\End(\E) = \E \otimes \E^*$. Interpolating between the last two estimates of \eqref{eq:p-tilde-corollary} we get that
	\begin{equation}\label{eq:p-tilde-interpolation}
		\|\X \widetilde{p}\|_{C^{\alpha/4}} \leq \mc{O}(1) \varepsilon^{\frac{\beta_d \alpha}{4}}.
	\end{equation}
	Estimating the right hand side of \eqref{eq:X-hom-p} using \eqref{eq:r-1-holder}, \eqref{eq:p-tilde-corollary}, \eqref{eq:multiplication-algebra}, and \eqref{eq:p-tilde-interpolation}, to show the second estimate in \eqref{eq:p-holder-final}, it suffices to show
	\begin{equation}\label{eq:suffices}
		\|\X^{\End}R\|_{C^{\alpha/4}} \leq \mc{O}(1) \varepsilon^{\frac{\beta_d \alpha}{4}}.
	\end{equation}
	To see this, we derive the identity $R^2 = \widetilde{p}^*\widetilde{p}$ to get:
	\begin{align*}
		2\X^{\End} R &= (\X^{\Hom} \widetilde{p})\widetilde{p}^* + \widetilde{p} (\X^{\Hom} \widetilde{p})^* - (R - \id_{\E}) \X^{\End} R - \X^{\End} R ( R - \id_{\E})\\ 
		&= \mc{O}_{C^{\alpha/4}}(\varepsilon^{\frac{\beta_d \alpha}{4}}(1 + \|X^{\End}r\|_{C^{\alpha/4}})),
	\end{align*}
	where in the last estimate we used \eqref{eq:p-tilde-interpolation}, \eqref{eq:p-tilde-corollary}, and \eqref{eq:r-interpolation-bound}. The required estimate \eqref{eq:suffices} follows by taking $\varepsilon$ small enough and absorbing the second part of the right hand side to the left hand side. This completes the proof.
\end{proof}

We are in shape to prove Theorem \ref{theorem:approximate-livsic}.

\begin{proof}[Proof of Theorem \ref{theorem:approximate-livsic}]
	The result readily follows from Lemma \ref{lemma:from-trace-to-parry} (and \eqref{equation:equality-approx}) and Lemma \ref{lemma:unitarisation}, where the latter lemma also relies on definition \eqref{eq:p_-def} and Lemmas \ref{lemma:small-trunk}, \ref{lemma:uniform-bound-p_-}, \ref{lemma:p_j-holder-1}, and \ref{lemma:p_j-holder-2}. The fact that $p$ can be chosen to be smooth follows from a standard approximation result (see \cite[Lemma E.45]{Dyatlov-Zworski-19}, where the Sobolev space scale is replaced by the H\"older-Zygmund scale).
\end{proof}
%%%
%%%
%%%

\section{Stability estimates}\label{sec:main-proof}

We now prove the stability estimates stated in the introduction of the paper.

\subsection{Stability estimates for Hermitian line bundles}

We first deal with the case of line bundles. We consider an Anosov manifold $(M,g)$ and $(\mc{L}, h) \rightarrow M$, a Hermitian line bundle equipped with two unitary connections $\nabla^{\mc{L}}_{1}$ and $\nabla^{\mc{L}}_{2}$ such that $\|\nabla^{\mc{L}}_{1} - \nabla^{\mc{L}}_{2}\|_{C^N_*} \leq K$ for some $K > 0$.

\begin{proof}[Proof of Theorem \ref{theorem:stability-line-bundles-new}.]

Define $\theta := i^{-1}(\nabla^{\mc{L}}_1 - \nabla^{\mc{L}}_2) \in C^N_*(\M, \Omega^1)$, where $\theta$ is a real-valued $1$-form. Consider the unitary cocycle on $\M \times (\mathbb{C}, h_{\mathrm{std}})$ defined by
	\[
		C((x, v), t) := \exp \left( i\int_0^t \theta(\dot{\gamma}(t))\, \dd t  \right),\quad (x, v) \in \M,\quad t \in \mathbb{R}.
	\]
	where $\gamma$ is the unit speed geodesic determined by $(x, v)$. This cocycle is induced by the first order differential operator $X + i\pi_1^*\theta$, where $\pi_1^*: C^\infty(M, \Omega^1) \to C^\infty(\M)$ is defined by $\pi_1^*\eta(x, v) := \eta(x)(v)$. By definition, $X + i\pi_1^*\theta$ is uniformly bounded by $(K, 1)$. Let 
	\[\eps := \sup_{\gamma \in \mc{G}^\sharp} \frac{1}{\ell(\gamma)} |\mathbf{W}(\nabla_1^{\mc{L}}, \gamma) - \mathbf{W}(\nabla_1^{\mc{L}}, \gamma)| = \sup_{\gamma \in \mc{G}^\sharp} \frac{1}{\ell(\gamma)} |\mathbf{W}(d + i\theta, \gamma) - 1|.\] 
	 We apply Theorem \ref{theorem:approximate-livsic} to the triple $(\M \times \mathbb{C}, h_{\mathrm{std}}, X + i\pi_1^*\theta)$ to obtain some $u \in C^\infty(\M, \Ss^1)$ such that 	
	 \begin{equation}\label{eq:livshic}
		i \pi_1^*\theta = u^{-1} Xu + iR,
	\end{equation}
	where $R \in C^N_*(\M)$ is real-valued and $\|R\|_{C^\alpha} = \mathcal{O}(\varepsilon^\tau)$, for some $\alpha, \tau > 0$. Note that here we use the canonical unitary isomorphism of $\End(\mc{L})$ with $\mathbb{C}$ (using the identity section), and that the endomorphism connection induced by $\nabla_1^{\mc{L}}$ is the trivial one in this identification (as can be seen since the identity section is parallel).
	\medskip 
	
Define a smooth closed $1$-form on $\M$ by $\varphi := \frac{du}{i u}$. Since $\pi^*: H^1(M) \to H^1(\M)$ is an isomorphism on de Rham cohomology (as follows from the Gysin sequence and since $M$ is not the $2$-torus, see e.g. \cite[Corollaries 8.10 and 9.5]{Merry-Paternain-11}), there is a harmonic $1$-form $w$ and $f \in C^\infty(\M)$ such that $\varphi = \pi^*w + df$. Applying $\iota_X$ and using \eqref{eq:livshic} we obtain
	\begin{equation}\label{eq:transport}
		Xf = \pi_1^*(\theta - w) - R.
	\end{equation}
	In the case where $R \equiv 0$, we would get that $I_1(\theta - w) = 0$ is in the kernel of the $X$-ray transform, defined by
	\[
		I_1 \eta = \left(\frac{1}{\ell(\Gamma_1)} \int_{\Gamma_1} \eta,\,\, \frac{1}{\ell(\Gamma_2)} \int_{\Gamma_2} \eta,\dotsc \right), \quad \eta \in C^\infty(M, \Omega^1),
	\] 
	where $\Gamma_1, \Gamma_2, \dotsc$ is an enumeration of closed geodesics in $(M, g)$. From this it is possible to show that $\theta - w$ is exact, so $\theta$ is closed and $[\theta/2\pi] \in H^1(M,\mathbb{Z})$, which allows to conclude that $\theta = dg/(ig)$ for some $g : M \rightarrow \mathbb{S}^1$, i.e. $\nabla_1^{\mc{L}}$ and $\nabla_2^{\mc{L}}$ are gauge equivalent through $f$. 
	
	In the present case, from \eqref{eq:transport} we have $I_1(\theta - w) = \mathcal{O}_{\ell^{\infty}}(\varepsilon^\tau)$. By Hodge decomposition, write 
	\[
		\theta - w = da + b, \quad  b \in \ker \left(d^*|_{C^N_*(\M, \Omega^1)}\right), \quad a \in C^{N + 1}_*(\M),
	\]
	where $d^*$ is the adjoint of $d$. Moreover, one has $b = \pi_{\ker (d^*)}(\theta - w)$ and $da = \pi_{\ran(d)}(\theta - w)$, where $\pi_{\ker (d^*)}$ and $\pi_{\ran(d)}$ are pseudodifferential operators of order $0$ (see e.g. \cite[eq. (2.13)]{Cekic-Lefeuvre-22}), so we get
	\begin{align*}
		\|b\|_{C^\alpha} \leq \mc{O}(1)\|\theta - w\|_{C^\alpha} \leq \mc{O}(1) (1 + \|w\|_{C^\alpha})
	\end{align*}
	as well as 
	\begin{equation}\label{eq:ab-bounded}
				\|b\|_{C^N_*} \leq \mc{O}(1) (1 + \|w\|_{C^N_*}), \quad \|a\|_{C^{N+1}_*} \leq \mc{O}(1) \|da\|_{C^N_*} \leq \mc{O}(1) (1 + \|w\|_{C^N_*}),
	\end{equation}
	where in the second estimate we also used that $d$ is an elliptic operator (when acting on functions). We claim that $w = \mc{O}_{C^\infty}(1)$ is bounded, which would show 
	\begin{equation}\label{eq:b-uniform}
		\|b\|_{C^\alpha} = \mc{O}(1).
	\end{equation} 
	Since the space of harmonic $1$-forms $\mathscr{H}^1(M)$ is finite dimensional, the injectivity of $I_1$ on $\ker (d^*)$ for Anosov manifolds (see \cite{Dairbekov-Sharafutdinov-03}) implies for each $L > 0$ there is a $C = C(M, g, L) > 0$ such that
	\begin{equation}\label{eq:x-ray-stable}
		\|h\|_{C^L_*} \leq C \|I_1h\|_{\ell^\infty}, \quad h \in \mathscr{H}^1(M).
	\end{equation}
	Now \eqref{eq:transport} implies $I_1w = \mc{O}_{\ell^\infty}(1)$ (using $\|\theta\|_{C^N_*} = \mc{O}(1)$ and $\|R\|_{C^\alpha} = \mc{O}(\eps^\tau)$), which by \eqref{eq:x-ray-stable} implies $w = \mc{O}_{C^\infty}(1)$ and proves the claim.

	Moreover, we have $I_1(da) = 0$ so $I_1(b) = I_1(\theta - w) = \mathcal{O}_{\ell^\infty}(\varepsilon^\tau)$. Hence, using the stability estimate for $I_1$ proved in \cite[Theorem 1.5]{Gouezel-Lefeuvre-19}, there is a $p = p(\alpha) \in (0, 1)$ such that:
	\begin{equation}\label{eq:b-small}
		\|b\|_{C^{\alpha/2}} \leq \mc{O}(1) \|I_1(b)\|^{p}_{\ell^\infty} \|b\|^{1 - p}_{C^{\alpha}} = \mc{O}(\varepsilon^{p \tau}),
	\end{equation}
	where in the last estimate we used \eqref{eq:b-uniform}. This implies that 
	\begin{equation}\label{eq:derived-estimate-end-result}
		\theta - w = da + \mathcal{O}_{C^{\alpha/2}}(\varepsilon^{p \tau}).
	\end{equation}

Write $\eta := w + da$. Then we have, for every closed geodesic $\gamma \subset M$:
	\begin{equation}\label{eq:smallholonomy}
		\int_\gamma \eta = \mc{O}(\eps^{p\tau})\ell(\gamma) + \int_\gamma \theta = \mc{O}(\eps^{p\tau})\ell(\gamma) + \int_{(\gamma, \dot{\gamma})} \frac{Xu}{iu}(\gamma, \dot{\gamma}) \in 2\pi \mathbb{Z} + \mc{O}(\eps^{p\tau})\ell(\gamma),
	\end{equation}
	where in the first equality we used \eqref{eq:derived-estimate-end-result}, in the second we used \eqref{eq:livshic} as well as $R = \mc{O}_{C^\alpha}(\eps^\tau)$, and finally we used that the integral of the logarithmic derivative $u^{-1} Xu$ belongs to $2\pi i \mathbb{Z}$.
	
Fix a basis of geodesic loops $(\gamma_i)_{i = 1}^{b_1(M)} \subset H_1(M; \mathbb{R})$ and a basis of real-valued harmonic $1$-forms $(h_j)_{j = 1}^{b_1(M)} \subset \mathscr{H}^1(M)$ such that 
\[\int_{\gamma_i} h_j = \delta_{ij},\quad i, j = 1, \dotsc, b_1(M),\]
where $\delta_{ij}$ denotes the Kronecker delta. By \eqref{eq:smallholonomy}, we have 
\[\int_{\gamma_i} \eta = 2 k_i \pi + r_i,\quad k_i \in \Z, \quad r_i = \mc{O}(\eps^{p \tau}), \quad i = 1, \dotsc, b_1(M).\] 
Define $\eta_r := \sum_{j=1}^{b_1(M)} r_j h_j = \mc{O}_{C^\infty}(\eps^{p\tau})$ so that 
	\begin{equation}\label{eq:eta-minus-remainder}
		\int_{\gamma_i} (\eta - \eta_r) = 2 k_i \pi, \quad i = 1, \dotsc, b_1(M).
	\end{equation} 
Fix a basepoint $x_0 \in M$ and define the gauge
\[G(x) := \exp \left(i \int_{\rho_x} (\eta - \eta_r) \right) \in \mathbb{S}^1,\quad x \in M,\] 
where $\rho_x$ is any path starting at $x_0$ and joining $x$; that $G$ is well-defined independently of the choice of $\rho_x$ follows from \eqref{eq:eta-minus-remainder}. Interpolating between \eqref{eq:b-small} and \eqref{eq:ab-bounded}, we conclude that $\|b\|_{C^{N - \eta}_*} = \mc{O}(\varepsilon^{\widetilde{\tau}})$ for some $\widetilde{\tau} \in (0, p\tau)$. Therefore, using also $G^{-1}dG = i(\eta - \eta_r)$ and $\eta_r = \mc{O}_{C^\infty}(\varepsilon^{p \tau})$, we see that
	\begin{equation}\label{eq:gain}
		\nabla_2^{\mc{L}} - \nabla_1^{\mc{L}} = i \theta = i\eta + ib = G^{-1}dG + (i \eta_r + ib) = G^{-1}dG + \mc{O}_{C^{N - \eta}_*}(\eps^{\widetilde{\tau}}).
	\end{equation}
	We note that the $\tau > 0$ claimed to exist in the statement of the theorem is actually $\widetilde{\tau}$ (which is strictly smaller than the $\tau$ from the beginning of the proof), completing the proof.
\end{proof}

\subsection{Stability estimates for higher rank vector bundles}

We now prove Theorem \ref{theorem:main}. We fix $\nabla^{\mc{E}} \in \mathcal{A}_{\E}^N$, a unitary connection on $\E$ satisfying the assumptions: \textbf{(A)} the connection $\nabla^{\E}$ is \emph{opaque}, namely $\ker (\pi^*\nabla^{\End(\E)}_X) = \C \id_{\E}$, and \textbf{(B)} the generalized X-ray transform on twisted $1$-forms with values in $\End(\E)$ is solenoidally injective. Note that the condition \textbf{(A)} is equivalent to the irreducibility of Parry's representation (see \cite[Proposition 3.5]{Cekic-Lefeuvre-22}) induced by the unitary cocycle $(\pi^*\E, \pi^*h, (\pi^*\nabla^{\E})_X)$, where $\pi^*h$ is the  pullback Hermitian metric on $\pi^*\E$.

\begin{proof}[Proof of Theorem \ref{theorem:main}]
Let 
\[\varepsilon := \sup_{\gamma \in \mc{G}^\sharp} \ell(\gamma)^{-1} \left|\W(a_1, \gamma)-\W(a_2, \gamma)\right|\]
and note that in what follows $\eps$ can be taken arbitrarily small, i.e. of the size $\mc{O}(\delta^{\frac{1}{\tau}})$, at the cost of also taking small $\delta$; in other words, it suffices to prove the stability estimate for $\eps$ small enough.

We will apply the approximate non-Abelian Liv\v{s}ic Theorem \ref{theorem:approximate-livsic}. By definition, the assumption that we work in an $\delta$-neighbourhood of $a_0$ guarantees that we may pick representatives $\nabla^{\E} + A_i$ of $a_i$ such that $\|A_i\|_{C_*^N} < \delta$ for $i = 1 , 2$. By continuity (see \cite[Section 5]{Cekic-Lefeuvre-22a}), we may then assume that $\nabla^{\E} + A_i$ are opaque for $i = 1, 2$, if $\delta > 0$ is taken small enough (depending only on $\nabla^{\E}$). Moreover, by taking $N \geq 1$ we may assume that $(\pi^*\E, \pi^*h, \X + A_i(X))$ for $i = 1, 2$ are uniformly bounded by $(K, \rk(\E))$, for some $K > 0$ and for $\delta > 0$ small enough (depending only on $(\E, h, \nabla^{\E})$). Therefore we get $p_{A_1, A_2} \in C^\infty(\M, \mathrm{U}(\E))$ such that:
\begin{equation}\label{eq:corollary-approx-livsic}
	\|\X_{A_1, A_2} p_{A_1, A_2}\|_{C^\alpha} = \mc{O}(\eps^\tau), \quad \|p_{A_1, A_2}\|_{C^\alpha} = \mc{O}(1),
\end{equation}
where $\X_{A_1, A_2} = \X^{\Hom}$ is the first order differential operator on $\Hom(\E, \E) = \End(\E)$ induced by $\X_1 := (\pi^*\nabla^{\E})_X + A_1(X)$ and $\X_2 := (\pi^*\nabla^{\E})_X + A_2(X)$. More precisely, we have
\[\X_{A_1, A_2}(\bullet) = \X^{\End}(\bullet) + A_2(X) \bullet - \bullet A_1(X),\]
where we write $\X^{\End} := (\pi^*\nabla^{\End})_X$. Here $A_i(X) (x, v) := \pi^*A_i(X)(x, v) = A_i(x)(v)$, for $i = 1, 2$.

We fix the exponent $s > 0$ of the anisotropic Sobolev space $\mc{H}^s_\pm$ such that $\alpha > 2s$. Let $\gamma$ be a small oriented contour around zero, such that it encloses no other resonances of $\X^{\End}$ except the zero resonance. By Lemmas \ref{lemma:perturbed-resolvent} and \ref{lemma:projector-v}, the spectral projector 
\[\Pi_{A_1, A_2}^+ := \frac{1}{2\pi i} \oint_\gamma (z + \X_{A_1, A_2})^{-1}\, dz\]
is well-defined for $\delta > 0$ small enough, and depends smoothly in $\mc{L}(\mc{H}^s_+, \mc{H}^s_+)$ on $(A_1, A_2)$ in the $(C^\alpha)^2$ norm. Moreover, it is the spectral projector to the unique resonance $\lambda_{A_1, A_2} \leq 0$ of $\X_{A_1, A_2}$ enclosed by $\gamma$ (here $\lambda_{A_1, A_2} \leq 0$ follows from \cite[Lemma 4.2]{Cekic-Lefeuvre-22}). 

For $N \gg 1$ we can consider the map 
\[\left[C_*^N(M,T^*M \otimes \mathrm{End}_{\mathrm{sk}}(\E))\right]^2 \ni (A_1, A_2) \mapsto p_{A_1, A_2} \in C^\alpha(\M, \operatorname{U}(\E)).\] 
It is not clear that this map is continuous with respect to $(A_1, A_2)$, as the constructions in Section \ref{sec:approximate-livsic} are quite indirect. Nevertheless, the following claim holds:

\begin{lemma}
\label{lemma:control}
There exist $C > 1$ and $\delta > 0$ (depending only on $(\E, h, \nabla^{\E})$) such that if $\|A\|_{C^N} < \delta$, then:
\begin{align}
\label{equation:control-p}
\frac{1}{C} &\leq \|p_{A_1, A_2}\|_{\mc{H}_+^s} \leq C, \\
\label{equation:control-pi-p}
\frac{1}{C} &\leq \|\Pi_{A_1, A_2}^+ p_{A_1, A_2}\|_{\mc{H}_+^s} \leq C.
\end{align}
\end{lemma}

Before proving the lemma, let us show how it implies the main result. We know that:
\[
	 \mc{O}_{\mc{H}_+}(\eps^{\tau}) = -\Pi_{A_1, A_2}^+ \X_{A_1, A_2} p_{A_1, A_2} = -\X_{A_1, A_2} \Pi_{A_1, A_2}^+ p_{A_1, A_2} = \lambda_{A_1, A_2} \Pi_{A_1, A_2}^+ p_{A_1, A_2},
\]
where in the first equality we used \eqref{eq:corollary-approx-livsic}, as well as the fact that $\|\Pi_{A_1, A_2}^+\|_{\mc{H}^s_+ \to \mc{H}^s_+} = \mc{O}(1)$ is uniformly bounded, and in the third equality we used the fact that $\X_{A_1, A_2}$ and $\Pi_{A_1, A_2}^+$ commute. Therefore, using the lower bound \eqref{equation:control-pi-p}, we obtain that
\begin{equation}\label{eq:small-resonance}
	|\lambda_{A_1, A_2}| = \mc{O}(\eps^{\tau}).
\end{equation}

We let $\phi(A_1, A_2)$ be the unique connection (in the vicinity of $\nabla^{\E}$ in $C^N_*$) such that $\phi(A_1, A_2) - (\nabla^{\E} + A_1) \in \ker (\nabla^{\E} + A_1)^*$, i.e. $\phi(A_1, A_2)$ is in the Coulomb gauge with respect to $\nabla^{\E} + A_1$ (here the superscript denotes the adjoint with respect to natural metrics). That $\phi(A_1, A_2)$ depends smoothly on $(A_1, A_2)$ follows from Lemma \cite[Lemma 4.1]{Cekic-Lefeuvre-22}. Then \cite[Lemma 4.5]{Cekic-Lefeuvre-22} applies (here we use the assumption {\bf (B)}) and gives:
\begin{equation}\label{eq:H-1/2-estimate}
	\|\phi(A_1, A_2) - (\nabla^{\E} + A_1)\|^2_{H^{-1/2}} \leq \mc{O}(1) |\lambda_{A_1, A_2}| = \mc{O}(\eps^{\tau}),
\end{equation}
where we used \eqref{eq:small-resonance} in the last equality. By assumption, we know that $\|\phi(A_1, A_2) - (\nabla^{\E} + A_1)\|_{C^N_*} = \mc{O}(1)$ and thus by interpolation with the estimate \eqref{eq:H-1/2-estimate}, we deduce that $\|\phi(A_1, A_2) - (\nabla^{\E} + A_1)\|_{C^{N - \eta}_*} = \mc{O}(\eps^{\widetilde{\tau}})$, for some other exponent $\widetilde{\tau} \in (0, \tau/2)$ depending on $\eta$. This completes the proof, granted Lemma \ref{lemma:control}.

\begin{proof}[Proof of Lemma \ref{lemma:control}]
	The upper bound of \eqref{equation:control-p} follows directly from \eqref{eq:corollary-approx-livsic} and the embedding $C^\alpha \xhookrightarrow{} \mc{H}^s_+$. Let us prove the lower bound. Define $h_{A_1, A_2} := p_{A_1, A_2}^{-1} \X_{A_1, A_2} p_{A_1, A_2}$. Then $h$ is skew-Hermitian since we may write $h_{A_1, A_2} = p_{A_1, A_2}^*\X_2 - \X_1$ (by using the Leibniz rule) and using the facts that $\X_1$ and $\X_2$ are unitary extensions, and that $p_{A_1, A_2}$ is also unitary. By \eqref{eq:corollary-approx-livsic} it follows that
\begin{equation}\label{eq:corollary-approx-livsic-h}
	\|h_{A_1, A_2}\|_{C^\alpha} = \mc{O}(\varepsilon^{\tau}).
\end{equation}
Hence $p_{A_1, A_2}$ satisfies the equation:
\[
	\widetilde{\X}_{A_1, A_2} p_{A_1, A_2} = 0, \quad \widetilde{\X}_{A_1, A_2}(\bullet) := \X^{\End}(\bullet) + A_2(X) \bullet - \bullet (A_1(X) + h_{A_1, A_2}).
\]
Thus $p_{A_1, A_2}$ is a resonant state at the zero resonance for the perturbed operator $\widetilde{\X}_{A_1, A_2}$ acting on $C^\infty(\M, \pi^*\End(\E))$; the triple $(\End(\E), \End(h), \widetilde{\X}_{A_1, A_2})$ is a linear unitary extension (here $\End(h)$ is the natural inner product on $\End(\E)$ induced by $h$). Using \eqref{eq:corollary-approx-livsic-h} and Lemma \ref{lemma:projector-v}, we conclude that for $\varepsilon > 0$ and $\delta > 0$ small enough the spectral projectors
\[
	\widetilde{\Pi}^{\pm}_{A_1, A_2} = \frac{1}{2\pi i} \oint_{\gamma} (z \pm \widetilde{\X}_{A_1, A_2})^{-1}\, dz
\]
are well-defined, and project onto the one dimensional resonant space $\mathbb{C} p_{A_1, A_2}$ of the unique resonance at $z = 0$ of $\pm \widetilde{\X}_{A_1, A_2}$ enclosed by $\gamma$. Since $(\widetilde{\X}_{A_1, A_2})^* = -\widetilde{\X}_{A_1, A_2}$, we have the relation
\begin{equation*}%\label{eq:projectors-symmetry}
	\left(\widetilde{\Pi}_{A_1, A_2}^+\right)^* = \widetilde{\Pi}_{A_1, A_2}^+ = \widetilde{\Pi}_{A_1, A_2}^-.
\end{equation*}
Let us write $u_{A_1, A_2} := \widetilde{\Pi}_{A_1, A_2}^+ \id_{\E}$. By Taylor expansion, using Lemma \ref{lemma:projector-v} we have
\begin{align}\label{eq:taylor-expansion}
\begin{split}
	&u_{A_1, A_2} := \widetilde{\Pi}_{A_1, A_2}^+ \id_{\E} = \id_{\E} + \mc{O}_{\mc{H}^s_\pm}(\|A_1\|_{C^\alpha} + \|A_2\|_{C^\alpha} + \|h_{A_1, A_2}\|_{C^\alpha})\\ 
	&\hspace{250pt} = \id_{\E} + \mc{O}_{\mc{H}^s_\pm}(\delta) + \mc{O}_{\mc{H}^s_\pm}(\eps^\tau),
\end{split}
\end{align}
where we used \eqref{eq:corollary-approx-livsic-h} in the last equality. Note here that we are not expanding in $(A_1, A_2)$, but in the background variables 
\[\left[C^\alpha(\M, \End(\E))\right]^2 \ni (S_1, S_2) \mapsto \frac{1}{2\pi i} \oint_\gamma (z + \X(S_1, S_2))^{-1}\, dz \in \mc{L}(\mc{H}^s_\pm, \mc{H}^s_\pm),\]
where $\X(S_1, S_2)(\bullet) := \X^{\End}(\bullet) + S_2 \bullet - \bullet S_1$. Write $f_{A_1, A_2} := u_{A_1, A_2} - \id_{\E}$ and $p_{A_1, A_2} = c_{A_1, A_2} u_{A_1, A_2}$, where $c_{A_1, A_2} \in \mathbb{C}$. Since $p_{A_1, A_2}$ is unitary, we have
\begin{equation}\label{eq:p-unitary-identity}
	\id_{\E} = p_{A_1, A_2}^* p_{A_1, A_2} = |c_{A_1, A_2}|^2 (\id_{\E} + \underbrace{f_{A_1, A_2} + f_{A_1, A_2}^* + f_{A_1, A_2}^* f_{A_1, A_2}}_{r :=}).
\end{equation}
Using \eqref{eq:taylor-expansion} and Lemma \ref{lemma:multiplication2}, for $L := \frac{n}{2} + \alpha$ we get that 
\begin{equation}\label{eq:remainder-small}
	\|r\|_{H^{-L}} = \mc{O}(\delta) + \mc{O}(\varepsilon^\tau) + (\mc{O}(\delta) + \mc{O}(\varepsilon^\tau))^2 \leq \frac{1}{2}\|\id_{\E}\|_{H^{-L}},
\end{equation}
for $\delta > 0$ and $\varepsilon > 0$ small enough (depending only on $(\E, h, \nabla^{\E})$). Here we used the fact that $\mc{H}^s_\pm$ can be chosen such that they are invariant under taking adjoints, see \S \ref{ssec:P-R-resonances}. By \eqref{eq:remainder-small} and \eqref{eq:p-unitary-identity}, we get that
\begin{equation*}%\label{eq:c-bounded}
	\frac{1}{2} \leq \frac{\|\id_{\E}\|_{H^{-L}}}{\|\id_{\E}\|_{H^{-L}} + \|r\|_{H^{-L}}} \leq |c_{A_1, A_2}|^2 \leq \frac{\|\id_{\E}\|_{H^{-L}}}{\|\id_{\E}\|_{H^{-L}} - \|r\|_{H^{-L}}} \leq 2.
\end{equation*}
Therefore, we obtain
\begin{equation}\label{eq:p-lower-bound}
	\|p_{A_1, A_2}\|_{\mc{H}^s_+} = |c_{A_1, A_2}| \|u_{A_1, A_2}\|_{\mc{H}^s_+} \geq \frac{1}{\sqrt{2}} (\|\id_{\E}\|_{\mc{H}^s_+} - \mc{O}(\delta) - \mc{O}(\varepsilon^\tau)) \geq \frac{1}{2} \|\id_{\E}\|_{\mc{H}^s_+},
\end{equation}
for $\varepsilon > 0$ and $\delta > 0$ small enough, where in the first inequality we used \eqref{eq:taylor-expansion}. This completes the proof of \eqref{equation:control-p}.

To show the bound \eqref{equation:control-pi-p}, we observe first that we may write using the resolvent identity
\[\Pi_{A_1, A_2}^+ - \widetilde{\Pi}_{A_1, A_2}^+ = -\frac{1}{2\pi i} \oint_\gamma (\X_{A_1, A_2} + z)^{-1} M_{h_{A_1, A_2}} (\widetilde{\X}_{A_1, A_2} + z)^{-1}\, dz,\]
where $M_{h_{A_1, A_2}}$ denotes multiplication on the right by $h_{A_1, A_2}$. By Lemmas \ref{lemma:perturbed-resolvent} and \ref{lemma:multiplication} we get (note that the resolvents make sense as above for $\varepsilon > 0$ and $\delta > 0$ small enough)
\begin{equation}\label{eq:projectors-control}
	\|\Pi_{A_1, A_2}^+ - \widetilde{\Pi}_{A_1, A_2}^+\|_{\mc{H}^s_+ \to \mc{H}^s_+} \leq \mc{O}(1) \|h_{A_1, A_2}\|_{C^\alpha} = \mc{O}(\eps^\tau),
\end{equation}
where in the last estimate we used \eqref{eq:corollary-approx-livsic-h}. We observe the identity
\[\Pi^+_{A_1, A_2} p_{A_1, A_2} = (\Pi^+_{A_1, A_2} - \widetilde{\Pi}_{A_1, A_2}^+) p_{A_1, A_2} + p_{A_1, A_2}.\]
We therefore obtain the bound
\begin{align*}
	&\|\Pi^+_{A_1, A_2} p_{A_1, A_2}\|_{\mc{H}^s_+} \geq \|p_{A_1, A_2}\|_{\mc{H}^s_+} - \|(\Pi^+_{A_1, A_2} - \widetilde{\Pi}_{A_1, A_2}^+) p_{A_1, A_2}\|_{\mc{H}^s_+}\\ 
	&\hspace{230pt}\geq \frac{1}{2}\|\id_{\E}\|_{\mc{H}^s_+} - \mc{O}(\varepsilon^\tau) \geq \frac{1}{4}\|\id_{\E}\|_{\mc{H}^s_+},
\end{align*}
for $\varepsilon > 0$ small enough, where in the second inequality we used \eqref{eq:p-lower-bound}, \eqref{equation:control-p}, and \eqref{eq:projectors-control}. We similarly get
\[\|\Pi^+_{A_1, A_2} p_{A_1, A_2}\|_{\mc{H}^s_+} \leq \|p_{A_1, A_2}\|_{\mc{H}^s_+} + \|(\Pi^+_{A_1, A_2} - \widetilde{\Pi}_{A_1, A_2}^+) p_{A_1, A_2}\|_{\mc{H}^s_+} = \mc{O}(1).\]
By taking $\varepsilon > 0$ small enough, the last two estimates complete the proof of the lemma.
\end{proof}
This completes the proof of the theorem.
\end{proof}

\subsection{Global stability estimates}

Finally, we prove Corollary \ref{corollary}.

\begin{proof}[Proof of Corollary \ref{corollary}]
	Let $\tau > 0$ be the exponent coming from Theorem \ref{theorem:main}. We argue by contradiction: assume there are sequences $(\mathfrak{a}_{1, i})_{i = 1}^\infty, (\mathfrak{a}_{2, i})_{i = 1}^\infty \subset \mathbb{A}_{\E}^N \setminus \mc{U}$ (which is a further subset of $\mathbb{A}_{\E}^{N - \eta} \setminus \mc{U}$), with $d_{C_{*}^N}(\mathfrak{a}_{j, i}, \mathfrak{a}_0) \leq K$ for all $i \in \mathbb{Z}_{\geq 1}$ and $j = 1, 2$, such that
	\begin{equation}\label{eq:contradiction}
		d_{C^{N - \eta}_*}(\mathfrak{a}_{1, i}, \mathfrak{a}_{2, i}) \geq i \left( \sup_{\gamma \in \mc{G}^\sharp} \ell(\gamma)^{-1} \left|\W(\mathfrak{a}_{1, i}, \gamma) - \W(\mathfrak{a}_{2, i}, \gamma)\right| \right)^{\tau}, \quad i \in \mathbb{Z}_{\geq 1}.
	\end{equation}
	Since the inclusion $C^N_* \subset C^{N - \eta}_*$ is compact, by extracting a subsequence we may assume $\mathfrak{a}_{j, i} \to \mathfrak{a}_j$ in $\mathbb{A}_{\E}^{N - \eta}$, for $j = 1, 2$, as $i \to \infty$. By \eqref{eq:contradiction} and the boundedness of its left hand side (by assumption), and continuity of the Wilson loop operator, we get $\mathbf{W}(\mathfrak{a}_1, \gamma) = \mathbf{W}(\mathfrak{a}_2, \gamma)$ for all $\gamma \in \mc{G}^\sharp$. By \cite[Theorem 4.5]{Cekic-Lefeuvre-22b}, this implies $\mathfrak{a}_1 = \mathfrak{a}_2 =: \mathfrak{a}$ (note that in \cite{Cekic-Lefeuvre-22b} we worked with smooth connections, but the proof carries over to the case where $N \gg 1$). Note that $\mathfrak{a} \not \in \mc{U}$ by continuity, so the stability estimates of Theorem \ref{theorem:main} are valid near $\mathfrak{a}$ in $C_*^N$. This contradicts \eqref{eq:contradiction} and completes the proof.
\end{proof}

\bibliographystyle{alpha}
\bibliography{Biblio}

\begin{thebibliography}{GPSU16}

\bibitem[Ano67]{Anosov-67}
D.~V. Anosov.
\newblock Geodesic flows on closed {R}iemannian manifolds of negative
  curvature.
\newblock {\em Trudy Mat. Inst. Steklov.}, 90:209, 1967.

\bibitem[Bal05]{Baladi-05}
Viviane Baladi.
\newblock Anisotropic {S}obolev spaces and dynamical transfer operators:
  {$C^\infty$} foliations.
\newblock In {\em Algebraic and topological dynamics}, volume 385 of {\em
  Contemp. Math.}, pages 123--135. Amer. Math. Soc., Providence, RI, 2005.

\bibitem[BCD11]{Bahouri-Chemin-Danchin-11}
Hajer Bahouri, Jean-Yves Chemin, and Rapha\"{e}l Danchin.
\newblock {\em Fourier analysis and nonlinear partial differential equations},
  volume 343 of {\em Grundlehren der Mathematischen Wissenschaften [Fundamental
  Principles of Mathematical Sciences]}.
\newblock Springer, Heidelberg, 2011.

\bibitem[BKL02]{Blank-Keller-Liverani-02}
Michael Blank, Gerhard Keller, and Carlangelo Liverani.
\newblock Ruelle-{P}erron-{F}robenius spectrum for {A}nosov maps.
\newblock {\em Nonlinearity}, 15(6):1905--1973, 2002.

\bibitem[Boh21]{Bohr-21}
Jan Bohr.
\newblock Stability of the non-abelian {$X$}-ray transform in dimension
  {$\geq3$}.
\newblock {\em J. Geom. Anal.}, 31(11):11226--11269, 2021.

\bibitem[BT07]{Baladi-Tsujii-07}
Viviane Baladi and Masato Tsujii.
\newblock Anisotropic {H}\"{o}lder and {S}obolev spaces for hyperbolic
  diffeomorphisms.
\newblock {\em Ann. Inst. Fourier (Grenoble)}, 57(1):127--154, 2007.

\bibitem[CEG23]{Chen-Erchenko-Gogolev-23}
Dong Chen, Alena Erchenko, and Andrey Gogolev.
\newblock Riemannian {Anosov} extension and applications.
\newblock {\em J. {\'E}c. Polytech., Math.}, 10:945--987, 2023.

\bibitem[CL21]{Cekic-Lefeuvre-21-2}
Mihajlo {Ceki{\'c}} and Thibault {Lefeuvre}.
\newblock {Generic injectivity of the X-ray transform}.
\newblock {\em To appear in Journal of Differential Geometry}, July 2021.

\bibitem[CL22a]{Cekic-Lefeuvre-22a}
Mihajlo Ceki{\'c} and Thibault Lefeuvre.
\newblock Generic dynamical properties of connections on vector bundles.
\newblock {\em Int. Math. Res. Not.}, 2022(14):10649--10703, 2022.

\bibitem[CL22b]{Cekic-Lefeuvre-22b}
Mihajlo {Ceki{\'c}} and Thibault {Lefeuvre}.
\newblock {Isospectral connections, ergodicity of frame flows, and polynomial
  maps between spheres}.
\newblock {\em To appear in Annales scientifiques de l'Ecole Normale
  Supérieure}, September 2022.

\bibitem[CL22c]{Cekic-Lefeuvre-22}
Mihajlo {Ceki{\'c}} and Thibault {Lefeuvre}.
\newblock {The Holonomy Inverse Problem}.
\newblock {\em To appear in Journal of European Mathematical Society}, 2022.

\bibitem[CP20]{Cekic-Paternain-20}
Mihajlo Ceki\'{c} and Gabriel~P. Paternain.
\newblock Resonant spaces for volume-preserving {A}nosov flows.
\newblock {\em Pure Appl. Anal.}, 2(4):795--840, 2020.

\bibitem[DS03]{Dairbekov-Sharafutdinov-03}
Nurlan~S. Dairbekov and Vladimir~A. Sharafutdinov.
\newblock Some problems of integral geometry on {A}nosov manifolds.
\newblock {\em Ergodic Theory Dynam. Systems}, 23(1):59--74, 2003.

\bibitem[DZ16]{Dyatlov-Zworski-16}
Semyon Dyatlov and Maciej Zworski.
\newblock Dynamical zeta functions for {A}nosov flows via microlocal analysis.
\newblock {\em Ann. Sci. \'{E}c. Norm. Sup\'{e}r. (4)}, 49(3):543--577, 2016.

\bibitem[DZ19]{Dyatlov-Zworski-19}
Semyon Dyatlov and Maciej Zworski.
\newblock {\em Mathematical theory of scattering resonances}, volume 200 of
  {\em Graduate Studies in Mathematics}.
\newblock American Mathematical Society, Providence, RI, 2019.

\bibitem[FH19]{Fisher-Hasselblatt-19}
Todd {Fisher} and Boris {Hasselblatt}.
\newblock {\em {Hyperbolic flows}}.
\newblock Berlin: European Mathematical Society (EMS), 2019.

\bibitem[Fin03]{Finch-03}
David~V. Finch.
\newblock The attenuated x-ray transform: recent developments.
\newblock In {\em Inside out: inverse problems and applications}, volume~47 of
  {\em Math. Sci. Res. Inst. Publ.}, pages 47--66. Cambridge Univ. Press,
  Cambridge, 2003.

\bibitem[FS11]{Faure-Sjostrand-11}
Fr\'{e}d\'{e}ric Faure and Johannes Sj\"{o}strand.
\newblock Upper bound on the density of {R}uelle resonances for {A}nosov flows.
\newblock {\em Comm. Math. Phys.}, 308(2):325--364, 2011.

\bibitem[GdP22]{Guillarmou-de-Poyferre-22}
Colin Guillarmou and Thibault de~Poyferr\'{e}.
\newblock A paradifferential approach for hyperbolic dynamical systems and
  applications.
\newblock {\em Tunis. J. Math.}, 4(4):673--718, 2022.
\newblock Appendix by Yannick Guedes Bonthonneau.

\bibitem[GL06]{Gouezel-Liverani-06}
S\'{e}bastien Gou\"{e}zel and Carlangelo Liverani.
\newblock Banach spaces adapted to {A}nosov systems.
\newblock {\em Ergodic Theory Dynam. Systems}, 26(1):189--217, 2006.

\bibitem[GL19]{Gouezel-Lefeuvre-19}
S{\'e}bastien {Gou{\"e}zel} and Thibault {Lefeuvre}.
\newblock {Classical and microlocal analysis of the X-ray transform on Anosov
  manifolds}.
\newblock {\em arXiv e-prints}, Apr 2019.

\bibitem[GL20]{Bonthonneau-Lefeuvre-20}
Yannick {Guedes Bonthonneau} and Thibault {Lefeuvre}.
\newblock {Radial estimates in Besov spaces, applications to some geometric
  problems}.
\newblock {\em arXiv e-prints}, page arXiv:2011.06403, November 2020.

\bibitem[GPSU16]{Guillarmou-Paternain-Salo-Uhlmann-16}
Colin Guillarmou, Gabriel~P. Paternain, Mikko Salo, and Gunther Uhlmann.
\newblock The {X}-ray transform for connections in negative curvature.
\newblock {\em Comm. Math. Phys.}, 343(1):83--127, 2016.

\bibitem[H{\"o}r97]{Hormander-97}
Lars H{\"o}rmander.
\newblock {\em Lectures on nonlinear hyperbolic differential equations},
  volume~26 of {\em Math\'{e}matiques \& Applications (Berlin) [Mathematics \&
  Applications]}.
\newblock Springer-Verlag, Berlin, 1997.

\bibitem[Kni02]{Knieper-02}
Gerhard Knieper.
\newblock Hyperbolic dynamics and {Riemannian} geometry.
\newblock In {\em Handbook of dynamical systems. Volume 1A}, pages 453--545.
  Amsterdam: North-Holland, 2002.

\bibitem[Kuc14]{Kuchment-14}
Peter Kuchment.
\newblock {\em The {R}adon transform and medical imaging}, volume~85 of {\em
  CBMS-NSF Regional Conference Series in Applied Mathematics}.
\newblock Society for Industrial and Applied Mathematics (SIAM), Philadelphia,
  PA, 2014.

\bibitem[Lan02]{Lang-02}
Serge Lang.
\newblock {\em Algebra}, volume 211 of {\em Graduate Texts in Mathematics}.
\newblock Springer-Verlag, New York, third edition, 2002.

\bibitem[Liv04]{Liverani-04}
Carlangelo Liverani.
\newblock On contact {A}nosov flows.
\newblock {\em Ann. of Math. (2)}, 159(3):1275--1312, 2004.

\bibitem[MP11]{Merry-Paternain-11}
Will Merry and Gabriel~P. Paternain.
\newblock {\em Inverse Problems in Geometry and Dynamics}.
\newblock Lecture notes, 2011.
\newblock Available online:
  \url{https://www.dpmms.cam.ac.uk/~gpp24/ipgd(3).pdf}.

\bibitem[Pat09]{Paternain-09}
Gabriel~P. Paternain.
\newblock Transparent connections over negatively curved surfaces.
\newblock {\em J. Mod. Dyn.}, 3(2):311--333, 2009.

\bibitem[Pat11]{Paternain-11}
Gabriel~P. Paternain.
\newblock B\"{a}cklund transformations for transparent connections.
\newblock {\em J. Reine Angew. Math.}, 658:27--37, 2011.

\bibitem[Pat12]{Paternain-12}
Gabriel~P. Paternain.
\newblock Transparent pairs.
\newblock {\em J. Geom. Anal.}, 22(4):1211--1235, 2012.

\bibitem[Pat13]{Paternain-13}
Gabriel~P. Paternain.
\newblock Inverse problems for connections.
\newblock In {\em Inverse problems and applications: inside out. {II}},
  volume~60 of {\em Math. Sci. Res. Inst. Publ.}, pages 369--409. Cambridge
  Univ. Press, Cambridge, 2013.

\bibitem[PSU12]{Paternain-Salo-Uhlmann-12}
Gabriel~P. Paternain, Mikko Salo, and Gunther Uhlmann.
\newblock The attenuated ray transform for connections and {Higgs} fields.
\newblock {\em Geom. Funct. Anal.}, 22(5):1460--1489, 2012.

\bibitem[PSUZ19]{Paternain-Salo-Uhlmann-Zhou-19}
Gabriel~P. Paternain, Mikko Salo, Gunther Uhlmann, and Hanming Zhou.
\newblock The geodesic {X}-ray transform with matrix weights.
\newblock {\em Amer. J. Math.}, 141(6):1707--1750, 2019.

\bibitem[{\v{S}}em06]{Semrl-06}
Peter {\v{S}}emrl.
\newblock Maps on matrix spaces.
\newblock {\em Linear Algebra Appl.}, 413(2-3):364--393, 2006.

\bibitem[Sha07]{Sharafutdinov-07}
Vladimir Sharafutdinov.
\newblock Variations of {D}irichlet-to-{N}eumann map and deformation boundary
  rigidity of simple 2-manifolds.
\newblock {\em J. Geom. Anal.}, 17(1):147--187, 2007.

\bibitem[Tay91]{Taylor-91}
Michael~E. Taylor.
\newblock {\em Pseudodifferential operators and nonlinear {PDE}}, volume 100 of
  {\em Progress in Mathematics}.
\newblock Birkh\"{a}user Boston, Inc., Boston, MA, 1991.

\bibitem[UV16]{Uhlmann-Vasy-16}
Gunther Uhlmann and Andr\'{a}s Vasy.
\newblock The inverse problem for the local geodesic ray transform.
\newblock {\em Invent. Math.}, 205(1):83--120, 2016.

\bibitem[Wil74]{Wilson-74}
Kenneth~G. Wilson.
\newblock Confinement of quarks.
\newblock {\em Phys. Rev. D}, 10:2445--2459, Oct 1974.

\bibitem[Woo68]{Wood-68}
Reginald Wood.
\newblock Polynomial maps from spheres to spheres.
\newblock {\em Invent. Math.}, 5:163--168, 1968.

\bibitem[Zho17]{Zhou-17}
Hanming Zhou.
\newblock Generic injectivity and stability of inverse problems for
  connections.
\newblock {\em Comm. Partial Differential Equations}, 42(5):780--801, 2017.

\bibitem[Zwo12]{Zworski-12}
Maciej Zworski.
\newblock {\em Semiclassical analysis}, volume 138 of {\em Graduate Studies in
  Mathematics}.
\newblock American Mathematical Society, Providence, RI, 2012.

\end{thebibliography}
\end{document}